\documentclass{article}
\usepackage[utf8]{inputenc}
\usepackage{amsfonts,amsthm,hyperref,enumitem,amsmath}
\usepackage{graphicx, caption}
\usepackage{amssymb}
\usepackage{amsthm}
\usepackage{amsmath}
\usepackage{enumitem}
\usepackage{framed}
\usepackage{thm-restate}
\usepackage{tikz}
\usetikzlibrary{calc,snakes}
\usepackage{soul}
\usepackage{changepage}
\usepackage{comment}
\usepackage{float}
\usepackage{hyperref}
\usepackage{lscape}

\usepackage{mathtools}

\newtheorem{lemma}{Lemma}[section]
\newtheorem{theorem}[lemma]{Theorem}
\newtheorem{corollary}[lemma]{Corollary}

\newtheorem{definition}[lemma]{Definition}
\newtheorem{observation}[lemma]{Observation}
\newtheorem{claim}[lemma]{Claim}

\newtheorem{problem}[lemma]{Problem}
\newtheorem{conjecture}[lemma]{Conjecture}

\newcommand\claimproofend{\renewcommand{\qedsymbol}{$\boxdot$}
\end{proof}
\renewcommand{\qedsymbol}{$\square$}}

\newcommand{\E}{\mathbb{E}}

\title{Approximate path decompositions of regular graphs}

\makeatletter
\newcommand{\labelinthm}[1]{%
   \label{temp#1}
   \protected@write \@auxout {}{\string \newlabel{#1}{{\emph{\ref{temp#1}}}{\thepage}{\emph{\ref{temp#1}}}{temp#1}{}} }
}
\makeatother
\newcounter{propcounter}

\newcommand{\eps}{\varepsilon}
\newcommand{\labelno}[1]{}
\newcommand{\N}{\mathbb{N}}
\newcommand\polysmall{\ll}

\author{Richard Montgomery\thanks{Mathematics Institute, University of Warwick, Coventry, UK. Research supported by the European Research Council (ERC) under the European Union Horizon 2020 research and innovation programme (grant agreement No.\ 947978). 
\\ Email: {\tt
richard.montgomery@warwick.ac.uk}.} 
\and 
Alp M\"uyesser\thanks{Department of Mathematics, University College London, Gower Street, London, UK. \\ {Emails}: {\tt alp.muyesser.21@ucl.ac.uk} and {\tt dralexeypokrovskiy@gmail.com.}} \and Alexey Pokrovskiy\footnotemark[2] \and Benny Sudakov\thanks{
Department of Mathematics, ETH, Z\"urich, Switzerland. Research supported in part by SNSF grant 200021-228014.
\newline
{Email}: {\tt benjamin.sudakov@math.ethz.ch}.
}}

\date{}

\begin{document}

\maketitle

\begin{abstract}
 We show that the edges of any $d$-regular graph can be almost decomposed into paths of length roughly $d$, giving an approximate solution to a problem of Kotzig from 1957. Along the way, we show that almost all of the vertices of a $d$-regular graph can be partitioned into $n/(d+1)$ paths, asymptotically confirming a conjecture of Magnant and Martin from 2009.
\end{abstract}


\section{Introduction}\label{sec:intro}


A typical decomposition problem asks if the edges of some graph can be partitioned into copies of another graph. The origins of this area can be traced back to Euler who asked the following question in 1782: for which $n$ does the balanced complete $4$-partite graph $K_{n,n,n,n}$ decompose into copies of the complete graph on $4$ vertices (i.e.\ copies of $K_4$)? Euler's problem is customarily phrased using the language of Latin squares and transversals, but the two formulations are equivalent. In 1847, Kirkman studied decompositions of a complete graph into triangles (copies of $K_3$). Such decompositions are also referred to as \textit{Steiner triple systems}. In 1882, Walecki studied decompositions of complete graphs into Hamilton paths and cycles.
All of this work, and other early approaches to graph decomposition problems, used exclusively algebraic or constructive techniques. In contrast, there have been many applications of the probabilistic method \cite{alon2016probabilistic} to graph decomposition problems in the last few decades, and several major problems have recently been resolved. Highlights include, but are certainly not limited to, Keevash's resolution \cite{keevash2014existence} of the Existence Conjecture for designs, which is a far reaching generalisation of the work of Kirkman (see also \cite{glock2023existence}), and the proof for large $n$ of Ringel's conjecture \cite{montgomery2021proof, keevash2014existence} which states that $K_{2n+1}$ decomposes into copies of any $n$-edge tree. We refer the reader to \cite{bottcher2022graph, glock2021extremal} for two recent surveys of the area.

The common denominator in these recent advances is that they concern decompositions of highly dense structures, such as complete graphs. 
There is also a large body of conjectures on the decomposition of sparse graphs, where there has been much less progress as existing proof techniques like those involving Szemer\'edi's regularity lemma only apply to dense graphs.  Perhaps the most famous conjecture about decomposing sparse graphs is the linear arboricity conjecture of Akiyama, Exoo, and Harary~\cite{akiyama1980covering} from 1980, which says that every graph with maximum degree $\Delta$ can be decomposed into $\lceil(\Delta+1)/2\rceil$ path forests. The best general bound on this is by Lang and Postle~\cite{lang2023improved} who proved that such a graph always has a decomposition into $\Delta/2+3\sqrt{\Delta}\log^4\Delta$ path forests. 
For decompositions into paths rather than path forests, there is a conjecture of Gallai which predicts that every connected $n$-vertex graph can be decomposed into at most $(n+1)/2$ paths (see~\cite{lovasz1968covering}). The best general result for this is that every such graph can be decomposed into at most $2n/3$ paths, proved in~\cite{dean2000gallai, yan1998path}. What if we require all the paths to be of the same length $d$? Here some extra conditions are required on the host graph $G$, since at the very least we need $d$ to divide $e(G)$. 
Botler, Mota, Oshiro, and Wakabayashi~\cite{botler2017decomposingb} showed that such a $P_d$-decomposition exists if, furthermore, the graph is $f(d)$-edge-connected, for some function $f$, a result motivated by a now-solved conjecture of Bar\'at and Thomassen (see~\cite{bensmail2017proof}). 
A natural way to ensure $d$ divides $e(G)$ when $d$ is odd is to require $G$ to be $d$-regular. In 1957, Kotzig~\cite{kotzig1957theorie} proved that a $3$-regular graph has a $P_3$-decomposition if and only if $G$ has a perfect matching, and raised the following general question for larger $d$.

\begin{problem}\label{prob:kotzig}
When $d$ is odd, which $d$-regular graphs can be decomposed into paths of length $d$?
\end{problem}
\par  In 1990, Bondy \cite{bondy1990perfect} proposed a further extension of this problem without the regularity assumption. There are very few general results known about these problems. Some progress has been made on a related problem of Favaron, Genest, and Koudier \cite{favaron2010regular}, who conjectured that a $d$-regular graph $G$ (for odd $d$) decomposes into copies of $P_d$ if $G$ contains a perfect matching. See \cite{botler2015decompositions,botler2017decomposing,botler2022decomposition} for some partial progress on this conjecture concerning the cases when $d$ is small, or when the host graph $G$ is assumed to have high girth.

As obtaining a precise decomposition appears to be so difficult, it is of independent interest to show that the existing conjectures hold true approximately, in part as it may eventually be used as a stepping stone towards a full resolution.  
For example, even the approximate version of the Existence Conjecture for designs was a longstanding open problem of Erd\H{o}s and Hanani \cite{erdos1963limit} before it was resolved through the influential work of Rödl \cite{rodl1985packing}. This approximate solution plays a key component in the work of Keevash \cite{keevash2014existence} resolving the full conjecture.

\par In this paper, we address a relaxed version of Kotzig's problem by showing that almost all of the edges of a $d$-regular graph can be decomposed into paths of length roughly $d$.

\begin{theorem}\label{Theorem_main} For every $\eps>0$, there exists $d_0\in\mathbb{N}$ such that the following holds for all $d\geq d_0$. If $G$ is a $d$-regular $n$-vertex graph, then all but at most $\eps nd$ edges of $G$ can be decomposed into paths of length $\lceil (1-\eps)d \rceil$.
\end{theorem}

Constructing large edge-disjoint path forests  plays a large part in our methods. However, note that Theorem~\ref{Theorem_main} cannot be strengthened to require that the paths in the decomposition can be arranged into $\lceil (d+1)/2\rceil$ path forests (as in the linear arboricity conjecture), as seen by considering the vertex-disjoint union of $d$-regular graphs of order $3d/2$. Our methods, though, can be used to make progress on a problem related to the linear arboricity conjecture, as follows. Note that, when $d$ is odd, if the linear arboricity conjecture is true for $d$-regular graphs then one of the path forests in such a decomposition would have at most $\frac{nd}{2}/\frac{d+1}{2}=n-\frac{n}{d+1}$ edges, and therefore (adding isolated vertices if necessary) form a spanning path forest with at most $\frac n {d+1}$ paths. Though naturally much weaker than the linear arboricity conjecture, showing even the existence of such a path forest in a $d$-regular graph appears very difficult. This (for odd or even $d$) is the topic of the following conjecture by Magnant and Martin \cite{magnant2009note} from 2009 
 (which, as per \cite{feige2014short}, also has an interesting connection to tour length problems).

\begin{conjecture}\label{conj:MM}
The vertices of every $d$-regular $n$-vertex graph can be partitioned into $\lfloor n/(d+1)\rfloor$ paths.
\end{conjecture}
The disjoint union of complete graphs on $d+1$ vertices shows that this conjecture would be optimal. It was confirmed for all $d\leq 5$ by Magnant and Martin \cite{magnant2009note}, and for $d=6$ by Feige and Fuchs \cite{feige2022path}. It has also recently been confirmed in the dense case ($d=\Omega(n)$) by Gruslys and Letzter \cite{gruslys2021cycle}, but in general even the weaker conjecture by Feige and Fuchs \cite{feige2022path} that every $d$-regular $n$-vertex graph can be partitioned into $O(n/(d+1))$ paths remains wide open. We prove the following approximate version of Conjecture~\ref{conj:MM}. 

\begin{theorem}\label{thm:pathcovers}
For every $\eps>0$, there exists $d_0\in\mathbb{N}$ such that the following holds for all $d\geq d_0$. Let $G$ be a $d$-regular $n$-vertex graph. Then, all but at most $\eps n$ of the vertices of $G$ can be partitioned into at most $n/(d+1)$ paths.
\end{theorem}

If we consider decomposing approximately the edges of a $d$-regular graph $G$ into copies of $P_\ell$ starting with small $\ell$, then the extremal examples for why we cannot increase $\ell$ beyond $d$ all contain small dense subgraphs (e.g.\ the cliques in the disjoint union of copies of $K_{d+1}$), which we refer to as `dense spots'. In $d$-regular graphs without these dense spots it would be relatively easy (using techniques developed for the linear arboricity conjecture) to decompose almost all of the edges of $G$ into paths with length $\lceil(1-\eps)d\rceil$ by using the local lemma to partition $V(G)$ into sets between which we can find many edge-disjoint matchings which combine to form many paths, and then iteratively connect some of these paths together using a set of reserved vertices to get longer and longer paths (as we do later in parts of our proof). The main idea in our proof is to set aside the dense spots in a $d$-regular graph and then show that if the preceding approach goes wrong then we will be able to connect paths which we cannot lengthen into dense spots, so that we can then decompose approximately the dense spots along with their attached paths into paths with length $\lceil(1-\eps)d\rceil$. To do so, we will need the set of reserved vertices to `sample' each dense spot appropriately in a precise manner, and how we manage this is the key technical novelty in our methods. However, there are many further challenges in developing these initial ideas into a proof of Theorem~\ref{Theorem_main}, and even approximately decomposing the dense spots with connected paths requires new ideas. Thus, in the next section, after outlining our notation, we give a detailed sketch of our methods. Following this, we outline the rest of the paper. After our proofs, we make some concluding remarks in Section~\ref{sec:concluding}, including on generalisations of path decomposition conjectures to other trees.


\section{Preliminaries}\label{sec:prelim}
Following an overview of our notation in Section~\ref{sec:notation}, in Section~\ref{sec:proofsketch} we give a sketch of the proof before outlining the rest of the paper. For ease of notation, we will assume throughout that $d$ is even, for example constructing paths forests $\mathcal{P}_i$, $i\in [d/2]$. The case where $d$ is odd follows almost identically as the proofs apply to $(d-1)$-regular graphs without any significant alteration.

\subsection{Notation}\label{sec:notation}
For a graph $G$, let $V(G)$ and $E(G)$ denote the vertex set and edge set of $G$, respectively, and let $|G|=|V(G)|$. For a vertex $v\in V(G)$, let $N_G(v)$ denote the neighbourhood of $v$ in $G$ and, for a subset $U\subset V(G)$, let $N_G(v,U)=N_G(v)\cap U$. We let $d_G(v)=|N_G(v)|$ denote the degree of $v$ and write $d_G(v,U)=|N_G(v,U)|$ for the degree of $v$ into a subset $U\subset V(G)$. As we do elsewhere, we omit $G$ from the subscript whenever there is no risk of confusion. The minimum and maximum degree of $G$ is denoted by $\delta(G)$ and $\Delta(G)$ respectively. For a subset $U\subseteq V(G)$, let $G[U]$ denote the graph induced by $U$ and, given a subset $U'\subset V(G)\setminus U$, let $G[U,U']$ denote the bipartite graph with bipartition $U\cup U'$ and edges of the form $uu'\in E(G)$ with $u\in U$ and $u'\in U'$.
Given a subset $U\subset V(G)$, we let $G-U$ be the graph $G[V(G)\setminus U]$, which is $G$ with the vertices in $U$ removed, and use similar natural and common other notation. When $\mathcal{F}$ is a collection of graphs, $V(\mathcal{F})$ and $E(\mathcal{F})$ will be the set of vertices and edges used in some graph in $\mathcal{F}$ respectively.

For each positive integer $n$, we let $[n]=\{1,\ldots, n\}$. Given real numbers $a,b,c$, we write $a=b\pm c$ to denote that $b-c\le a\le b+c$.  We say a graph $G$ is $(1\pm \gamma)d$-regular if $d_G(v)=(1\pm \gamma)d$ for each $v\in V(G)$. We use standard notation for ``hierarchies'' of constants, writing $x\ll y$ to mean that there is a non-decreasing function $f : (0,1] \rightarrow (0, 1]$ such that all the relevant subsequent statements hold for $x\leq f(y)$. Hierarchies with multiple constants are defined similarly, with the functions chosen from right to left, for example, for $x\ll y\ll z$.  We omit rounding signs where they are not crucial.

In this paper, we will often add paths to connect paths in a path forest. Given two path forests $\mathcal{P}$ and $\mathcal{Q}$ we use $\mathcal{P}+\mathcal{Q}$ to denote the graph with vertex set $V(\mathcal{P})\cup V(\mathcal{Q})$ and edge set $E(\mathcal{P})\cup E(\mathcal{Q})$. Furthermore, we exclusively use this when the resulting graph is also a path forest, and each path in $\mathcal{Q}$ has at least one endpoint among the endpoints of $\mathcal{P}$.


\subsection{Proof sketch}\label{sec:proofsketch}
Let $G$ be an $n$-vertex $d$-regular graph and let $1/d\ll \eps$ and $\ell\leq (1-\eps)d$. Suppose we wish to find edge-disjoint paths of length $\ell$ in $G$ covering all but at most $\eps n d$ edges (noting that if we can do this with $\ell=(1-\eps)d$ then we will have proved Theorem~\ref{Theorem_main}).
The following natural strategy, developed in part for attacking the linear arboricity conjecture (see, for example, Section~1.1 in \cite{ferber2020towards} which builds on \cite{alon2001linear, kelmans2001asymptotically}), can accomplish this if $\ell\leq d^{1/5-o(1)}$. Take a random partition of $V(G)$ into roughly equal sets $A_1\cup \ldots \cup A_{\ell+1}$. Using the local lemma (see Section~\ref{sec:local}), show that, with positive probability, each bipartite subgraph $G[A_i,A_j]$ with $i\neq j$ is almost $(d/(\ell+1))$-regular.
Taking such a partition, then, decompose each $G[A_i,A_j]$ into $k_0$ matchings $M_{ij,k}$, $k\in [k_0]$, where $k_0$ is only slightly more than $d/(\ell+1)$ (using, say, Vizing's theorem).
Decompose the complete graph $K_{\ell+1}$ into paths with length $\ell$ (it is well-known that this is possible whenever $\ell+1$ is even), and then, for each path and each $k$, take the corresponding matchings $M_{ij,k}$ for the edges $ij$ of the path, and combine them to find many paths of length $\ell$ in $G$ (along with some shorter paths where the matchings do not align exactly). Carried out carefully, this approach will give paths of length $\ell$ covering all but at most $\eps nd$ edges, as long as $\ell\leq d^{1/5-o(1)}$ so that not too many edges are lost from the potentially misaligned matchings. The constant $1/5$ here is the natural barrier for our implementation of this method later, but more generally such an approach encounters a strong natural barrier for $\ell\approx d^{1/2}$ (see \cite{lang2023improved}).

Of course, we wish to have an approach that works for $\ell$ up to $(1-\eps)d$. A tempting route forward is to first take aside a small `sample' set $X$ of $\approx pn$ vertices (likely chosen via the local lemma), with $p\ll \eps$, where $G-X$ is close to regular, and then decompose almost all of $G-X$ into paths with length $d^{1/5-o(1)}$. Then, we could use vertices from $X$ to iteratively join these paths into longer paths. 
The above partitioning and matching argument naturally produces edge-disjoint path forests $\mathcal{P}_1,\ldots,\mathcal{P}_{d/2}$ covering most of the edges of $G-X$ using paths with length $d^{1/5-o(1)}$. Given any set of $(1+\eps)n/d$ paths in one of these path forests, if we take one endpoint from each path to form the set $Y$, then, as the neighbours of the vertices in $Y$ must overlap, we could hope that some of the overlap will be sampled into $X$ so that we can join up two of the paths using a single vertex in $X$, and, perhaps, even to do this iteratively until the path forest contains at most $(1+\eps)n/d$ paths.

By making sure that the initial path forests do not together use any vertex as an endpoint too much, and that no path forest counts too many neighbours of a vertex among its endpoints, and by further dividing $X$ into several subsets to exhaust in turn while connecting paths, this can be made to work (see Sections~\ref{sec:initialpathforest} and \ref{sec:betterpathforest}), and, furthermore, even, used to prove Theorem~\ref{thm:pathcovers} (see Section~\ref{sec:pathcovers}). However, this approach will approximately decompose the $n$-vertex $d$-regular graph $G$ into paths with \emph{average} length $(1-\eps)d$ (see Lemma~\ref{Lemma_decomposition_path_forests}), rather than paths of length $(1-\eps)d$, so really is only the starting point for our proof of Theorem~\ref{Theorem_main}.

To go further, we consider what might stop us from joining up more paths in our path forests via $X$. After all, if we could continue to connect paths together until our paths forests mostly had paths with length at least $Kd$, for some $1/K\ll \eps$, then each such path can be further decomposed into paths with length $(1-\eps)d$ and at most $d-1$ other edges, where these other edges will then in total be a small portion of the edges of the graph $G$ so that we will not need to decompose them. The disjoint union of copies of $K_{d+1}$ demonstrates that hoping to always get paths of this length, $Kd$, is fanciful. However, the presence of small dense subgraphs (which we shall call `dense spots') turns out to be the only thing that can stop this from working. Indeed, suppose we had $r:=n/Kd$ vertex-disjoint paths $P_1,\ldots, P_r$ in $G-X$, and could not find a short path (say of length at most $k$, with $1/d \ll 1/k\ll 1/K$) in $G$ between two endpoints of different paths, so that the short path has all its interior vertices in $X$. Then, if, for each $i\in [r]$, $Y_i$ was the set of vertices in $X$ which can be reached by a path in $G$ of length at most $k/2$ from an endpoint of $P_i$ while using otherwise only vertices from $X$, then the sets $Y_i$, $i\in [r]$, must all be disjoint. Then, one of these sets $Y_i$ will be small (with $\leq Kd$ vertices), which will imply that $G[Y_i]$ must contain some dense subgraph. It is reasonable (if challenging), then, to hope that, starting with some initial paths forests, we could join the endpoints of these paths together via short connecting paths using new vertices in $X$ until most of the paths are either \textbf{a)} long (with length $\geq Kd$) or \textbf{b)} connected into a `dense spot' in $G[X]$.

In the case \textbf{b)}, we cannot hope to decompose these paths along with the dense spot in $G[X]$ they are connected to, for, as $|X|\approx pn$ and $X$ will be chosen randomly, such a dense spot (say a sample of a copy of $K_{d+1}$) may only have around $pd$ vertices, and thus only have paths of length up to $pd$. Therefore, we will have to argue that we can take the sample $X$ accurately enough that any dense spot in $G[X]$ must lie within some dense spot in $G$, so that we can mostly decompose the original dense spot along with any paths we have attached to it into paths with length $(1-\eps)d$. Reordering, slightly, then, we will do the following (see Figure~\ref{fig:proofsketch}):
\begin{enumerate}[label = \roman*)]
    \item Take a maximal collection $\mathcal{F}=\{G_1,\ldots,G_t\}$ of dense spots in $G$ and a small `sample set' $X$.\label{sketch:first}
    \item Mostly decompose $G-V(\mathcal{F})-X$ into $d/2$ path forests consisting of medium-length paths.\label{sketch:outside}
    \item  Iteratively join paths within each path forest together using short paths with internal vertices in $X$.
    \item Having exhausted the possible connections, we argue that, except for some small number of paths, most of the paths must be connectable into a dense spot in $G[X]$ and therefore, somehow, connectable into one of the dense spots $G_i$, $i\in [t]$.\label{sketch:connect} 
    \item Mostly decompose the long paths into paths of length $\ell=(1-\eps)d$ and, similarly, mostly decompose the dense spots along with the paths connected into them.\label{sketch:inside}\label{sketch:last}
\end{enumerate}

\begin{figure}
\begin{center}
\begin{tikzpicture}

\def\wi{8}
\def\hgt{6}
\def\horiz{(0.14285*\wi,0)}
\def\verti{(0,0.2*\hgt)}
\def\minihoriz{0.16666667*\horiz}
\def\vx{0.035cm}

\draw ($(0,0.97*\hgt)-0.14285*0.4*(\wi,0)$) node {\footnotesize $V(G)$};

\draw [black!50] ($0.2*(0,\hgt)$) to ++($6*0.14285*(\wi,0)$) to ++($0.8*(0,\hgt)$) to ++($-6*0.14285*(\wi,0)$) to ++($-0.8*(0,\hgt)$);
\draw [black!75,densely dotted] ($0.14285*(\wi,0)+(0,0.2*\hgt)$) to ++($(0,0.8*\hgt)$);
\draw [black!75,densely dotted] ($2*0.14285*(\wi,0)+(0,0.2*\hgt)$) to ++($(0,0.8*\hgt)$);
\draw [black!75,densely dotted] ($5*0.14285*(\wi,0)+(0,0.6*\hgt)$) to ++($(0,0.4*\hgt)$);
\draw [black!50] ($5*0.14285*(\wi,0)+(0,0.2*\hgt)$) to ++($(0,0.4*\hgt)$);
\draw [black!75,densely dotted] ($0.4*(0,\hgt)$) to ++($2*0.14285*(\wi,0)$);
\draw [black!50] ($0.6*(0,\hgt)$) to ++($6*0.14285*(\wi,0)$);
\draw [black!75,densely dotted] ($0.8*(0,\hgt)$) to ++($6*0.14285*(\wi,0)$);
\draw [black!75,densely dotted] ($3*0.14285*(\wi,0)+0.6*(0,\hgt)$) to ++($0.4*(0,\hgt)$);
\draw [black!75,densely dotted] ($4*0.14285*(\wi,0)+0.6*(0,\hgt)$) to ++($0.4*(0,\hgt)$);

\foreach \xx/\yy in {1/3,1/4}
{
\foreach \zz in {1,...,5}
{
\coordinate (A\xx\yy\zz) at ($\xx*(0.14285*\wi,0)-1*(0.14285*\wi,0)+\yy*(0,0.2*\hgt)-0.5*(0,0.2*\hgt)+\zz*(0.166666667*0.14285*\wi,0)$);
\draw [fill] (A\xx\yy\zz)  circle [radius=\vx];
}
}

\coordinate (G1) at ($(5.5*(0.14285*\wi,0)+2.75*(0,0.2*\hgt)-(0.15*0.14285*\wi,0)$);
\draw (G1) circle[x radius=0.08*\wi,y radius=0.03*\hgt];
\draw (G1) node {\footnotesize $G_1$};
\coordinate (G2) at ($(5.5*(0.14285*\wi,0)+2.75*(0,0.2*\hgt)-(0.15*0.14285*\wi,0)-2.8*(0,0.03*\hgt)$);
\draw (G2) circle[x radius=0.08*\wi,y radius=0.03*\hgt];
\draw (G2) node {\footnotesize $G_2$};
\coordinate (G3) at ($(5.5*(0.14285*\wi,0)+2.75*(0,0.2*\hgt)-(0.15*0.14285*\wi,0)-5.6*(0,0.03*\hgt)$);
\draw (G3) circle[x radius=0.08*\wi,y radius=0.03*\hgt];
\draw (G3) node {\footnotesize $G_3$};
\coordinate (Gt) at ($(5.5*(0.14285*\wi,0)+2.75*(0,0.2*\hgt)-(0.15*0.14285*\wi,0)-10.2*(0,0.03*\hgt)$);
\draw (Gt) circle[x radius=0.08*\wi,y radius=0.03*\hgt];
\draw (Gt) node {\footnotesize $G_t$};
\draw ($0.4*(Gt)+0.6*(G3)$) node {\footnotesize \vdots};

\coordinate (G1vx1) at ($(G1)-0.0625*(\wi,0)+0.008*(0,\hgt)$);
\draw [fill] (G1vx1)  circle [radius=\vx];
\coordinate (G1vx2) at ($(G1)-0.0625*(\wi,0)-0.008*(0,\hgt)$);
\draw [fill] (G1vx2)  circle [radius=\vx];
\coordinate (G2vx) at ($(G2)-0.0625*(\wi,0)$);
\draw [fill] (G2vx)  circle [radius=\vx];
\coordinate (Gtvx) at ($(Gt)-0.0625*(\wi,0)$);
\draw [fill] (Gtvx)  circle [radius=\vx];

\foreach \xx/\yy in {2/2,2/5,3/4,3/5,4/4,4/5,5/5,5/4}
{
\foreach \zz in {1,...,5}
{
\coordinate (A\xx\yy\zz) at ($\xx*(0.14285*\wi,0)-0.5*(0.14285*\wi,0)+\yy*(0,0.2*\hgt)-0*(0,0.2*\hgt)-\zz*(0,0.166666667*0.2*\hgt)$);
}
}

\foreach \xx/\yy in {1/5,2/4,6/4}
{
\foreach \zz in {1,...,5}
{
\coordinate (A\xx\yy\zz) at ($\xx*(0.14285*\wi,0)-1*(0.14285*\wi,0)+\yy*(0,0.2*\hgt)-0*(0,0.2*\hgt)+\zz*0.166666667*(0.14285*\wi,-0.2*\hgt)$);
}
}

\foreach \xx/\yy in {1/2,6/5,2/3}
{
\foreach \zz in {1,...,5}
{
\coordinate (A\xx\yy\zz) at ($\xx*(0.14285*\wi,0)-1*(0.14285*\wi,0)+\yy*(0,0.2*\hgt)-1*(0,0.2*\hgt)+\zz*0.166666667*(0.14285*\wi,0.2*\hgt)$);
}
}

\coordinate (A235prime) at ($(A235)+0.05*(\wi,0)$);
\coordinate (A234prime) at ($(A234)+0.05*(\wi,0)-0.02*(0,\hgt)+0.166666667*(0.14285*\wi,0)$);
\coordinate (A232prime) at ($(A232)+0.075*(\wi,0)+3*0.166666667*(0.14285*\wi,0)$);
\coordinate (A231prime) at ($(A231)+0.05*(\wi,0)+4*0.166666667*(0.14285*\wi,0)$);

\coordinate (A222prime) at ($(A222)+0.05*(\wi,0)+2*0.166666667*(0.14285*\wi,0)$);
\coordinate (A223prime) at ($(A223)+0.075*(\wi,0)+2*0.166666667*(0.14285*\wi,0)$);

\coordinate (A345prime) at ($(A345)-0.166666667*0.2*2*(0,\hgt)$);
\coordinate (A343prime) at ($(A345prime)+3*0.166666667*(0.14285*\wi,0)$);

\draw [red] (A235) -- (A235prime);
\draw [red] (A234) -- (A234prime);
\draw [red] (A232) -- (A232prime);
\draw [red] (A231) -- (A231prime);
\draw [red] (A222) -- (A222prime);
\draw [red] (A223) -- (A223prime);
\draw [red] (A345) -- (A345prime);
\draw [red] (A343) -- (A343prime);

\draw[snake=coil,segment aspect=0, segment amplitude=.4mm,segment length=2mm,orange] (A225) -- (Gtvx);
\draw[snake=coil,segment aspect=0, segment amplitude=.4mm,segment length=2mm,red] (A223prime) -- (A232prime);
\draw[snake=coil,segment aspect=0, segment amplitude=.4mm,segment length=2mm,red] (A222prime) -- (A231prime);
\draw[snake=coil,segment aspect=0, segment amplitude=.4mm,segment length=2mm,red] (A235prime) -- (A234prime);
\draw[snake=coil,segment aspect=0, segment amplitude=.4mm,segment length=2mm,red] (A343prime) -- (A345prime);
\draw[snake=coil,segment aspect=0, segment amplitude=.4mm,segment length=2mm,orange] (A545) -- (G1vx1);
\draw[snake=coil,segment aspect=0, segment amplitude=.4mm,segment length=2mm,orange] (A443) -- (G1vx2);
\draw[snake=coil,segment aspect=0, segment amplitude=.4mm,segment length=2mm,orange] (A445) -- (G2vx);

\foreach \xx/\yy/\xxx/\yyy in {1/2/1/3,1/4/1/5,1/5/2/5,2/5/3/5,3/5/4/5,2/3/2/4,4/5/5/5,6/4/6/5,5/4/6/4}
{
\foreach \zz in {1,...,5}
{
\draw (A\xx\yy\zz) -- (A\xxx\yyy\zz);
}
}
\foreach \xx/\yy/\xxx/\yyy in {1/3/1/4}
{
\foreach \zz in {1,...,5}
{
\draw [blue,thick] (A\xx\yy\zz) -- (A\xxx\yyy\zz);
}
}

\foreach \xx/\yy/\xxx/\yyy in {5/5/6/5,1/2/2/2}
{
\foreach \zz/\zzz in {1/5,2/4,3/3,4/2,5/1}
{
\draw (A\xx\yy\zz) -- (A\xxx\yyy\zzz);
}
}

\foreach \xx/\yy/\xxx/\yyy in {4/4/5/4,2/4/3/4}
{
\foreach \zz in {1,...,4}
{
\draw (A\xx\yy\zz) -- (A\xxx\yyy\zz);
}
}
\foreach \xx/\yy/\xxx/\yyy in {3/4/4/4}
{
\foreach \zz in {1,2,4,5}
{
\draw (A\xx\yy\zz) -- (A\xxx\yyy\zz);
}
}

\draw [red,thick] ($0.399*(-0.02,\hgt)$) to ++($0.14285*(\wi,0)$) to ++($0.2*(0,\hgt)$) to ++($-0.14285*(\wi,0)$) to ++($-0.2*(0,\hgt)$);
\draw [red,thick] ($0.601*(0,\hgt)$) to ++($0.14285*(\wi,0)$) to ++($0.2*(0,\hgt)$) to ++($-0.14285*(\wi,0)$) to ++($-0.2*(0,\hgt)$);

\draw [orange,thick] ($0.199*(0,\hgt)+0.14385*(\wi,0)$) to ++($0.14285*(\wi,0)$) to ++($0.2*(0,\hgt)$) to ++($-0.14285*(\wi,0)$) to ++($-0.2*(0,\hgt)$);
\draw [orange,thick] ($0.401*(0,\hgt)+0.14385*(\wi,0)$) to ++($0.14285*(\wi,0)$) to ++($0.2*(0,\hgt)$) to ++($-0.14285*(\wi,0)$) to ++($-0.2*(0,\hgt)$);

\draw [red] ($2.5*0.2*(0,\hgt)-(0.35,0)$) node  {\footnotesize $A_j$};
\draw [red] ($3.5*0.2*(0,\hgt)-(0.35,0)$) node  {\footnotesize $A_i$};
\draw [blue] ($3*0.2*(0,\hgt)-(0.45,0)$) node  {\footnotesize $M_{ij,k}$};
\draw ($(A225)-0.1666667*0.5*(0,\hgt)$) node  {\footnotesize \textcolor{orange}{$A_i'$} and \textcolor{orange}{$A_j'$}};

\draw  ($3.5*0.14285*(\wi,0)+0.4*(0,\hgt)$) node  {\footnotesize $X$};

\draw [fill] (A235prime)  circle [radius=\vx];
\draw [fill] (A234prime)  circle [radius=\vx];
\draw [fill] (A232prime)  circle [radius=\vx];
\draw [fill] (A231prime)  circle [radius=\vx];
\draw [fill] (A222prime)  circle [radius=\vx];
\draw [fill] (A223prime)  circle [radius=\vx];
\draw [fill] (A345prime)  circle [radius=\vx];
\draw [fill] (A343prime)  circle [radius=\vx];
\foreach \xx/\yy in {2/2,2/5,3/4,3/5,4/4,4/5,5/5,5/4}
{
\foreach \zz in {1,...,5}
{
\draw [fill] (A\xx\yy\zz)  circle [radius=\vx];
}
}

\foreach \xx/\yy in {1/5,2/4,6/4}
{
\foreach \zz in {1,...,5}
{
\draw [fill] (A\xx\yy\zz)  circle [radius=\vx];
}
}

\foreach \xx/\yy in {1/2,6/5,2/3}
{
\foreach \zz in {1,...,5}
{
\draw [fill] (A\xx\yy\zz)  circle [radius=\vx];
}
}

\end{tikzpicture}\hspace{1cm}\begin{tikzpicture}

\def\wi{4.4}
\def\hgt{3.3}%
\def\horiz{(0.14285*\wi,0)}
\def\verti{(0,0.2*\hgt)}
\def\minihoriz{0.16666667*\horiz}
\def\vx{0.035cm}

\draw [white] (0,-0.175*6) to ++(0,0.8*6);

\draw ($(0,0.97*\hgt)-0.14285*0.4*(\wi,0)-(0.45,0.05)$) node {\footnotesize $V(K_{\ell+1})$};

\draw [black!50] ($0.2*(0,\hgt)$) to ++($2*0.14285*(\wi,0)$) to ++($0.4*(0,\hgt)$) to ++($4*0.14285*(\wi,0)$) to ++($0.4*(0,\hgt)$) to ++($-6*0.14285*(\wi,0)$) to ++($-0.8*(0,\hgt)$);
\draw [black!75,densely dotted] ($0.14285*(\wi,0)+(0,0.2*\hgt)$) to ++($(0,0.8*\hgt)$);
\draw [black!75,densely dotted] ($2*0.14285*(\wi,0)+(0,0.2*\hgt)$) to ++($(0,0.8*\hgt)$);
\draw [black!75,densely dotted] ($5*0.14285*(\wi,0)+(0,0.6*\hgt)$) to ++($(0,0.4*\hgt)$);
\draw [black!75,densely dotted] ($0.4*(0,\hgt)$) to ++($2*0.14285*(\wi,0)$);
\draw [black!75,densely dotted] ($0.6*(0,\hgt)$) to ++($6*0.14285*(\wi,0)$);
\draw [black!75,densely dotted] ($0.8*(0,\hgt)$) to ++($6*0.14285*(\wi,0)$);
\draw [black!75,densely dotted] ($3*0.14285*(\wi,0)+0.6*(0,\hgt)$) to ++($0.4*(0,\hgt)$);
\draw [black!75,densely dotted] ($4*0.14285*(\wi,0)+0.6*(0,\hgt)$) to ++($0.4*(0,\hgt)$);

\foreach \xx/\yy in {1/3,1/4}
{
\foreach \zz in {3}
{
\coordinate (A\xx\yy\zz) at ($\xx*(0.14285*\wi,0)-1*(0.14285*\wi,0)+\yy*(0,0.2*\hgt)-0.5*(0,0.2*\hgt)+\zz*(0.166666667*0.14285*\wi,0)$);
\draw [fill] (A\xx\yy\zz)  circle [radius=\vx];
}
}

\foreach \xx/\yy in {2/2,2/5,3/4,3/5,4/4,4/5,5/5,5/4}
{
\foreach \zz in {3}
{
\coordinate (A\xx\yy\zz) at ($\xx*(0.14285*\wi,0)-0.5*(0.14285*\wi,0)+\yy*(0,0.2*\hgt)-0*(0,0.2*\hgt)-\zz*(0,0.166666667*0.2*\hgt)$);
\draw [fill] (A\xx\yy\zz)  circle [radius=\vx];
}
}

\foreach \xx/\yy in {1/5,2/4,6/4}
{
\foreach \zz in {3}
{
\coordinate (A\xx\yy\zz) at ($\xx*(0.14285*\wi,0)-1*(0.14285*\wi,0)+\yy*(0,0.2*\hgt)-0*(0,0.2*\hgt)+\zz*0.166666667*(0.14285*\wi,-0.2*\hgt)$);
\draw [fill] (A\xx\yy\zz)  circle [radius=\vx];
}
}

\foreach \xx/\yy in {1/2,6/5,2/3}
{
\foreach \zz in {3}
{
\coordinate (A\xx\yy\zz) at ($\xx*(0.14285*\wi,0)-1*(0.14285*\wi,0)+\yy*(0,0.2*\hgt)-1*(0,0.2*\hgt)+\zz*0.166666667*(0.14285*\wi,0.2*\hgt)$);
\draw [fill] (A\xx\yy\zz)  circle [radius=\vx];
}
}

\foreach \xx/\yy/\xxx/\yyy in {1/2/1/3,1/4/1/5,1/5/2/5,2/5/3/5,3/5/4/5,2/3/2/4,4/5/5/5,6/4/6/5,5/4/6/4}
{
\foreach \zz in {3}
{
\draw (A\xx\yy\zz) -- (A\xxx\yyy\zz);
}
}
\foreach \xx/\yy/\xxx/\yyy in {1/3/1/4}
{
\foreach \zz in {3}
{
\draw  (A\xx\yy\zz) -- (A\xxx\yyy\zz);
}
}

\foreach \xx/\yy/\xxx/\yyy in {5/5/6/5,1/2/2/2}
{
\foreach \zz/\zzz in {3/3}
{
\draw (A\xx\yy\zz) -- (A\xxx\yyy\zzz);
}
}

\foreach \xx/\yy/\xxx/\yyy in {4/4/5/4,2/4/3/4}
{
\foreach \zz in {3}
{
\draw (A\xx\yy\zz) -- (A\xxx\yyy\zz);
}
}
\foreach \xx/\yy/\xxx/\yyy in {3/4/4/4}
{
\foreach \zz in {3}
{
\draw (A\xx\yy\zz) -- (A\xxx\yyy\zz);
}
}

\draw ($0.5*(A143)+0.5*(A133)-(0.6,0)$) node {\footnotesize $P:$};
\draw ($(A143)-(0.2,0)$) node {\footnotesize $i$};
\draw ($(A133)-(0.2,0)$) node {\footnotesize $j$};
\draw ($(A223)+(0.2,0)$) node {\footnotesize $j'$};
\draw ($(A233)+(0.2,0)$) node {\footnotesize $i'$};

\end{tikzpicture}
\end{center}\vspace{-0.5cm}
\caption{As shown on the left, having found a maximal collection of vertex-disjoint dense spots $G_1,\ldots,G_t$ in a graph $G$, we take a `sample set' $X$ which will intersect with each of these dense spots. Partitioning $V(G)\setminus (X\cup V(G_1)\cup \ldots\cup V(G_t))$ into sets $A_1\cup \ldots\cup A_{\ell+1}$, we then cover most of the edges between each $A_i$ and $A_j$ with matchings $M_{ij,k}$, for $k\in [d']$ with $d'\approx d/s$. For a path $P$ with length $\ell$ in $K_{\ell+1}$ (as on the right, with endvertices $i',j'\in V(K_{\ell+1}$) and each $k$ we put together the matchings $M_{ij,k}$ for each $ij\in E(P)$ to get a path forest of few paths, most of which end in $A_{i'}$ and $A_{j'}$. We then connect up as many as possible of these endvertices with short paths using internal vertices in $X$ (shown in red) and connect more of them into the dense spots with short paths using vertices in $X$ (shown in orange).
}\label{fig:proofsketch}
\end{figure}
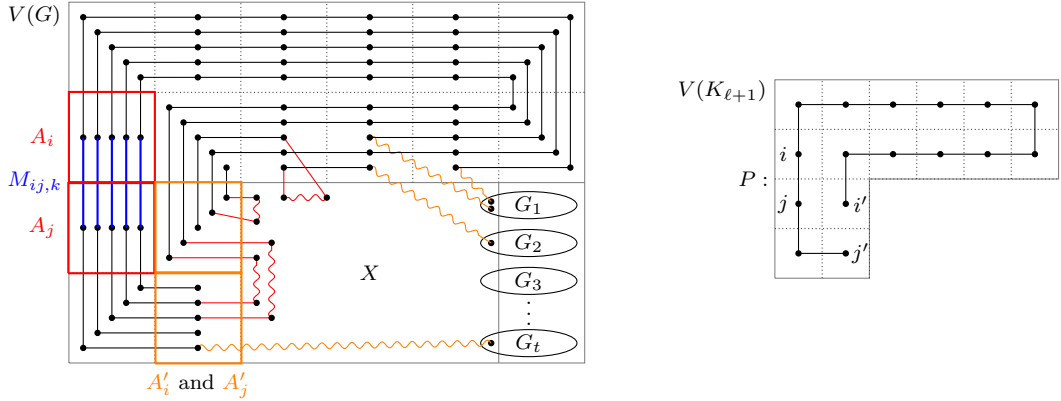

By this scheme, we will find the required decomposition of all but at most $\eps n d$ edges of $G$ into paths of length $\ell=(1-\eps)d$, but in setting aside the collection $\mathcal{F}$ we have introduced a further issue. While we can choose $X$ so that $G-X$ is nearly regular, we now have the problem that the graph $G-V(\mathcal{F})-X$ may not be nearly regular (which we require for partitioning into vertex sets and finding many edge-disjoint matchings between them). Critically, here, we will use two results shown by very recent techniques of Chakraborti, Janzer, Methuku, and Montgomery~\cite{regularising,edgedisjointcycles} to efficiently find nearly-regular subgraphs. Moreover, we will develop these results non-trivially, using them to show that we can mostly decompose $G-V(\mathcal{F})-X$ into a small number of edge-disjoint nearly-regular subgraphs (see Section~\ref{sec:nearregular}). We also apply one of the results of~\cite{regularising,edgedisjointcycles} while decomposing the dense spots along with the attached paths.

This completes an overall sketch of our methods for Theorem~\ref{Theorem_main}. In Section~\ref{sec:divisionintothreelemmas}, we will show that Theorem~\ref{Theorem_main} follows from three key lemmas, Lemmas~\ref{Lemma_decomposition_path_forests}, \ref{Lemma_X_connects}, and \ref{Lemma_decompose_inside_dense_bits}, which roughly correspond to steps \ref{sketch:outside}, \ref{sketch:connect} and \ref{sketch:inside} above, respectively. The proofs of each of these three key lemmas will need more new ideas than could be included in the sketch above, and they are discussed in more detail later where appropriate.
In Section~\ref{sec:tools}, we cover some tools we will need, including some comments on our use of the local lemma in Section~\ref{sec:local}. In Section~\ref{sec:outsidedensebits}, we give the decomposition into path forests we will use for $G-V(\mathcal{F})-X$, thus proving Lemma~\ref{Lemma_decomposition_path_forests}. In Section~\ref{sec:connecttodensebits}, we will show the existence of a good `sample set' $X$ and show how to connect some paths to small dense spots within $G[X]$ and hence to some dense spot in $\mathcal{F}$, thus proving Lemma~\ref{Lemma_X_connects}. In Section~\ref{sec:insidedensebits}, we will decompose the dense spots along with some attached paths, proving Lemma~\ref{Lemma_decompose_inside_dense_bits} and hence completing the proof of Theorem~\ref{Theorem_main}. Finally, in Section~\ref{sec:concluding}, we will make some concluding remarks.


\subsection{Proof of Theorem~\ref{Theorem_main} subject to three key lemmas}\label{sec:divisionintothreelemmas}
We now introduce two key definitions to formalise some of the notions in Section~\ref{sec:proofsketch}. The first we will use to quantify what we want from our `dense spots'.
\begin{restatable}{definition}{defdensespot}\label{def:dense}
    A graph $G$ is $(\eta, d, K)$-\textbf{dense} if $0<|V(G)|\leq Kd$ and $\delta(G)\geq (1-\eta)d$.
\end{restatable}
Our second definition records properties of path forests, where we require each forest to not contain too many paths and that the endpoints of the paths are relatively well spread out around a graph.
\begin{definition}\label{def:bounded}
Say a collection of path forests $\mathcal{P}_1,\ldots,\mathcal{P}_\ell$ is \textbf{$(m,\Delta_0,\Delta_1)$-bounded} if each $\mathcal{P}_i$, $i\in [\ell]$, has at most $m$ paths, each vertex appears as the endvertex of at most $\Delta_0$ paths in total in all of the path forests, and, for each $i\in [\ell]$, each vertex has at most $\Delta_1$ neighbours among the endvertices of $\mathcal{P}_i$.
\end{definition}

We can now give our lemma which we will use to find a good collection of path forests for step \ref{sketch:outside} in Section~\ref{sec:proofsketch}. The lemma applies more generally than to a $d$-regular graph as we apply it to a regular graph with a maximal set of dense spots (and a `sample set') removed. 

\begin{restatable}{lemma}{lemmadecompoutsidedense}\label{Lemma_decomposition_path_forests}
    Let $1/d \ll \gamma \ll \eps\leq 1$. Let $G$ be an $n$-vertex graph with maximum degree at most $d$ in which all but at most $\gamma n$ vertices have degree at least $(1-\gamma)d$.
     Then, $G$ contains a $((1+\eps)n/d,d^{1/4},d^{1/4})$-bounded edge-disjoint collection of $d/2$ path forests which cover all but at most $\eps nd$ of the edges of $G$.
\end{restatable}
 Note that (as seen by the disjoint union of $(d+1)$-vertex cliques), the bound $(1+\eps)n/d$ in Lemma~\ref{Lemma_decomposition_path_forests} is close to optimal. Given $d/2$ edge-disjoint path forests which each contain $(1+\eps)n/d\leq 2n/d$ paths in an $n$-vertex $d$-regular graph, note that there are in total at most $2n$ endvertices, so that on average each vertex appears in total as the endvertex of at most 2 paths. Thus, the corresponding upper bound of $d^{1/4}$ in Lemma~\ref{Lemma_decomposition_path_forests} is relatively unambitious; while it could be pushed much further with our methods this will be sufficient for our purpose. Similarly, for just one of these path forests, on average a vertex can expect to have at most $2$ neighbours among the endpoints of the at most $(1+\eps)n/d\leq 2n/d$ paths. Again, the corresponding upper bound of $d^{1/4}$ in Lemma~\ref{Lemma_decomposition_path_forests} is chosen loosely, only so that it is comfortably enough for our later proofs.

Next, we give the lemma that represents the heart of our proof. It shows the existence of our `sample set' $X$ such that, given a collection of path forests as produced by Lemma~\ref{Lemma_decomposition_path_forests}, we can join paths together using vertices from $X$ and find further paths to get a collection of path forests in which the paths are connected into the dense spots (in a well spread manner similar to the conditions in Definition~\ref{def:bounded}) or can be collectively mostly decomposed into paths with length $(1-\eps)d$. 

\stepcounter{propcounter}
\begin{restatable}{lemma}{lemmaconnecttodense}\label{Lemma_X_connects}
Let $1/d\ll p,\eta  \polysmall1/K \polysmall \eps $.
 Let $G$ be a $d$-regular $n$-vertex graph.
 Let $\mathcal{F}=\{G_1, \ldots, G_t\}$ be a maximal family of vertex-disjoint $(\eta, d, K)$-dense subgraphs. Then, there is a set $X\subset V(G)$ with
 \begin{enumerate}[label = {{\emph{\textbf{A\arabic{enumi}}}}}]
 \item $|X|\leq 2pn$ and, for each $v\in V(G)$, $d_G(v,X)=(1\pm \eta)pd$, and\labelinthm{propX:minusone}
\item   for each $j\in [t]$, $|X\cap V(G_j)|\leq 2p|V(G_j)|$ and, for each $v\in V(G_j)$, $d_G(v,X\cap V(G_i))=(1\pm 3\eta) 2pd$,\labelinthm{propX:zero}
\end{enumerate}
  such that the following holds.

For any $(2n/d,d^{1/4},d^{1/4})$-bounded edge-disjoint collection of path forests $\mathcal{P}_1, \mathcal{P}_2, \ldots, \mathcal{P}_{d/2}$ in $G - V(\mathcal{F})-X$ we can find in $G-(V(\mathcal{F})\setminus X)$ an edge-disjoint collection of path forests $\mathcal{P}'_1, \mathcal{P}'_2, \ldots, \mathcal{P}'_{d/2}$, such that the following hold.
\begin{enumerate}[label = {{\emph{\textbf{A\arabic{enumi}}}}}]\addtocounter{enumi}{2}
    \item For each $i\in [d/2]$, every path in $\mathcal{P}'_{i}$ has both of its endvertices in $X\cap V(\mathcal{F})$.\labelinthm{propX:one}
    \item \labelinthm{propX:twob} Each vertex in $G$ is in total an endvertex of at most $\sqrt{d}$ paths in $\mathcal{P}'_i$, $i\in [d/2]$.
    \item \labelinthm{propX:twoc} For each $j\in [t]$ and $i\in [d/2]$ at most $\sqrt{d}$ of the paths in $\mathcal{P}'_i$ end in $G_j$.
    \item All but at most $\eps nd/4$ of the edges of $E(\mathcal{P}_1\cup \ldots \cup \mathcal{P}_{d/2})\setminus  E(\mathcal{P}'_1\cup \ldots \cup \mathcal{P}'_{d/2})$ can be decomposed into copies of $P_{(1-\eps)d}$.\labelinthm{propX:three}
\end{enumerate}
\end{restatable}

Finally, we give our lemma which can decompose the dense spots along with some paths which all have their endpoints among the dense spots in a relatively well spread manner, as follows.

    \stepcounter{propcounter}
\begin{restatable}{lemma}{decomposedensespots}\label{Lemma_decompose_inside_dense_bits}
 Let $1/d\ll  \eta \ll p, 1/K\polysmall \eps$.
    Let $G$ be a graph consisting of a vertex-disjoint family $\mathcal{F}=\{G_1,\ldots,G_t\}$ of $(\eta,d,K)$-dense graphs and path forests $\mathcal{P}_1,\ldots, \mathcal{P}_{d/2}$ where each path in each path forest has both its endpoints in $V(\mathcal{F})$ and all its internal vertices not in $V(\mathcal{F})$. Suppose also that the following properties hold.

    \begin{enumerate}[label = {{\emph{\textbf{B\arabic{enumi}}}}}]
        \item\labelinthm{prop:densedecomp1} For each $i\in [t]$, there is some $X_i\subseteq V(G_i)$ so that $X_i$ contains all of the vertices of $V(\mathcal{P}_1)\cup\ldots\cup V(\mathcal{P}_{d/2})$ in $V(G_i)$ (which are necessarily endpoints), and, for each $v\in V(G_i)$, $d_{G_i}(v,X_i)=(1\pm \eta)pd$.
        \item\labelinthm{prop:densedecomp2}  Each vertex $v\in V(\mathcal{F})$ is an endpoint of in total at most $\sqrt{d}$ paths from $\mathcal{P}_1,\ldots, \mathcal{P}_{d/2}$.
        \item\labelinthm{prop:densedecomp3} For each $j\in [t]$ and $i\in [d/2]$ at most $\sqrt{d}$ of the paths in $\mathcal{P}_i$ have at least one endpoint in $G_j$.
    \end{enumerate}
    Then, all but at most $\eps d\cdot |V(\mathcal{F})|$ of the edges of $G$ can be decomposed into copies of $P_{(1-\eps)d}$.
\end{restatable}


Subject to the proof of these three key lemmas, we can now prove Theorem~\ref{Theorem_main}, following the steps \ref{sketch:first}--\ref{sketch:last} in Section~\ref{sec:proofsketch}.

\begin{proof}[Proof of Theorem~\ref{Theorem_main}]
Let $K\in \N$ and $\eta,p>0$ satisfy
\[
1/d\ll  \eta\ll p \ll 1/K \ll \eps\leq 1.
\]
Let $G$ be an $n$-vertex $d$-regular graph. Let $\mathcal{F}=\{G_1,\ldots,G_t\}$ be a maximal collection of vertex-disjoint $(\eta,d,K)$-dense graphs in $G$.
Using Lemma~\ref{Lemma_X_connects}, let $X\subset V(G)$ be such that the following hold.
\stepcounter{propcounter}
\begin{enumerate}[label = {{{\textbf{\Alph{propcounter}\arabic{enumi}}}}}]
\item $|X|\leq 2pn$, and, for each $v\in V(G)$, $d_G(v,X)\leq 2pd$.\label{C1}
\item For each $i\in [t]$, $|X\cap V(G_i)|\leq 2p|V(G_i)|$ and, for each $v\in V(G_i)$, $d_{G}(v,X\cap V(G_i))=(1\pm \eta)pd$.\label{littleprop:2}
\item For any $(2n/d,d^{1/4},d^{1/4})$-bounded edge-disjoint collection of path forests $\mathcal{P}_1, \mathcal{P}_2, \ldots, \mathcal{P}_{d/2}$ in $G - V(\mathcal{F})-X$, we can find in $G-V(\mathcal{F}\setminus X)$ an edge-disjoint collection of path forests $\mathcal{P}'_1, \mathcal{P}'_2, \ldots, \mathcal{P}'_{d/2}$ such that the following hold.\label{prop:Xuseful}
\begin{enumerate}[label = \textbf{\roman{enumii})}]
   \item For each $i\in [d/2]$, every path in $\mathcal{P}'_{i}$ has both of its endvertices in $X\cap V(\mathcal{F})$.\label{littleprop:1}
   \item Each vertex in $V(G)$ is an endvertex of at most $\sqrt{d}$ paths in $\mathcal{P}'_i$, $i\in [d/2]$.\label{littleprop:2b}
   \item For each $j\in [t]$, and each $i\in [d/2]$, at most $\sqrt{d}$ of the paths in $\mathcal{P}_i'$ end in $G_j$.\label{littleprop:2c}
   \item All but at most $\eps nd/4$ edges of $E(\mathcal{P}_1\cup \ldots \cup \mathcal{P}_{d/2})\setminus  E(\mathcal{P}'_1\cup \ldots \cup \mathcal{P}'_{d/2})$ can be decomposed into copies of $P_{(1-\eps)d}$.\label{littleprop:3}
\end{enumerate}
\end{enumerate}

Let $G'=G-V(\mathcal{F})-X$.
If $|V(G')|\leq \eps n/2$, then let $\mathcal{P}_1=\ldots=\mathcal{P}_{d/2}=\emptyset$, noting that these cover all but at most $\eps n d/4$ edges of $G'$ and they therefore trivially form a $(2n/d,d^{1/4},d^{1/4})$-bounded collection of paths which cover all but at most $\eps nd/4$ of the edges of $G'$. Suppose, otherwise, that $|V(G')|>\eps n/2$.
As each $G_i$, $i\in [t]$, is $(\eta,d,K)$-dense, the number of edges incident on $V(G')$ with, for some $i\in [t]$, one vertex in some $G_i$ is at most $\sum_{i\in [t]}|V(G_i)|\cdot \eta d=\eta d|V(\mathcal{F})|$. Therefore, by \ref{C1}, 
\[
|E(G')|\geq \frac{1}{2}d|V(G')|-\eta d|V(\mathcal{F})|-2pd|X|\geq \frac{1}{2}d(1-5p)|V(G')|.
\]
Let $\gamma=\sqrt{5p}$.
As $\Delta(G')\leq d$, if $N$ is the number of vertices in $G'$ with degree at most $(1-\gamma)d$ in $G'$, then $N\cdot \gamma d\leq 5dp|V(G')|$, so that $N\leq \gamma|V(G')|$. Thus, as $1/d\ll p\ll\eps$, by Lemma~\ref{Lemma_decomposition_path_forests}, $G'$ contains a $(2n/d,d^{1/4},d^{1/4})$-bounded edge-disjoint collection of path-forests
$\mathcal{P}_1,\ldots,\mathcal{P}_{d/2}$ which together cover all but at most $\eps nd/4$ of the edges of $G'=G-V(\mathcal{F})-X$.

Then, by \ref{prop:Xuseful}, we can find an edge-disjoint collection of path-forests $\mathcal{P}'_1,  \ldots, \mathcal{P}'_{d/2}$ in $G-V(\mathcal{F})$ such that \ref{littleprop:1}--\ref{littleprop:3} hold. 
Using \ref{littleprop:3}, let $\mathcal{Q}_1$ be a set of edge-disjoint copies of $P_{(1-\eps)d}$ which decompose
all but at most $\eps nd/4$ edges of $E(\mathcal{P}_1\cup \ldots \cup \mathcal{P}_{d/2})\setminus  E(\mathcal{P}'_1\cup \ldots \cup \mathcal{P}'_{d/2})$.

Let $G''=\big(\bigcup_{i\in [t]}G_i\big)\cup \big(\bigcup_{i\in [d/2]}\bigcup_{P\in \mathcal{P}_i'}P\big)$. By Lemma~\ref{Lemma_decompose_inside_dense_bits}
(with $X_i=X\cap V(G_i)$ for each $i\in [t]$, $\eta'=3\eta$,  and using \ref{littleprop:2} and \ref{littleprop:1}--\ref{littleprop:2c}), let $\mathcal{Q}_2$ be a set of edge-disjoint copies of $P_{(1-\eps)d}$ which decomposes
all but at most $\eps nd/4$ of the edges of $G''$.

Then, $\mathcal{Q}_1\cup \mathcal{Q}_2$ is an edge-disjoint collection of copies of $P_{(1-\eps)d}$ which decomposes all but at most $\eps nd/2$ edges of $\mathcal{P}_1,\ldots,\mathcal{P}_{d/2},\mathcal{F}$, and therefore all but at most $3\eps nd/4$ of the edges in $E(G')\cup E(\mathcal{F})$. As $G_i$ is $(\eta,d,K)$-dense for each $i\in [t]$, each vertex in $V(\mathcal{F})$ has at most $\eta d$ neighbours in $V(G')$, so that there are at most $\eta n d$ edges in $G$ between $V(\mathcal{F})$ and $V(G')$. Furthermore, as $|X|\leq 2pn$ by \ref{C1}, there are at most $2pnd$ edges in $G$ between $X$ and $V(G)\setminus X$. Therefore, as the only edges of $G$ which are not in $E(G')\cup E(\mathcal{F})$ are those with a vertex in $X$ or which are between $G'$ and $V(\mathcal{F})$, $\mathcal{Q}_1\cup \mathcal{Q}_2$ decomposes all but at most $3\eps nd/4+\eta nd+2pnd\leq \eps nd$ of the edges of $G$ into copies of $P_{(1-\eps)d}$, as required.
\end{proof}


\subsection{Tools}\label{sec:tools}
We collect here some standard results we will use in our proofs.
\subsubsection{Concentration inequalities}
We will use the following standard version of Chernoff's bound (see, for example, \cite[Corollary 2.2 and Theorem 2.10]{janson2011random}).
\begin{lemma}[Chernoff's bound]\label{chernoff} 
Let $X$ be a random variable with mean $\mu$ which is binomially distributed or hypergeometrically distributed. 
Then, for any $0<\gamma<1$, we have that $\mathbb{P}(|X-\mu|\geq \gamma \mu)\leq 2e^{-\mu \gamma^2/3}$.
\end{lemma}
Given a product probability space $\Omega = \prod_{i\in[n]} \Omega_i$, a random variable $X\colon \Omega\to \mathbb{R}$ is called $C$-Lipschitz if $|X(\omega)-X(\omega')|\leq C$ whenever $\omega$ and $\omega'$ differ in at most $1$ co-ordinate. We will use the following standard version of Azuma's inequality (see, for example, \cite[Corollary 2.27]{janson2011random}).
\begin{lemma}[Azuma's inequality]\label{lem:Azuma}
Let $X$ be $C$-Lipschitz random variable on a product probability space with $n$ co-ordinates. Then, for any $t>0$,
$$\mathbb{P}(|X-\mathbb{E}(X)|> t)\leq 2e^{\frac{-t^2}{2nC^2}}.$$
\end{lemma}

\subsubsection{The local lemma}\label{sec:local}
We will use the local lemma of Lov\'asz (see~\cite{alon2016probabilistic}) many times throughout our proofs, so after stating the form that we use we will make some remarks on how we use it.%
\begin{lemma}[The local lemma]\label{Lemma_local_symmetric} Let $B_1,B_2,\ldots, B_n$ be events in an arbitrary probability space.  Suppose that, for each $i\in [n]$, the event $B_i$ is mutually independent of a set of at most $\Delta$ of the other events and $\mathbb{P}(B_i)\leq p$.
 Then, if $ep(\Delta+1)\leq 1$, the probability that none of the events $B_i$, $i\in [n]$, occurs is strictly positive.
\end{lemma}

Each time we apply the local lemma we will do so within a $d$-regular (or approximately $d$-regular) $n$-vertex graph $G$. The `bad events' $B_i$, $i\in I$, that we use, where $I$ is some appropriate index set, will all be of the form that `a binomial or hypergeometric variable is not concentrated around its mean', where these variables will depend on some partition of the vertices (or in one case of the edges) for which each vertex (or edge) is placed in a set of the partition independently at random.
Thus, using Lemma~\ref{chernoff}, we can bound the probability of each `bad event' occurring, which in each instance will be at most $p:=e^{-d^{0.1}/100}$. 

Each bad event will, for some central vertex, depend only on the placement of vertices within a distance $k$ of the central vertex, for some $k$ with $1/d\ll 1/k$. Changing the placement of a single vertex (or edge) in the partition will affect at most $d^k$ of the `bad events'. Therefore, for each `bad event', $B$ say, after identifying the central vertex, we take the set of events which depend on any vertex within a distance $k$ of the central vertex, and note that $B$ is mutually independent of the set of all the other `bad events'. As we have a graph with maximum degree $d$, the number of events we omitted from the set is at most $\Delta:=2d^k\cdot d^k$. Therefore, as $1/d\ll 1/k$, we comfortably have that $ep(\Delta+1)= e\cdot e^{-d^{0.1}/100} (2d^k+1)\cdot d^k\leq 1$, and we can apply the local lemma, Lemma~\ref{Lemma_local_symmetric}, to show that there is some instance in which none of the `bad events' occur.

Let us also note that, when for example choosing a subset $X\subset V(G)$ by including vertices independently at random with probability $p$, the `bad events' that we wish to consider may include that $|X|\neq (1\pm \gamma)pn$ (where $1/d\ll p,\gamma$). This event will not be independent of any other bad event we define as it is influenced by every vertex in $G$. We could deal with this using the asymmetric local lemma, but, instead, to keep using Lemma~\ref{Lemma_local_symmetric}, we will take an arbitrary partition of $V(G)$ into sets $V_1,\ldots,V_t$ which each have size between $d/2$ and $d$. If, for each $i\in [t]$, $B_i$ is the event that $V_i\cap X\neq (1\pm \gamma)p|V_i|$ then, when no event $B_i$, $i\in [t]$, holds then $|X|=(1\pm \gamma)pn$ and, furthermore, each of these events only depends on the location of at most $d$ vertices.

As an example of how we use the local lemma, we will give here the details of our first application, from Lemma~\ref{sec:initialpathforest}. Let $s=2\lceil d^{0.15}\rceil$ and $\eta=4d^{-0.4}$ with $1/d\ll\eta \ll 1$. Suppose that $G$ has vertices degrees which are $(1\pm \eta/2)d$ and let $V(G)=V_1\cup \ldots\cup V_t$ be a partition of $V(G)$ into sets of size between $d/2$ and $d$. Let $V(G)=A_1\cup \ldots\cup A_{s}$ be a partition of $V(G)$ formed by taking each $v\in V(G)$ and, independently at random, placing it in each $A_i$, $i\in [s]$, with probability $1/s$.
For each $v\in V(G)$ and $i\in [s]$, let $B_{v,i}$ be the event that $d_{G}(v,A_i)\neq(1\pm \eta)d/s$. 
For each $i\in [s]$ and $j\in [t]$, let $B_{i,j}$ be the event that $|A_i\cap V_j|\neq (1\pm \eta)|V_j|/s$.
Note that, by Chernoff's bound, each of these events occur with probability at most $\exp(-\eta^2\cdot (d/2s)/3)\leq \exp(-d^{0.1})$. For each $v\in V(G)$ and $i\in [s]$, the event $B_{v,i}$ only depends on whether each vertex in $N_G(v)$ is in $A_i$ or not, so therefore on the placement of at most $d$ vertices in the partition $A_1\cup\ldots\cup A_s$. Similarly, for each $i\in [s]$ and $j\in [t]$, $B_{i,j}$ only depends on the placement of at most $d$ vertices in the partition $A_1\cup\ldots\cup A_s$.
For any vertex $w\in V(G)$, the placement of $w$ only affects an event $B_{v,i}$ with $v\in V(G)$ and $i\in [s]$ if $v$ is a neighbour of $w$, so at most $ds$ pairs $(v,i)$. Furthermore, the placement of $w$ only affects an event $B_{i,j}$ with $i\in [s]$ and $j\in [t]$ if $w\in V_j$, and thus affects only at most $s$ such events. Therefore, the placement of each vertex affects at most $ds+s\leq d^2$ events. Therefore, each event in $\{B_{v,i}:v\in V(G),i\in [s]\}\cup \{B_{i,j}:i\in [s],j\in [t]\}$ is mutually independent of a set of all but at most $\Delta:=d\cdot d^2=d^3$ other events. Thus, as $ep\Delta=e\cdot \exp(-d^{0.1})\cdot d^3\leq 1$, by Lemma~\ref{Lemma_local_symmetric}, we have that, with positive probability, none of the events in $\{B_{v,i}:v\in V(G),i\in [s]\}\cup \{B_{i,j}:i\in [s],j\in [t]\}$ hold. Therefore, we can take a partition of $V(G)$ into $A_1\cup \ldots\cup A_{s}$ for which none of these events hold, and hence, for each $v\in V(G)$ and $i\in [s]$ we have $d_G(v,A_i)=(1\pm \eta)d/s$ and, for each $i\in [s]$, as no event $B_{i,j}$, $j\in [t]$, holds, we have $|A_i|=(1\pm \eta)n/s$.

Finally, here let us highlight the instance in which we use the most `bad events', in the proof of Lemma~\ref{Lemma_good_sample_exists}. There, working with a $d$-regular graph $G$, we will consider the collection $\mathcal{S}$ of all subsets of $V(G)$ of size at most $k$ such that these vertices are all pairwise at most $k$ apart in $G$, where $1/d\ll 1/k$. Each vertex in $G$ will be allocated to some set $X$ or not independently at random with probability $p$. There will be some bad events $B_i$, $i\in I$, where, for each $U\in \mathcal{S}$, at most $d|U|+1\leq dk+1$ of the indices $i\in I$ will contain $U$. Any event $B_i$ whose index includes $U\in \mathcal{S}$ will depend only whether or not the vertices in $\bigcup_{v\in U}N_G(v)$ are in $X$, where this set has size at most $|U|d\leq dk$. As each vertex $w$ will appear as a neighbour of at most $(2d^{2k+1})^{k-1}$ sets in $\mathcal{S}$, we will have that at most $d^{2k^2}\cdot (dk+1)$ events $B_i$, $i\in I$, depend on whether or not $w$ is in $X$. As $1/d\ll 1/k$, then, we can think of this as every event depending on at most a polynomial of $d$ (whose power is a function of $k$) many other bad events, while the probability of each bad event will be $\exp(-d^{0.1}/100)$, allowing us to apply Lemma~\ref{Lemma_local_symmetric} once again.

\subsubsection{Matchings in bipartite graphs}
In order to find the matchings $M_{ij,k}$ described in Section~\ref{sec:proofsketch}, we will use the following result.
\begin{restatable}{lemma}{vizinglemma}\label{lem:vizing} Let $d,n>0$ and let $0<\gamma<1$ satisfy $\gamma d\geq 1$. Let $G$ be a bipartite graph with vertex class sizes $(1\pm \gamma)n$ and $d(v)=(1\pm \gamma)d$ for each $v\in V(G)$. Then, $G$ contains $(1-10\sqrt{\gamma})d$ edge-disjoint matchings which have size $(1-10\sqrt{\gamma})n$.
\end{restatable}
\begin{proof} Note that $e(G)\geq (1-\gamma)d\cdot (1-\gamma)n> (1-2\gamma)nd$. 
By Vizing's theorem, the edge set of $G$ can be partitioned into at most $(1+\gamma)d+1\leq (1+2\gamma)d$ matchings, where we used that $\gamma d\geq 1$. If the statement of the lemma fails, then 
\[
e(G)\leq (1-10\sqrt{\gamma})d\cdot (1+\gamma)n + (10\sqrt{\gamma} +2\gamma)d\cdot (1-10\sqrt{\gamma})n= (1-97\gamma +30\gamma^{3/2})nd \leq (1 - 2\gamma)nd,
\] 
a contradiction.
\end{proof}

\subsubsection{Decomposing complete graphs into paths}
In decomposing our dense spots, we will use the following result from \cite{bottcher2016approximate}, which is the special case for paths of a result for bounded degree trees. 
\begin{theorem}[Böttcher-Hladk\'y-Piguet-Taraz] \label{packingpaths} Let $1/n\ll \eps$. Let $\mathcal{P}$ be a collection of paths of length at most $n$ whose total length is at most $\binom{n}{2}$. Then, $\mathcal{P}$ packs into $K_{(1+\eps)n}$.
\end{theorem}

\section{Decomposing outside the dense bits}\label{sec:outsidedensebits}

In this section, we  prove Lemma~\ref{Lemma_decomposition_path_forests}. We begin, in Section~\ref{sec:initialpathforest}, by finding an initial collection of path forests which are, together, weakly bounded (see Definition~\ref{def:bounded}), before using this to find a collection of path forests in Section~\ref{sec:betterpathforest} which are more strongly bounded. In Section~\ref{sec:pathcovers}, we give a short proof of Theorem~\ref{thm:pathcovers} from this. In Section~\ref{sec:nearregular}, we take near-regularisation results from~\cite{regularising,edgedisjointcycles} and develop them to find very nearly regular subgraphs in a roughly regular graph. Finally, we put this all together in Section~\ref{sec:prooflemmadecomppathforest} to prove Lemma~\ref{Lemma_decomposition_path_forests}.

\subsection{Initial path forests}\label{sec:initialpathforest}
We prove the following lemma by partitioning the vertex set of a regular graph using the local lemma and putting together matchings between different vertex classes, as sketched in Section~\ref{sec:proofsketch}.
\begin{restatable}{lemma}{lemmadecompoutsidedensee}\label{Lemma_decomposition_path_forests_regular}
    Let $1/d \ll  \eps\leq 1$ and $\gamma=2d^{-0.4}$. Let $G$ be an $n$-vertex graph with vertex degrees in $(1\pm \gamma)d$.
     Then, $G$ contains an $(n/d^{1/8},d^{7/8},d^{7/8})$-bounded edge-disjoint collection $\mathcal{P}_1,\ldots, \mathcal{P}_{d/2}$ of path forests which cover all but at most $\eps nd$ edges of $G$.
\end{restatable}
\begin{proof}   Let $s=2\lceil d^{0.15}\rceil$ and $\eta=4d^{-0.4}$.
Let $V(G)=A_1\cup \ldots\cup A_{s}$ be a partition of $V(G)$ formed by taking each $v\in V(G)$ and, independently at random, placing it in each $A_i$, $i\in [s]$, with probability $1/s$.

By the local lemma (see Section~\ref{sec:local} for details), with positive probability, we can have the following properties.
\stepcounter{propcounter}
\begin{enumerate}[label = {{{\textbf{\Alph{propcounter}\arabic{enumi}}}}}]
\item For each $v\in V(G)$ and $i\in [s]$, $d_{G}(v,A_i)=(1\pm \eta)d/s$.\label{prop:degreeintoAi}
\item For each $i\in [s]$, $|A_i|=(1\pm \eta)n/s$.\label{prop:Aisize}
\end{enumerate}

Let $\eta'=20\sqrt{\eta}=40d^{-0.2}$ and $d'=(1-\eta')d/s$. For each edge $e=jk$ in the complete $s$-vertex graph $K_s$, using Lemma~\ref{lem:vizing}, \ref{prop:degreeintoAi} and \ref{prop:Aisize},  find $d'$ edge-disjoint matchings in ${G}[A_j,A_k]$ which each have at least $(1-\eta')n/s$ edges.
Call these matchings $M_{e,i}$, $i\in [d']$.
Now, take paths $Q_1,\ldots,Q_{s/2}$ of length $s-1$ in the complete graph $K_s$ in which each vertex is the endpoint of exactly one path. (Such a decomposition is well-known to exist when $s$ is even. See, for example, Figure~\ref{pathtorotate}.)
For each $j\in [s/2]$ and $k\in [d']$, let $\mathcal{P}_{j,k}$ be the union of all the matchings $M_{e,k}$ with $e\in E(Q_j)$ and let $\mathcal{P}'_{j,k}$ be the subcollection of paths in $\mathcal{P}_{j,k}$ of length $s-1$.

\begin{figure}
\begin{center}
\begin{tikzpicture}[line join=bevel]
\def\rad{0.8}
\def\vx{0.03cm}

\foreach \n in {1,...,12}
{
\coordinate (A\n) at (\n*30:\rad);
}
\begin{scope}[transform canvas={rotate=-30}]
\draw [thick,red!50,densely dotted] (A3) -- (A8) -- (A4) -- (A7) -- (A5)-- (A6);
\draw [thick,red!50,densely dotted] (A3) -- (A9) -- (A2) -- (A10) -- (A1) -- (A11) -- (A12);
\end{scope}

\draw [thick] (A3) -- (A8) -- (A4) -- (A7) -- (A5)-- (A6);
\draw [thick] (A3) -- (A9) -- (A2) -- (A10) -- (A1) -- (A11) -- (A12);

\foreach \n in {1,...,12}
{
\draw [fill] ($1*(A\n)$) circle [radius=1.5*\vx];
}
\end{tikzpicture}\end{center}\vspace{-0.5cm}
\caption{For $s=12$, an embedding of a path of length $s-1=11$ which can be rotated $s/2=6$ times to form together a complete graph with $s=12$ vertices, with one clockwise rotation depicted in red, giving a decomposition of $K_s$ into paths of length $s-1$ in which each vertex appears as an endpoint exactly once.}\label{pathtorotate}
\end{figure}
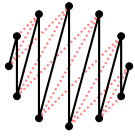

\begin{claim}\label{nottoomanymissing} For each $j\in [s/2]$ and $k\in [d']$, there are at most $\eps n/3$ edges of
$\mathcal{P}_{j,k}$ which are not in $\mathcal{P}'_{j,k}$.
\end{claim}
\begin{proof} Set $j\in [s/2]$ and $k\in [d']$. For each $i\in [s]$ which is not an endpoint of $Q_k$, note that the number of vertices in $A_i$ which can be an endpoint of a path in $\mathcal{P}_{j,k}$ is at most $2(|A_i|-(1-\eta')n/s)\leq 3\eta'n/s$, using \ref{prop:Aisize}. Therefore, at most $3\eta'n$ paths in $\mathcal{P}_{j,k}$ can have an endpoint in some $A_i$ where $i$ is a middle vertex of $Q_k$. As each path in $\mathcal{P}_{j,k}\setminus \mathcal{P}'_{j,k}$ has such an endpoint, there are at most $(s-1)\cdot 3\eta' n\leq \eps n/3$ edges of $\mathcal{P}_{j,k}$ which are not in $\mathcal{P}'_{j,k}$, where we have used that $s=2\lfloor d^{0.15}\rfloor$, $\eta'=40d^{-0.2}$ and $1/d\ll \eps$.
\claimproofend

Relabel the path forests $\mathcal{P}'_{j,k}$, $j\in [s/2]$ and $k\in [d']$, as $\mathcal{P}_i$, $i\in [sd'/2]$. Noting that $sd'/2=(1-\eta')d/2\leq d/2$, let $\mathcal{P}_i=\emptyset$ for each $sd'/2<i\leq d/2$. The number of edges in $\mathcal{P}_i$, $i\in [d/2]$, is, by Claim~\ref{nottoomanymissing} and as $1/s,\eta'\ll \eps$, at least
\begin{align*}
\frac{sd'}{2}\cdot (s-1)\cdot \frac{(1-\eta')n}{s}-\frac{sd'}{2}\cdot \frac{\eps n}3
&\geq \frac{sd'}{2}\left(\left(1-\frac{\eps}{3}\right)n-\frac{\eps n}{3}\right) = \frac{(1-\eta')d}{2}\cdot \left(1-\frac{2\eps}{3}\right)n\geq \frac{n(1+\eta)d}{2}-\eps nd \\
&\geq |E(G)|-\eps nd.
\end{align*}
Thus, it is only left to show that the path forests $\mathcal{P}_i$, $i\in [d/2]$, form an  $(n/d^{1/8},d^{7/8},d^{7/8})$-bounded collection.

For each $i\in [d/2]$, if $\mathcal{P}_i\neq \emptyset$, then there are some distinct $j,k\in V(K_s)$ such that all the endvertices of $\mathcal{P}_i$ are in $A_j\cup A_k$. Thus, by \ref{prop:degreeintoAi}, each vertex in $G$ has at most $4d/s\leq d^{7/8}$ neighbours among the endvertices of $\mathcal{P}_i$ and, by \ref{prop:Aisize}, $\mathcal{P}_i$ contains at most $2n/s\leq n/d^{1/8}$ paths.
Furthermore, for each $v\in V(G)$, if $j,k$ are such that $v\in A_j$ and $j$ is an endvertex of $Q_k$, then the only paths in $\mathcal{P}_i$, $i\in [d/2]$, that can end in $v$ are those from $\mathcal{P}_{k,k'}$, $k'\in [d']$, so that $v$ is the endvertex of at most $d'\leq d/s\leq  d^{7/8}$ paths in $\mathcal{P}_i$, $i\in [d/2]$. Therefore, the path forests $\mathcal{P}_i$, $i\in [d/2]$, form an  $(n/d^{1/8},d^{7/8},d^{7/8})$-bounded collection, as required.
\end{proof}


\subsection{Improving a bounded collection of path forests}\label{sec:betterpathforest}
Having set aside a set $Y$ (which fulfils a similar `sample set' function as $X$ in the sketch in Section~\ref{sec:proofsketch}) in a $d$-regular graph $G$, we can then improve a collection of path forests in $G-Y$ by taking each path forest in turn and iteratively connecting paths in any path forest using a path of length two whose middle vertex comes from $Y$ (for practical reasons, we actually write this in the proof through taking some maximal set $I$). In order to do this efficiently, we use the local lemma to partition $Y$ into ten parts and find as many such paths using a middle vertex from these sets, exhausting each set in turn. Furthermore,  for all but only a few of the remaining endpoints in the resulting path forest, we attach an extra edge randomly from this endpoint to the last set in the partition of $Y$, using the local lemma with regard to all these choices to make sure that no vertex has too many neighbours among these new endpoints. This last step is critical in ensuring that (once a few paths are removed from the resulting path forests) the endpoints of the path forests will be well spread.

\begin{lemma}\label{Lemma_all_first_connecting_paths_newnew}
    Let $1/d\ll \gamma\polysmall p\ll\eps$. Let $G$ be a $n$-vertex graph with $\Delta(G)\leq d$.
    Let $Y\subset V(G)$ be such that $|Y|\leq (1+\gamma)pn$ and, for each $v\in V(G)$, $d_G(v,Y)=(1\pm \gamma)pd$. Let $\mathcal{P}_1,\ldots, \mathcal{P}_{d/2}$ be edge-disjoint path forests in $G-Y$ which are $(n/d^{1/9},d^{8/9},d^{8/9})$-bounded.

    Then, $G$ contains an edge-disjoint collection of path forests  $\mathcal{P}'_1,\ldots, \mathcal{P}'_{d/2}$ such that at most $\eps n/d$ paths can be removed from each $\mathcal{P}_i'$, $i\in [d/2]$, to get altogether  a $((1+\eps)n/d,d^{1/4},d^{1/4})$-bounded edge-disjoint collection of path forests, and, for each $i\in [d/2]$, $E(\mathcal{P}_i)\subset E(\mathcal{P}'_i)$.
\end{lemma}
\begin{proof} Let $\Delta_0=d^{1/4}$ and $\Delta_1=d^{1/4}$.
Take a maximal set $I\subset [d/2]$ for which there are edge-disjoint collections $\mathcal{Q}_i$, $\mathcal{R}_i$, $\mathcal{P}_i'$, $i\in I$, of path forests in $G$, and, for each $i\in I$, a path forest $\mathcal{P}_i''$ consisting of some of the paths in $\mathcal{P}'_i$, such that, for each $i\in I$,
\stepcounter{propcounter}
\begin{enumerate}[label = {{{\textbf{\Alph{propcounter}\arabic{enumi}}}}}]
\item\label{newnew1} each path in $\mathcal{Q}_i$ has length 2 with endvertices among the endvertices of $\mathcal{P}_i$ and its middle vertex in $Y$,
\item\label{newnew2} each path in $\mathcal{R}_i$ has length 1 with an endvertex among the endvertices of $\mathcal{P}_i$ and an endvertex in $Y$,
\item\label{newnew2a} the paths in $\mathcal{Q}_i$ are vertex disjoint from the paths in $\mathcal{R}_i$,
\item\label{newnew3} $\mathcal{P}_i$, $\mathcal{Q}_i$, $\mathcal{R}_i$ combine to form $\mathcal{P}_i'=\mathcal{P}_i+\mathcal{Q}_i+\mathcal{R}_i$,
\item\label{newnew4} $\mathcal{P}_i''$ contains all but at most of the $\eps n/d$ paths in $\mathcal{P}_i'$,
\end{enumerate}
and
the collection $\mathcal{P}''_i$, $i\in I$, is $((1+\eps)n/d,\Delta_0,\Delta_1)$-bounded.

Take such path forests $\mathcal{Q}_i,\mathcal{R}_i,\mathcal{P}_i',\mathcal{P}_i''$, $i\in I$.
Note that if $I=[d/2]$, then taking $\mathcal{P}'_i$, $i\in [d/2]$, gives us a collection of edge-disjoint path forests in $G$ which satisfies the conditions in the lemma (as seen by considering $\mathcal{P}''_i$, $i\in [d/2]$). Therefore, assume for contradiction that $I\neq [d/2]$.

Let $G'$ be the graph of all of the edges of $G$ which are not in any of the paths in $\mathcal{Q}_i$ or $\mathcal{R}_i$ for any $i\in I$. As the collection $\mathcal{P}_i$, $i\in [d/2]$, is $(n/d^{1/9},d^{8/9},d^{8/9})$-bounded, each vertex $v\in V(G)\setminus Y$ appears as the endpoint of at most $d^{8/9}$ paths in total among the paths of $\mathcal{P}_i$, $i\in [d/2]$, and therefore $v$ has at most $d^{8/9}$ neighbouring edges in total in the path forests $\mathcal{Q}_i$ and $\mathcal{R}_i$, $i\in I$.
Thus, as $d_G(v,Y)\geq (1-\gamma)pd$ and $1/d\ll \gamma,p$, we have $d_{G'}(v,Y)\geq (1-2\gamma)pd$. Take a random partition $Y=Y_1\cup \ldots\cup Y_{10}$ by choosing the location of each $v\in Y$ independently and uniformly at random. Using the local lemma, we can assume that, for each $v\in V(G)\setminus Y$ and $j\in [10]$, $d_{G'}(v,Y_j)\geq (1-3\gamma)pd/10$ and $|Y_j|\leq (1+2\gamma)pn/10$.

Using that $I\neq [d/2]$, pick $i\in [d/2]\setminus I$.
For each $1\leq j\leq 10$ in turn, let $\mathcal{Q}_{i,j}$ be a maximal set of vertex-disjoint paths of length 2 in $G'$ such that each path in $\mathcal{Q}_{i,j}$ has its middle vertex in $Y_j$ and its endvertices among the endvertices of $\mathcal{P}_i+\mathcal{Q}_{i,1}+\ldots+\mathcal{Q}_{i,j-1}$, and $\mathcal{P}_i+\mathcal{Q}_{i,1}+\ldots+\mathcal{Q}_{i,j}$ is a path forest. We will show the following.

\begin{claim}\label{claim:nottoomanypaths}
For each  $0\leq j\leq 10$, $\mathcal{P}_i+\mathcal{Q}_{i,1}+\ldots+\mathcal{Q}_{i,j}$ contains at most $(1+\eps/2)n/d+n/d^{(j+1)/10}$ paths.
\end{claim}
\begin{proof}
We will prove this by induction on $j$, noting that it follows for $j=0$ as $\mathcal{P}_i$ contains at most $n/d^{1/9}\leq n/d^{1/10}$ paths by the $(n/d^{1/9},d^{8/9},d^{8/9})$-boundedness condition. Then, suppose $j\in [10]$, and that $\mathcal{P}_i+\mathcal{Q}_{i,1}+\ldots+\mathcal{Q}_{i,j-1}$ contains at most $(1+\eps/2)n/d+n/d^{j/10}$ paths. Therefore, $\mathcal{Q}_{i,j}$ contains at most $3n/d^{j/10}$ paths, and so, if $Y_j'\subset Y_j$ is the set of vertices in $Y_j$ not appearing in any path in $\mathcal{Q}_{i,j}$, then $|Y_j\setminus Y_j'|\leq 3n/d^{j/10}$.

Let $X_j\subset V(G)\setminus Y$ be a maximal set of endvertices of $\mathcal{P}_i+\mathcal{Q}_{i,1}+\ldots+\mathcal{Q}_{i,j}$ which are in $V(G)\setminus Y$ and which are such that no two vertices in $X_j$ are endpoints of the same path in $\mathcal{P}_i+\mathcal{Q}_{i,1}+\ldots+\mathcal{Q}_{i,j}$.
By the maximality of $\mathcal{Q}_{i,j}$, no two vertices in $X_j$ have a common neighbour in $Y_j'$ in $G'$. Thus, the number of edges in $G'$ between $X_j$ and $Y_j'$ is at most $|Y_j'|\leq |Y_j|\leq (1+2\gamma)pn/10$. 

On the other hand, each vertex  $v\in X_j\subset V(G)\setminus Y$ satisfies $d_{G'}(v,Y_j)\geq (1-3\gamma)pd/10$, so 
the number of edges in $G'$ between $X_j$ and $Y_j'$ is (as each vertex in $Y_j\setminus Y_j'$ has at most $d^{8/9}$ neighbours among the endvertices of $\mathcal{P}_i$ by the $(n/d^{1/9},d^{8/9},d^{8/9})$-boundedness of the $\mathcal{P}_{k}$, $k\in [d/2]$) at least
\[
|X_j|\cdot (1-3\gamma)pd/10-|Y_j\setminus Y_j'|\cdot d^{8/9}\geq |X_j|\cdot (1-3\gamma)pd/10-d^{8/9}\cdot 3n/d^{j/10},
\]
so that, as $1/d\ll p$ and $\gamma\ll \eps$,
\[
|X_j|\leq \frac{(1+2\gamma)pn/10+d^{8/9}\cdot 3n/d^{j/10}}{(1-3\gamma)pd/10}\leq (1+6\gamma)n/d+n/d^{(j+1)/10}\leq (1+\eps/2)n/d+n/d^{(j+1)/10}.
\]
Then, by the definition of $X_j$, we have that $\mathcal{P}_i+\mathcal{Q}_{i,1}+\ldots+\mathcal{Q}_{i,j}$ contains at most
$(1+\eps/2)n/d+n/d^{(j+1)/10}$ paths, as required.
\claimproofend

Let $\mathcal{Q}_i=\mathcal{Q}_{i,1}+\ldots+\mathcal{Q}_{i,10}$ so that, by the claim with $j=10$, $\mathcal{P}_i+\mathcal{Q}_i$ contains at most $(1+\eps)n/d$ paths as $1/d\ll \eps$. Let $Z_i$ be the set of vertices in $Y_{10}$ which appear in $\mathcal{Q}_{i,10}$, so that, as $\mathcal{P}_i+\mathcal{Q}_{i,1}+\ldots+\mathcal{Q}_{i,9}$ contains at most $(1+\eps/2)n/d+n/d$ paths by Claim~\ref{claim:nottoomanypaths} with $j=9$, we have $|Z_i|\leq 3n/d$. Let $Z\subset Y$ be the set of vertices in $Y$ which are the endvertices of more than $\Delta_0/2=d^{1/4}/2$ of the paths in $\mathcal{P}''_j$, $j\in I$. Then, as $\mathcal{P}''_j$, $j\in I$, is  $((1+\eps)n/d,\Delta_0,\Delta_1)$-bounded and hence contains at most $|I|\cdot (1+\eps)n/d$ paths in total, we have, as $|I|\leq d/2$,
\begin{equation}\label{eqn:Zuppbound}
|Z|\leq \frac{2\cdot |I|\cdot (1+\eps)n/d}{\Delta_0/2}\leq \frac{4n}{\Delta_0}. 
\end{equation}
Let $A_i$ be the set of endvertices of the paths in $\mathcal{P}_i+\mathcal{Q}_i$ for which both endvertices have at least $pd/30$ neighbours in $Y_{10}\setminus (Z\cup Z_i)$ in $G'$. Note that, by the maximality of $\mathcal{Q}_{i,10}$, the sets $N_{G'}(u,Y_{10})\setminus Z_i$ and $N_{G'}(u',Y_{10})\setminus Z_i$ are disjoint for $u,u'\in A_i$ unless $u$ and $u'$ are endpoints of the same path in $\mathcal{P}_i+\mathcal{Q}_i$.
Note then that $A_i$ partitions naturally into a set of unordered pairs $\mathcal{F}_i$ where if $\{u,u'\}\in \mathcal{F}_i$ then $u$ and $u'$ are endpoints of the same path in $\mathcal{P}_i+\mathcal{Q}_i$. For each $\{u,u'\}\in \mathcal{F}_i$, pick a vertex $v_u$ from $N_{G'}(u,Y_{10}\setminus (Z\cup Z_i))$ and a vertex $v_{u'}\in N_{G'}(u',Y_{10}\setminus(Z\cup Z_i))$ uniformly and independently at random subject to $v_u\neq v_{u'}$ and add $uv_u$ and $u'v_{u'}$ to $\mathcal{R}_i$. 

For each $v\in V(G)$, let $B_v$ be the event that $v$ has more than $\Delta_1=d^{1/4}$ neighbours among the set $\{v_u,v_{u'}:\{u,u'\}\in \mathcal{F}_i\}$.
For each $w\in N_G(v,Y_{10})$, there is at most one pair $\{u,u'\}\in \mathcal{F}_i$ for which $w\in N_{G'}(u)$(by the maximality again of $\mathcal{Q}_{i,10}$). Therefore, as $u,u'\in A_i$ here, $\mathbb{P}(w\in \{v_u,v_{u'}\})\leq 2/(pd/30)=60/pd$. Therefore, by Chernoff's bound, $\mathbb{P}(B_v)\leq e^{-\sqrt{d}}$ as $1/d\ll p$. Note that, for each $\{u,u'\}\in \mathcal{F}_i$, the choice of $v_u$ and $v_{u'}$ can affect $B_v$ only if $v$ is within a distance 2 of $u$ or $u'$ in $G$. Therefore, by the local lemma, we can assume that each vertex in $V(G)$ has at most $\Delta_1=d^{1/4}$ neighbours among the set $\{v_u,v_{u'}:\{u,u'\}\in \mathcal{F}_i\}$. Let $\mathcal{P}_i'=\mathcal{P}_i+\mathcal{Q}_i+\mathcal{R}_i$, and let $\mathcal{P}_i''$ be the paths in $\mathcal{P}_i'$ with both endvertices in $Y_{10}$. Then, by the choice from the local lemma, we have that every vertex in $V(G)$ has at most $\Delta_1$ neighbours among the endvertices of $\mathcal{P}_i''$. 
As, furthermore, none of the endvertices of the paths in $\mathcal{P}_i''$ are in $Z$, and there are at most $(1+\eps)n/d$ paths in $\mathcal{P}_i'$ and therefore $\mathcal{P}_i''$, we have that $\mathcal{P}''_j$, $j\in I\cup \{i\}$, is $((1+\eps)n/d,\Delta_0,\Delta_1)$-bounded, while \ref{newnew1}--\ref{newnew3} hold for each $j\in I\cup \{i\}$ by construction. Therefore, if we can show \ref{newnew4} for $i$ we will get a contradiction to the choice of $I$, and then we can conclude that $I=[d/2]$, which, as noted above, finishes the proof.

To prove \ref{newnew4}, we simply need to show that at most $\eps n/d$ of the paths in $\mathcal{P}_i'$ have an endpoint which is not in $A_i$. Take a maximal set $X$ of endvertices of the paths in $\mathcal{P}_i$ which come from different paths, such that each vertex in $X$ has at most $pd/30$ neighbours in $Y_{10}\setminus (Z\cup Z_i)$. Recall that each vertex in $V(G)\setminus Y$ has at least $(1-3\gamma)pd/10$ neighbours in $Y_{10}$ in $G'$. Thus, each vertex in $X$ has at least $pd/20$ neighbours in $Z_i\cup Z$ in $G'$.

Now, as the vertices in $X$ are endpoints of different paths in $\mathcal{P}_i+\mathcal{Q}_i$, we have concluded that they share no neighbours in $Y_{10}\setminus Z_i$, and therefore each vertex in $Z\setminus Z_i$ has at most 1 neighbour in $G'$, so there are at most $|Z|\leq 4n/\Delta_0$ edges between $X$ and $Z\setminus Z_0$ in $G'$ by \eqref{eqn:Zuppbound}.
Furthermore, we deduced that $|Z_i|\leq 3n/d$ and so, as $\mathcal{P}_j$, $j\in [d/2]$, is $(n/d^{1/9},d^{8/9},d^{8/9})$-bounded, and every vertex in $X$ is an endvertex of some path in $\mathcal{P}_i$, there are at most $|Z_i|\cdot d^{8/9}\leq 3n/d^{1/9}$ edges from $X$ to $Z_i$. Therefore, as $1/d\ll p,\eps$,
\[
|X|\leq \frac{(4n/\Delta_0)+(3n/d^{1/9})}{pd/30}\leq \frac{\eps n}{d},
\]
as required. Thus, by the definition of $X$, there are at most $\eps n/d$ paths in $\mathcal{P}_i'$ which are not in $\mathcal{P}_i''$. This completes the proof of \ref{newnew4}, and hence the lemma.
\end{proof}


\subsection{Proof of Theorem~\ref{thm:pathcovers}}\label{sec:pathcovers}

Lemma~\ref{Lemma_decomposition_path_forests_regular} and Lemma~\ref{Lemma_all_first_connecting_paths_newnew} already allow us to prove Theorem~\ref{thm:pathcovers}.

\begin{proof}[Proof of Theorem~\ref{thm:pathcovers}]
Let $d_0,\gamma$ and $p$ be such that $1/d_0\ll \gamma\ll  p\ll \eps$ and let $d\geq d_0$. Let $G$ be a $d$-regular $n$-vertex graph. Let $Y\subset V(G)$ be a random set chosen by including each vertex independently at random with probability $p$. Using the local lemma, and that $G$ is $d$-regular and $1/d\ll \gamma,p$, we can assume that $d_G(v,Y)=(1\pm \gamma)pd$ for each $v\in V(G)$ and $|Y|=(1\pm \gamma)pn$. Note that $G-Y$ has vertex degrees in $(1\pm \gamma)(1-p)d$.
Therefore, by Lemma~\ref{Lemma_decomposition_path_forests_regular}, and as $1/d\ll p$, $G-Y$ contains an $(n/d^{1/9},d^{8/9},d^{8/9})$-bounded edge-disjoint collection $\mathcal{P}_1,\ldots, \mathcal{P}_{d/2}$ of path forests which cover all but at most $\eps nd/8$ of the edges of $G-Y$. As there are at most $n\cdot (1+\gamma)pd\leq \eps nd/8$ edges with a vertex in $Y$, these path forests contain all but at most $\eps nd/4$ of the edges of $G$. Then, applying Lemma~\ref{Lemma_all_first_connecting_paths_newnew} we get an edge-disjoint collection $\mathcal{P}'_1,\ldots, \mathcal{P}'_{d/2}$ of path forests which cover all but at most $\eps nd/4$ edges of $G$
such that (simplifying the conclusion) for each $i\in [d/2]$, $\mathcal{P}_i'$ has at most $(1+\eps/4)n/d$ paths.

Taking $j\in [d/2]$ to maximise $|E(\mathcal{P}'_j)|$, we have that $|E(\mathcal{P}'_j)|\geq (nd/2-\eps nd/4)/(d/2)=(1-\eps/2)n$. 
Letting $k$ be the number of paths in $\mathcal{P}'_j$, we have that $k\leq (1+\eps/4)n/d$, and note that we are done if $k\leq n/(d+1)$, so we can assume otherwise. Then, observe that the total number of edges in the $\lfloor n/(d+1)\rfloor$ longest paths in $\mathcal{P}'_j$ is at least
\[
(1-\eps/2)n\cdot \frac{\lfloor n/(d+1)\rfloor}{k}\geq (1-\eps/2)n\cdot \frac{n/(d+1)-1}{(1+\eps/4)n/d}\geq (1-\eps/2)^2n\geq (1-\eps)n.
\]
Thus, $\mathcal{P}'_j$ contains a set of at most $n/(d+1)$ paths which cover all but at most $\eps n$ vertices of $G$, as required.
\end{proof}

\subsection{Almost-spanning nearly-regular subgraphs}\label{sec:nearregular}
The purpose of this section is to prove the following lemma, Lemma~\ref{lem_regularise_nearly_spanning_defective_vertices}. As discussed in Section~\ref{sec:proofsketch}, the removal of dense spots can significantly weaken the regularity of the graph we are decomposing. Given a graph with such weaker conditions, the following lemma partitions most of the edges into a few subgraphs which are all very close to being regular, so that we can later apply Lemma~\ref{Lemma_decomposition_path_forests_regular} to each of them.
 
\begin{lemma}\label{lem_regularise_nearly_spanning_defective_vertices}
     Let $1/d \ll \mu \ll \eps\leq 1$ and $\eta=d^{-0.9}$. Let $G$ be an $n$-vertex graph with $\Delta(G)\leq d$. Then, there is some $k\in \N$ such that there are edge-disjoint subgraphs $G_1,\ldots,G_k$ in $G$ and some even $d_1,\ldots,d_k\geq \mu d$ such that all but at most $\eps nd$ edges of $G$ are in $\bigcup_{i\in [k]}G_i$, $\sum_{i\in [k]}d_i\leq (1+\eps)d$, and, for each $i\in [k]$,  $G_i$ is $(1\pm \eta)d_i$-regular. 
\end{lemma}



To prove Lemma~\ref{lem_regularise_nearly_spanning_defective_vertices}, we will use the following two lemmas, which follow using very recent methods of Chakraborti, Janzer, Methuku and Montgomery~\cite{regularising,edgedisjointcycles}. For completion, we prove the lemmas in an appendix. Each lemma takes a graph which is roughly regular (with degrees varying by at most a constant multiple $C$, and a multiple $(1+\gamma)$, respectively) and finds within it a subgraph whose vertex degrees only differ by at most the addition of a constant multiple of the logarithm of the average degree and  which has average degree comparable to the original graph. In the second lemma, the average degree of the subgraph is particularly close to the original average degree and the subgraph contains almost all of the original vertices.

\begin{restatable}{lemma}{regularisingtwo}\label{thm:nearreg}
    Let $1/C'\polysmall 1/C\leq 1$. For any $d$, if a graph has degrees between $d$ and $Cd$, then it contains a subgraph with degrees between $d'$ and $d'+C'\log d'$ for some $d'\geq d/C'$.
\end{restatable}
\begin{restatable}{lemma}{regularisingone}\label{lem:allverticesregularspanningnearregular}
    There exists some $C>0$ such that for each $1/d\ll \gamma\leq 1/100$ the following holds. Let $G$ be a graph with $d \leq \delta(G), \Delta(G) \leq (1+\gamma)d$. Then, for some $d'\geq (1-40\gamma)d$, $G$ contains a subgraph $G'$ with $|V(G')|\geq (1-40\gamma)|V(G)|$ and $d'\leq \delta(G'), \Delta(G') \leq d'+C\log d'$.
\end{restatable}

To prove Lemma~\ref{lem_regularise_nearly_spanning_defective_vertices}, we start by taking a maximal collection $H_1,\ldots,H_t$ of nearly-regular subgraphs of $G$ whose average degree is not too far below that of $G$. It will follow easily from Lemma~\ref{thm:nearreg} that they will cover most of the edges of $G$, however, these subgraphs may contain significantly fewer vertices than $G$ (and so, in particular, the sum of the average degree of the subgraphs may exceed $d$). However, each vertex of $G$ will be in very few of the subgraphs $H_i$, so we can randomly choose $k$ (for some particular $k$) almost-spanning subgraphs of $G$ by, for the $j$th subgraph, allocating each vertex to a different subgraph $H_i$ which contains that vertex and including in the $j$th subgraph any edge of any $H_i$ which has both vertices allocated to $H_i$ for the $j$th subgraph. These $k$ subgraphs will not be edge-disjoint. In fact, using the local lemma, we will choose such subgraphs so that each edge in each $H_i$ will appear in many of the $k$ subgraphs; for each of these edges we will then randomly allocate it to one such subgraph and delete it from all the others. By carefully choosing all the probabilities involved, and using again the local lemma, we will be able to get subgraphs $G_i'$, $i\in [k]$, which contain almost all of the vertices of $G$ and which are nearly regular. Finally, as these subgraphs will not be quite regular enough (as the application of the local lemma cannot preserve as much of the near-regularity as we would like), we will boost the near-regularity of each $G_i'$ using Lemma~\ref{lem:allverticesregularspanningnearregular} to get a more regular subgraph $G_i\subset G_i'$ with only the loss of relatively few edges or vertices.

\begin{proof}[Proof of Lemma~\ref{lem_regularise_nearly_spanning_defective_vertices}] Let $\gamma$ and $\beta$ be such that $1/d\ll \gamma\ll \mu \ll \beta\polysmall \eps$ and note that we can assume $\eps\ll 1$. 
Let $H_1,\ldots, H_t$ be a maximal edge-disjoint collection of subgraphs of $G$ such that, for each $i\in [t]$, there is some $d_i'\geq \beta d$ such that $H_i$ is $(1\pm \gamma)d_i'$-regular. Let $H=\bigcup_{i\in [t]}H_i$.

    \begin{claim} We have $|E(G)\setminus E(H)|\leq \eps n d/3$.\label{clm:manysubgraphs}
    \end{claim}
    \begin{proof}
        Suppose, for contradiction, that $|E(G)\setminus E(H)|> \eps n d/3$. Then, by the well-known folklore result, $G\setminus H$ has a subgraph ${G}'$ with minimum degree at least $\eps d/3$. Using that $\Delta(G)\leq d$ and $\beta\ll \eps$, apply Lemma~\ref{thm:nearreg} with $C=3/\eps$ to find a subgraph $H_{t+1}$ of ${G}'$ that is $(1\pm \gamma)d_{t+1}$-regular for some $d_{t+1}\geq \beta d$, contradicting the maximality of the subgraphs $H_i$, $i\in [t]$.
    \end{proof}


 For each $i\in [t]$, let $\beta_i=(1-\gamma)d_i'/d\geq \beta/2$. For each $v\in V(G)$, let $I_v$ be the set of $i\in [t]$ with $v\in V(H_i)$, noting that 
 \[
 d\geq d_G(v)\geq \sum_{i\in I_v}d_{H_i}(v)\geq \sum_{i\in I_v}\beta_id\geq |I_v|\cdot \beta d/2,
 \] so that $|I_v|\leq 2/\beta$. Take $k=1/2\mu$. For every $j\in [k]$, take disjoint subsets $R_1^{(j)},\ldots, R_t^{(j)}$ of $V(G)$, by determining the location  of each $v\in V(G)$ independently and uniformly at random subject to the following rule.
\begin{itemize}
\item For each $i\in I_v$, there are exactly $\lfloor \beta_i k\rfloor$ values of $j\in [k]$ for which $v\in R_i^{(j)}$.
\end{itemize}

To see that this is possible, note that, for each $v\in V(G)$, $\sum_{i\in I_v}\beta_id\leq \sum_{i\in I_v}d_{H_i}(v)\leq d$, so that $\sum_{i\in I_v}\beta_i\leq 1$.

For each $i\in [t]$, let $H'_i\subset H_i$ be the subgraph of edges $e\in E(H_i)$ such that the number of $j\in [k]$ with $V(e)\subset R_i^{(j)}$ is $(1\pm \beta)\beta_i^2k$. Let $H'=\bigcup_{i\in [t]}H_i'$.

\begin{claim}\label{clm:HHprimeclose}
    With positive probability, $|E(H)\setminus E(H')|\leq \eps n d/6$ and, for each $i\in [t]$ and $j\in [k]$, if $R_i^{(j)}\neq \emptyset$, then $H'_i[R_i^{(j)}]$ is $(1\pm \beta)\beta_i^2d$-regular.
\end{claim}
\begin{proof}
For each $v\in V(G)$ and $i\in I_v$, let $B_{v,i}$ be the event that at least $\eps d_i'/3$ edges incident to $v$ in $H_i$ are not in $H'_i$. Let $v\in V(G)$ and $i\in I_v$, and suppose we know which sets $R_i^{(j)}$, $j\in [k]$, contain $v$.
Then, for each $u\in N_{H_i}(v)$ the probability that $uv\in E(H')$, or, equivalently, that $u$ is not in $(1\pm \beta)\beta_i^2k$ of the $\lfloor \beta_i k\rfloor$ sets $R_i^{(j)}$, $j\in [k]$ which contain $v$, is, by Chernoff's bound applied to the appropriate hypergeometric variable, at most $\exp(-\beta^2\beta_i^2k/12)\leq \eps/6$, where we have used that $1/k\ll \beta,\eps$ and $\beta_i\geq \beta/2$. Furthermore, these events are independent over each $u\in N_{G_i}(v)$. Thus, by Chernoff's bound, $\mathbb{P}(B_v)\leq 2e^{-\eps^2(1-\gamma)d_i'/10^3}\leq e^{-\sqrt{d}}$, as $d_i'\geq \beta d$ and $1/d\ll \beta,\eps$.

For each $v\in V(G)$, $i\in I_v$ and $j\in [k]$, let $B_v^{i,j}$ be the event that $v\in R_i^{(j)}$ but $d_{H'_i}(v,R_i^{(j)})\neq (1\pm \beta)\beta_i^2d$. 
Let $v\in V(G)$, $i\in I_v$ and $j\in [k]$ and suppose we know which sets $v$ lies in, and that one of these sets is $R_i^{(j)}$.
Then, for each $u\in N_{H_i}(v)$ and $j\in [k]$ the probability that $u\in R_i^{(j)}$ is $\lfloor \beta_ik\rfloor/k =(1\pm \gamma)\beta_i$, and, when this happens, the probability that, of the $\lfloor \beta_i k\rfloor-1$ other sets $R_i^{(j)}$ which contain $v$, $u$ is not in $(1\pm \gamma)(\beta_i^2k-1)$ of them is, by Chernoff's bound applied to the appropriate hypergeometric variable, at most $\exp(-\gamma^2\beta_i^2k/12)\leq \gamma$. Thus, for each $u\in N_{H_i}(v)$ and $j\in [k]$ with $v\in R_i^{(j)}$, the probability that $u\in R_i^{(j)}$ and $u\in N_{H_i'}(v)$ is $(1\pm 2\gamma)\beta_i$.
Note furthermore, that these events are independent over each $u\in N_{G_i}(v)$ and $d_{H_i}(v)=(1\pm \gamma)d_i'=(1\pm 3\gamma)\beta_id$. Thus, $\mathbb{P}(B_v^{i,j})\leq e^{-\sqrt{d}}$ by Chernoff's bound.

Note that any bad event we have defined depends on the placement of a vertex $w$ only if the vertex $v$ in the subscript is within a distance 1 of each other in $G$, where $\Delta(G)\leq d$, and the placement of each vertex $w$ therefore affects at most $(d+1)\cdot (2/\beta)\cdot k \leq d^2$ events (where we have used that, for each $v\in V(G)$, $|I_v|\leq 2/\beta$, and $1/d\ll \mu,\beta$). Thus, by the local lemma, with positive probability none of the events $B_{v,i}$, $v\in V(G)$ and $i\in I_v$, or $B_v^{i,j}$, $v\in V(G)$, $i\in I_v$ and $j\in [k]$ hold. If none of the events $B_v^{i,j}$, $v\in V(G)$, $i\in I_v$ and $j\in [k]$ hold, then, for each $i\in [t]$ and $j\in [k]$, if $R_i^{(j)}\neq \emptyset$, then $H'_i[R_i^{(j)}]$ is $(1\pm \beta)\beta_i^2d$-regular. We will show that if none of the events $B_{v,i}$, $v\in V(G)$ and $i\in I_v$, hold, then $|E(H)\setminus E(H')|\leq \eps nd/6$, which completes the proof of the claim.

For this, note that, if none of the events $B_{v,i}$, $v\in V(G)$ and $i\in I_v$, hold, then
\begin{align*}
|E(H)\setminus E(H')|&= \frac{1}{2}\sum_{v\in V(G)}(d_H(v)-d_{H'}(v))=\frac{1}{2}\sum_{v\in V(G)}\sum_{i\in I_v}(d_{H_i}(v)-d_{H_i'}(v))
\leq \frac{1}{2}\sum_{v\in V(G)}\sum_{i\in I_v}\frac{\eps d_i'}{3}\\
&\leq \frac{1}{2}\sum_{v\in V(G)}\sum_{i\in I_v}\frac{\eps d_{H_i}(v)}{3}
\leq \frac{1}{2}\sum_{v\in V(G)}\frac{\eps d}{3}=\frac{\eps nd}{6}.
\end{align*}
Thus, with positive probability, the properties in the claim hold.
\claimproofend

    Fix a choice of $R_i^{(j)}$, $i\in [t]$ and $j\in [k]$, satisfying the properties in the claim. For each $i\in [t]$ and $e\in E(H'_i)$, independently and uniformly at random assign $e$ to some $j\in [t]$ with $V(e)\subset R_i^{(j)}$ and let $f(e)=j$. Then, for each $j\in [k]$, let $G'_j$ be the graph with vertex set $\bigcup_{i\in [t]}R_i^{(j)}$ and edge set $\bigcup_{i\in [t]}\{e\in E(H'_i):f(e)=j\}$. Let $d'=d/k$ and note that the graphs $G'_j$, $j\in [k]$, are, by design, edge disjoint.

\begin{claim}
    With positive probability, for each $j\in [k]$, $G_j'$ is $(1\pm 4\beta)d'$-regular.
\end{claim}
\begin{proof} Let $v\in V(G)$, $i\in I_v$ and $j\in [k]$ with $v\in R_i^{(j)}$. Recall that the degree of $v$ in $H_i'[R_i^{(j)}]$ is $(1\pm \beta)\beta_i^2d$, and each edge $e$ in $H_i'$ has $(1\pm \beta)\beta_i^2k$ indices $j'\in [k]$ with $V(e)\subset R_i^{(j')}$. Therefore, the expected degree of $v$ in $G_j'$ is at least $(1- \beta)\beta_i^2d/(1+\beta)\beta_i^2k$ and at most $(1+\beta)\beta_i^2d/(1-\beta)\beta_i^2k$, so that 
\[
\E(d_{G_j'}(v))=(1\pm 3\beta)d/k=(1\pm 3\beta)d'.
\]
Thus, as $1/d\ll \mu,\beta$ and $k=1/2\mu$, using Chernoff's bound, for each $v\in V(G)$, $j\in [k]$, we have with probability at most $e^{-\sqrt{d}}$ that if $v\in V(G_j')$ then $d_{G_j'}(v)\neq(1\pm 4\beta) d'$. For each edge $e$, which graph $G'_j$, $j\in [k]$, is in affects the degrees of at most $2d-1$ vertices in each of the $k$ graphs, and each degree is affected by the location of at most $d$ edges, and therefore an application of the local lemma implies the claim.
\claimproofend

Using this claim, we can thus partition $H'\subset G$ into edge-disjoint graphs $G_i$, $i\in [k]$, in $G$ such that, for each $i\in [k]$, $G_i'$ is $(1\pm 4\beta)d'$-regular. For each $i\in [k]$, by Lemma~\ref{lem:allverticesregularspanningnearregular} (with $\gamma=10\beta$ and $d=1-4\beta$) we can find some $G_i\subset G'_i$ and even $d_i$ with $(1-400\beta)(1-4\beta)d'\leq d_i\leq (1+5\beta)d'$ such that $|V(G_i)|\geq (1-400\beta)|V(G'_i)|$ is $(1\pm \eta)d_i$-regular (where we have used that $\eta=d^{-0.9}$ and $1/d\ll \beta,1/k$). Note that, 
\begin{align*}
\sum_{i\in [k]}|E(G'_i)\setminus E(G_i)|&\leq \sum_{i\in [k]}((1+4\beta)d'\cdot |V(G_i)\setminus V(G'_i)|+n\cdot ((1+4\beta)-(1-400\beta)(1-4\beta))d'\\
&\leq k\cdot (1+4\beta)d'\cdot 400\beta n+k\cdot n\cdot 410\beta d'\leq 1000\beta dn\leq \eps n/3,
\end{align*}
as $\beta\ll \eps$ and $d'=d/k$.
Thus, combined with the properties from Claim~\ref{clm:manysubgraphs} and Claim~\ref{clm:HHprimeclose}, we have that the graphs $G_1,\ldots,G_k$ cover all but at most $2\eps nd/3+\eps nd/6\leq \eps nd$ edges of $G$, so that, as $\mu\ll \beta$, they satisfy the conditions of the lemma as
\[
\sum_{i\in [k]}d_i\leq (1+5\beta)\sum_{i\in [k]}d'=(1+5\beta)d\leq (1+\eps)d.\qedhere
\]
\end{proof}


\subsection{Proof of Lemma~\ref{Lemma_decomposition_path_forests}}\label{sec:prooflemmadecomppathforest}

Given Lemmas~\ref{Lemma_decomposition_path_forests_regular},~\ref{Lemma_all_first_connecting_paths_newnew} and~\ref{lem_regularise_nearly_spanning_defective_vertices} it is now short work to prove Lemma~\ref{Lemma_decomposition_path_forests}.
\begin{proof}[Proof of Lemma~\ref{Lemma_decomposition_path_forests}] Recall that we have $1/d\ll \gamma \ll \eps\leq 1$ and $G$ is an $n$-vertex graph with maximum degree at most $d$ in which all but at most $\gamma n$ vertices have degree at least $(1-\gamma)d$. Let $p$ satisfy $\gamma\ll p\ll \eps$. 
Let $Y\subset V(G)$ be a random set chosen by including each vertex independently at random with probability $p$. Using the local lemma, we can assume that, for every $v\in V(G)$ with $d_G(v)\geq (1-\gamma)d$, we have $d_G(v,Y)=(1\pm 2\gamma)pd$, and that $|Y|\leq (1+\gamma)pn$. 

Let $X$ be the set of vertices $v\in V(G)\setminus Y$ with $d_G(v,Y)=(1\pm 2\gamma)pd$, so that $|X|\geq n-(1+\gamma)pn-\gamma n$, and let $G'=G[X]$. Note that
\[
|E(G)\setminus E(G')|\leq ((1+\gamma)pn+\gamma n)\cdot d\leq \eps nd/4.
\]
Let $\eta=d^{-0.9}$. By Lemma~\ref{lem_regularise_nearly_spanning_defective_vertices}, there is some $k\in \mathbb{N}$ such that there are edge-disjoint subgraphs $G_1,\ldots,G_k$ in $G'$ and some even $d_1,\ldots,d_k\geq \gamma d$ such that all but at most $\eps nd/4$ edges of $G'$ are in $\bigcup_{i\in [k]}G_i$, $\sum_{i\in [k]}d_i\leq (1+\eps/10)d$ and, for each $i\in [k]$, $G_i$ is $(1\pm \eta)d_i$-regular. Note that $k\leq 2/\gamma$.
For each $i\in [k]$, using Lemma~\ref{Lemma_decomposition_path_forests_regular}, find in $G_i$ an $(n/d_i^{1/8},d_i^{7/8},d_i^{7/8})$-bounded edge-disjoint collection $\mathcal{P}^i_{1},\ldots, \mathcal{P}^i_{d_i/2}$ of path forests which cover all but at most $\eps nd_i/3$ edges of $G$. If $(\sum_{i\in [k]}d_i)/2< d/2$, then relabel these path forests as $\mathcal{P}_j$, $1\leq j\leq (\sum_{i\in [k]}d_i)/2$ and, for each $(\sum_{i\in [k]}d_i)/2\leq i\leq d/2$, let $\mathcal{P}_i$ be an empty path forest. If $(\sum_{i\in [k]}d_i)/2\geq d/2$, then let $\mathcal{P}_j$, $j\in [d/2]$, be path forests from $\mathcal{P}_j^i$, $j\in [d_i/2]$ and $i\in [k]$, which maximise the total number of edges in $\mathcal{P}_j$, $j\in [d/2]$. Note that, as $\sum_{i\in [k]}d_i\leq (1+\eps/10)d$, these path forests contain all but at most $\eps nd/4$ of the edges in $\mathcal{P}_j^i$, $j\in [d_i/2]$ and $i\in [k]$.

Note that, as $k\leq 2/\gamma$ and $1/d\ll \gamma$, the collection of path forests $\mathcal{P}_i$, $i\in [d/2]$, is $(n/d^{1/9},d^{8/9},d^{8/9})$-bounded, and contains all but at most $\eps n\cdot\sum_{i\in [t]}d_i/4\leq \eps nd/4$ of the edges of $\mathcal{P}_j^i$, $j\in [d_i/2]$ and $i\in [k]$.
Then, by Lemma~\ref{Lemma_all_first_connecting_paths_newnew}, there is a $((1+\eps)n/d,d^{1/4},d^{1/4})$-bounded collection of edge-disjoint path forests in $G'$ which covers all of the edges of $\mathcal{P}_i$, $i\in [d/2]$. Thus, this collection of paths forests  covers all but at most $\eps nd/4$ of the edges of $\mathcal{P}_j^i$, $j\in [d_i/2]$ and $i\in [k]$, and hence all but at most $\eps nd/2$ of the edges of $\bigcup_{i\in [t]}G_i$, and hence all but at most  $3\eps nd/4$ of the edges of $G'$, and hence all but at most $\eps nd$ of the edges of $G$, as required.
\end{proof}

\section{Connecting to the dense bits}\label{sec:connecttodensebits}

In this section we prove Lemma~\ref{Lemma_X_connects}. In Section~\ref{sec:goodsamples}, we show that there will exist a `good sample' $X$ which samples certain important dense spots appropriately, before using this to connect paths to a maximal set of dense spots through $X$ in Section~\ref{sec:lemmaXproof}.

\subsection{Sampling of dense spots}\label{sec:goodsamples}
In order to facilitate our use of the local lemma, we need some control over the number of dense spots that we consider for our `sampling'. We do this using the following definition.

\begin{definition}
Say a $(\mu,d,K)$-dense graph $H\subset G[X]$ is \textbf{$(k,\eta,d)$-approximable in $G$ by neighbourhoods in $X$} if $G$ contains a set $V$ of at most $k$ vertices, which are pairwise a distance at most $k$ apart in $G$, and such that $|V(H)\triangle \bigcup_{v\in V}N_G(v,X)|\leq \eta d$.
\end{definition}

We now formalise the notion of a good sample. The following definition will be used with $\mathcal{V}$ as the set of vertex sets of the dense spots in the maximal collection $\mathcal{F}$ (see Section~\ref{sec:proofsketch}). The set $X$ should be thought of as a random set (with properties chosen to exist together with strictly positive probability by the local lemma) with vertex density $p$ in $V(G)$, which then has approximately the expected size, as well as approximately the expected intersection with each of the sets in $V\in \mathcal{V}$, and, moreover, similar properties with respect to the neighbourhoods of each vertex in all these subsets (see \ref{good:vertexdegree} and \ref{good:sizeinsubsets}). Finally, and most importantly, if there is an (appropriately approximable) dense spot in $G[X]$, then there must be a larger dense spot $H$ in $G$ such that if $H$ intersects on quite a few vertices in $V\in \mathcal{V}$ then the original smaller dense spot must also intersect that $V$ (this will be used to join the smaller dense spot to one of the dense spots in $\mathcal{F}$ in our application).

\begin{definition}\label{defn:goodsample}
A set $X\subset V(G)$ is a $(d,\gamma,k, \eta, K, p,\mathcal{V})$-\textbf{good-sample in $G$}, where $\mathcal{V}$ is a set of subsets of $V(G)$, if the following hold.
\stepcounter{propcounter}
\begin{enumerate}[label = {{\emph{\textbf{\Alph{propcounter}\arabic{enumi}}}}}]
    \item\labelinthm{good:vertexdegree} $|X|=(1\pm \gamma)p|V(G)|$ and, for any $v\in V(G)$ we have that $d_G(v,X)= (1\pm \gamma)pd$.

    \item\labelinthm{good:sizeinsubsets} For each $V\in \mathcal{V}$, $|V\cap X|=(1\pm \gamma)p|V|$, and, for any $v\in V(G)$ we have that $d_G(v,V\cap X)= p\cdot d_G(v,V)\pm \gamma pd$.
    \item\labelinthm{good:denseinheritance} For any $H\subset G[X]$ which is $(\eta,pd,K)$-dense and $(k,\eta,pd)$-approximable in $G$ by neighbourhoods in $X$, there exists some $H'\subset G$ which is $(6\eta,d,2K)$-dense such that, for each $V\in \mathcal{V}$, if $|V\cap V(H')|\geq d/2$, then $V\cap V(H)\neq \emptyset$.

\end{enumerate}
\end{definition}

Given these definitions, using the local lemma (as discussed in Section~\ref{sec:local}), it is relatively straightforward to show the existence of the good samples that we will need, as follows.

\begin{lemma}\label{Lemma_good_sample_exists}
Let $1/d\ll \gamma,1/k,\eta,1/K,p$.
    Let $G$ be a $d$-regular graph and let $\mathcal{V}$ be a set of vertex-disjoint subsets of $V(G)$ which each have size at least $d/2$ and at most $2Kd$. Then, there is a  $(d,\gamma,k, \eta, K, p,\mathcal{V})$-{good-sample} $X$ in $G$.
\end{lemma}

\begin{proof} Let $t=\lceil|V(G)|/d\rceil$, and let $A_{i}$, $i\in [t]$, be a partition of $V(G)$ into sets which each have size between $d/2$ and $d$.
Let $X\subset V(G)$ be chosen by including each vertex independently at random with probability $p$.
For each $v\in V(G)$, let $B_v$ be the event that $d_G(v,X)\neq (1\pm \gamma)pd$.
For each $V\in\mathcal{V}$, let $B_{V}$ be the event that $|{V}\cap X|\neq (1\pm \gamma)p|{V}|$.
For each $i\in [t]$, let $B_{i}$ be the event that $|A_{i}\cap X|\neq (1\pm \gamma)p|A_{i}|$. For each $V\in \mathcal{V}$ and $v\in V(G)$, let $B_{V,v}$ be the event that $d_G(v,V\cap X)\neq p\cdot d_G(v,V)\pm \gamma pd$.

Let $\mathcal{S}$ consist of all of the subsets of $V(G)$ of size at most $k$ such that these vertices are all pairwise at most $k$ apart in $G$. For each $U\in \mathcal{S}$, let $Y_U=\cup_{u\in U}N_G(u)$ and let $B_U$ be the event that the following do not all hold.
\stepcounter{propcounter}
\begin{enumerate}[label = {{{\textbf{\Alph{propcounter}\arabic{enumi}}}}}]
\item For each $v\in V(G)$, we have $d_G(v,Y_U\cap X)\leq p\cdot d_G(v,Y_U)+\eta pd$.\label{prop:BU1}
\item If $Z_U$ is the set of vertices in $Y_U$ with at least $(1-3\eta)d$ neighbours in $G[Y_U]$, then $|Z_U\cap X|\leq p|Z_U|+\eta p d$.\label{prop:BU2}
\item $|Y_U\cap X|= p|Y_U|\pm \eta pd$.\label{prop:BU3}
\end{enumerate}
For each $U\in \mathcal{S}$ and $V\in \mathcal{V}$ with $Y_U\cap V\neq \emptyset$, let $B_{U,V}$ be the event that the following does not  hold.
\begin{enumerate}[label = {{{\textbf{\Alph{propcounter}\arabic{enumi}}}}}]\addtocounter{enumi}{3}
\item $|(Y_U\cap V)\cap X|\geq p|Y_U\cap V|-\eta pd$.\label{prop:BUV}
\end{enumerate}
Using the local lemma, we can assume that none of the events we have defined holds (see Section~\ref{sec:local}). 

As no event $B_i$, $i\in [t]$, or $B_v$, $v\in V(G)$, holds, we have that \ref{good:vertexdegree} holds. As no event $B_{V}$, $V\in \mathcal{V}$, or $B_{V,v}$, $V\in \mathcal{V}$ and $v\in V(G)$, holds, we have that \ref{good:sizeinsubsets} holds. 
It is left then only to show that \ref{good:denseinheritance} holds.

For this, suppose we have $H\subset G[X]$ which is $(\eta,pd,K)$-dense and $(k,\eta,pd)$-approximable in $G$ by neighbourhoods in $X$. By the approximability of $H$, there exists some $U\in \mathcal{S}$ such that $|V(H)\triangle(Y_U\cap X)|\leq \eta pd$. Note that every vertex in $V(H)\cap Y_U$ has at least $(1-2\eta)pd$ neighbours in $Y_U\cap X$ as $H\subset G[X]$ is $(\eta,pd,K)$-dense, and, by \ref{prop:BU1}, if a vertex $v\in V(G)$ has at least $(1-2\eta)pd$ neighbours in $Y_U\cap X$, then it has at least $\frac{1}{p}(d_G(v,Y_U\cap X)-\eta pd)\geq (1-3\eta)d$ neighbours in $Y_U$ in $G$, and hence belongs to $Z_U$ (as defined in \ref{prop:BU2}). 
Therefore, $Z_U\cap X\subset Y_U\cap X$ has size at least $|V(H)\cap Y_U|\geq |Y_U\cap X|-\eta pd$ and hence 
\[
p|Z_U|+\eta pd\overset{\ref{prop:BU2}}{\geq} |Z_U\cap X|\geq |Y_U\cap X|-\eta pd\overset{\ref{prop:BU3}}{\geq} p|Y_U|-2\eta pd,
\]
so that $|Z_U|\geq |Y_U|-3\eta d$. Thus, setting $H':=G[Z_U]$, we have $\delta(H')\geq (1-6\eta)d$ by the definition of $Z_U$ in \ref{prop:BU2}. Furthermore,
\[
|Z_U|\leq |Y_U|\overset{\ref{prop:BU3}}{\leq} \frac{1}{p}(|Y_U\cap X|+\eta pd)\leq \frac{1}{p}(|V(H)|+2\eta pd)\leq \frac{1}{p}(Kpd+2\eta pd)\leq 2Kd,
\]
and therefore $H'$ is $(6\eta,d,2K)$-dense.

To show the last part of \ref{good:denseinheritance} holds for $H'=G[Z_U]$, suppose that $V(H')$ intersects with at least $d/2$ vertices in $V$, for some $V\in \mathcal{V}$. Then, as $Y_U$ contains $Z_U$ as so also interests with at least $d/2$ vertices in $V$, and as the event $B_{U,V}$ does not hold, we have $|(Y_U\cap V)\cap X|\geq pd/4$. As $|(Y_U\cap X)\setminus V(H)|\leq \eta pd$, we have by \ref{prop:BUV} that $V\cap V(H)\neq \emptyset$.
Thus, \ref{good:denseinheritance} holds, and therefore $X$ is a $(d,\gamma,k, \eta, K, p,\mathcal{V})$-{good-sample} in $G$, as required. 
\end{proof}


\subsection{Proof of Lemma~\ref{Lemma_X_connects}}\label{sec:lemmaXproof}

We come now to perhaps the most important part of our proof. Using a good sample, say $X$, provided by Lemma~\ref{Lemma_good_sample_exists}, we show that $X$ has the property required in Lemma~\ref{Lemma_X_connects}. That is, given a bounded edge-disjoint collection of path-forests $\mathcal{P}_1, \ldots, \mathcal{P}_{d/2}$ in $G - V(\mathcal{F})-X$ (as in the set up of Lemma~\ref{Lemma_X_connects}), we join lots of these paths together in an edge-disjoint fashion before joining all but at most a few of the resulting paths into the dense spots, all in some relatively well spread manner (that is, \ref{propX:twob} and \ref{propX:twoc} will hold). To do this, we join as many of the paths together in each path forest using vertices in $X$ (using paths with length at most $2k$ where $1/d\ll 1/k\ll \eta,p$, which together form the path forests denoted in the proof by $\mathcal{Q}_i$, $i\in [d/2]$), and then try to join as many of the paths into the dense spots in a well spread fashion using vertices in $X$ (using paths with length at most $k$, which together form the path forests denoted by $\mathcal{R}_i$, $i\in [d/2]$). We then show (for Claim~\ref{clm:actuallyallpathsgood}) that all of the $d/2$ resulting path forests will not have that many paths which do not end in some dense spot. 

We will show this by contradiction: If we can pick many endvertices of some $\mathcal{P}_i$ that have not had a path joined to them in $\mathcal{Q}_i$ or $\mathcal{R}_i$, then, looking at the iterative neighbourhoods in $X$ from these vertices, if they intersect within distance $k$ we will have a path we could perhaps add to $\mathcal{Q}_i$ (if the path uses no vertices in $\mathcal{Q}_i$). This will imply that not many of these iterating neighbourhoods can grow very large, and thus allow us to connect one of the path ends to an approximable dense spot in $G[X]$, and, hence, to one of the dense spots in the maximal collection $\mathcal{F}$, using \ref{good:denseinheritance}. In order to get a contradiction, we will need these iterating neighbourhoods to avoid vertices in $\mathcal{Q}_i$ and $\mathcal{R}_i$, the edges of all the connecting paths we have found, and some further vertices in order to ensure the final paths will be well spread (for \ref{propX:twob} and \ref{propX:twoc}) -- all of these vertices we collect into the set $Z_i$ in the proof of Claim~\ref{clm:actuallyallpathsgood} (see~\ref{Zreason1}--\ref{Zreason3}) and all of these edges we delete from $G$ to get $G'$. We cannot find a subgraph of $G'[X\setminus Z_i]$ with all its degrees close to $pd$ (as we would need for a simple iterating neighbourhood argument to work), but what we can do is find a large subset $X_k$ of almost all of $X\setminus Z_i$ such that we can iteratively find many neighbours in $G'$ in $X$ from vertices in $X_k$ as long as we only iterate up to $k$ times (by finding neighbours in subsequent sets $X_{k-1}\subset \ldots\subset X_1\subset X\setminus Z_i$ that we found in reverse order, showing they are large by induction, as at \eqref{eqn:induct}). Then, if there are too many unconnected endvertices of some $\mathcal{P}_i$, from at least one such endpoint we can do this iteration without growing in total very much (as intersecting iterated neighbourhoods would provide a short path connect two more original endpoints together).

Not only will this allow us to connect one more of the path ends in $\mathcal{P}_i$ to a dense spot in $G'[X\setminus Z_i]$, but, in this iteration we can easily find a sequence of vertices which are at distance at most $2k$ apart pairwise in $G$ whose neighbourhoods in $X$ together approximate the dense spot we find in $G'[X\setminus Z_i]$. This allows us to use \ref{good:denseinheritance} to connect the endpoint of $\mathcal{P}_i$ into a dense spot in $\mathcal{F}$, gaining the contradiction that will conclude the proof of the key claim, Claim~\ref{clm:actuallyallpathsgood}, from which we easily deduce Lemma~\ref{Lemma_X_connects}.

\begin{proof}[Proof of Lemma~\ref{Lemma_X_connects}] 
Let $\mathcal{V}=\{V(G_i):i\in [t]\}$ and take $\gamma$ and  $k$ such that $1/d\ll \gamma\ll 1/k\ll p,\eta$.
Using Lemma~\ref{Lemma_good_sample_exists}, let $X\subset V(G)$ be a $(d,\gamma/4,k, \eta/6, K/2, p,\mathcal{V})$-{good-sample} in $G$. That is, the following hold.
\stepcounter{propcounter}
\begin{enumerate}[label = {{{\textbf{\Alph{propcounter}\arabic{enumi}}}}}]

    \item\label{gd:vertexdegree} $|X|=(1\pm \gamma/4)pn$ and, for any $v\in V(G)$ we have that $d_G(v,X)= (1\pm \gamma/4)pd$.
    \item\label{gd:sizeinsubsets} For each $i\in [t]$, $|V(G_i)\cap X|=(1\pm \gamma/4)p|V(G_i)|$, and, for any $v\in V(G)$ we have that $d_G(v,V(G_i)\cap X)= p\cdot d_G(v,V(G_i))\pm \gamma pd/4$.
    \item\label{gd:denseinheritance} For any $H\subset G[X]$ which is $(\eta/6,pd,K/2)$-dense and $(k,\eta/6,pd)$-approximable in $G$ by neighbourhoods in $X$,
    there exists some $H'\subset G$ which is $(\eta,d,K)$-dense such that if $H'$ intersects with some $G_{i}$, $i\in [t]$, in at least $d/2$ vertices
    then $V(G_i)\cap V(H)\neq \emptyset$.
\end{enumerate}

Due to the maximality of $\mathcal{F}=\{G_1,\ldots,G_t\}$, and as any two $(\eta,d,K)$-dense spots in $G$ that intersect on some vertex must intersect in at least $(1-2\eta)d\geq d/2$ vertices (as $\Delta(G)\leq d$), \ref{gd:denseinheritance} immediately implies the following.
\begin{enumerate}[label = {{{\textbf{\Alph{propcounter}\arabic{enumi}}}}}]\addtocounter{enumi}{3}
    \item\label{gd:denseinheritanceupdated} For any $H\subset G[X]$ which is $(\eta/6,pd,K/2)$-dense and $(k,\eta/6,pd)$-approximable in $G$ by neighbourhoods in $X$, there exists  some $i\in [t]$ with $V(G_i)\cap V(H)\neq \emptyset$.
    \end{enumerate}

We will show that $X$ satisfies the required property. Firstly, \ref{propX:minusone} follows from \ref{gd:vertexdegree}, and \ref{propX:zero} follows from \ref{gd:sizeinsubsets} and each $G_i$ being $(\eta, d, K)$-dense. Now, take any $(2n/d,d^{1/4},d^{1/4})$-bounded edge-disjoint collection of path-forests $\mathcal{P}_1, \ldots, \mathcal{P}_{d/2}$ in $G - V(\mathcal{F})-X$.
We will show that there exists in $G-(V(\mathcal{F})\setminus X)$ an edge-disjoint collection of path forests $\mathcal{P}'_1, \mathcal{P}'_2, \ldots, \mathcal{P}'_{d/2}$
such that \ref{propX:one}--\ref{propX:three} hold, completing the proof of the lemma.

Let $Y=V(G)\setminus (V(\mathcal{F})\cup X)$.
Let $\mathcal{Q}_1, \ldots, \mathcal{Q}_{d/2}$ be edge-disjoint path forests of paths in $G[X\cup Y]$ with length at most $2k$ and no internal vertices in  $Y$, such that, for each $i\in [d/2]$, all of the endvertices of $\mathcal{Q}_i$ are among the endvertices of $\mathcal{P}_i$, and $\mathcal{Q}_i+\mathcal{P}_i$ creates a path forest. Subject to this, maximise the total number of all of the paths in  $\mathcal{Q}_1, \ldots, \mathcal{Q}_{d/2}$.
Then, take edge-disjoint path forests $\mathcal{R}_1, \ldots, \mathcal{R}_{d/2}$ of paths in $G[Y\cup X]$ with length at most $k$, such that, for each $i\in [d/2]$, the paths in $\mathcal{R}_i$ each have one endvertex among the endvertices of $\mathcal{P}_i+\mathcal{Q}_i$ and  one endvertex in $X\cap V(\mathcal{F})$ and all of their internal vertices in $X\setminus V(\mathcal{F}\cup \mathcal{P}_i\cup \mathcal{Q}_i)$, and such that each vertex in $V(G)$ is an endvertex in total of at most $\sqrt{d}$ paths in $\mathcal{R}_i$, $i\in [d/2]$, and, for each $j\in [t]$ and $i\in [d/2]$, at most $\sqrt{d}$ of the paths in $\mathcal{R}_i$ end in $G_j$.
Subject to this, maximise the total number of all of the paths in  $\mathcal{R}_1, \ldots, \mathcal{R}_{d/2}$.
 For each $i\in [d/2]$, let $\mathcal{P}_i'$ be the set of paths in $\mathcal{P}_i+\mathcal{Q}_i+\mathcal{R}_i$ which have both endvertices in $V(\mathcal{F})$. By construction, we have that \ref{propX:one}, \ref{propX:twob} and \ref{propX:twoc} hold, so it is left only to prove \ref{propX:three}.

For each $i\in [d/2]$, let $\mathcal{P}_i''$ be the path forest of paths in $\mathcal{P}_i+\mathcal{Q}_i+\mathcal{R}_i$ which are not in $\mathcal{P}_i'$. We will show the following.

\begin{claim}\label{clm:actuallyallpathsgood} For each $i\in [d/2]$, $\mathcal{P}''_i$ contains at most $16n/Kd$ paths.
\end{claim}
\begin{proof} 
Suppose, for a contradiction, that there is some $i\in [d/2]$ for which $\mathcal{P}''_i$ contains at least $16n/Kd$ paths. 
Let $Z_i$ be the set of vertices $v\in X$ such that one of the following holds.
\stepcounter{propcounter}
\begin{enumerate}[label = {{{\textbf{\Alph{propcounter}\arabic{enumi}}}}}]
\item $v$ is in a path in $\mathcal{Q}_i+\mathcal{R}_i$.\label{Zreason1}
\item $v$ is in $G_j$ for some $j\in [t]$ such that at least $\sqrt{d}/2$ of the paths in $\mathcal{P}'_{i}$ end in $G_j$.\label{Zreason2}
\item $v$ is an endpoint of at least $\sqrt{d}/2$ paths in the collection $\mathcal{P}'_{i'}$, $i'\in [d/2]$.\label{Zreason3}
\end{enumerate}
Then,
\begin{equation}\label{eqn:Zupperbound}
|Z_i|\leq \frac{2n}{d}\cdot 2k+Kd\cdot \frac{2\cdot n/d}{\sqrt{d}/2}+\frac{d/2\cdot 2n/d\cdot 2}{\sqrt{d}/2}\leq \frac{10Kn}{\sqrt{d}}.
\end{equation}
Let $G'$ be the graph $G$ with every edge in $\bigcup_{i\in [d/2]}\mathcal{Q}_i$ and $\bigcup_{i\in [d/2]}\mathcal{R}_i$ removed, so that
\begin{equation}\label{eqn:EGupperbound}
|E(G)\setminus E(G')|\leq \frac{d}{2}\cdot 2\cdot \frac{2n}{d}\cdot 2k=4nk.
\end{equation}
Let $X_0=X\setminus Z_i$. Iteratively, for each $1\leq j\leq k$, let $X_j$ be the set of vertices in $X_{j-1}$ with at least $(1-\gamma/2)pd$ neighbours in $X_{j-1}$ in $G'$. We will show, by induction, that, for each $0\leq j\leq k$,
\begin{equation}\label{eqn:induct}
|X\setminus X_j|\leq (16/\gamma)^{j}\cdot \frac{10Kn}{\sqrt{d}}.
\end{equation}
Note first that this holds for $j=0$ by \eqref{eqn:Zupperbound}.
Let then $j\in [k]$ and assume that $|X\setminus X_{j-1}|\leq (16/\gamma)^{j-1}\cdot 10Kn/\sqrt{d}$.
Then, as every vertex in $X_{j-1}$ has at least $(1-\gamma/4)pd$ neighbours in $X$ in $G$ by \ref{gd:vertexdegree}, every vertex in $X_{j-1}\setminus X_j$ has at least $\gamma pd/4$ neighbouring edges which are either in $G[X]\setminus G'$
or which lead to a vertex of $X\setminus X_{j-1}$. Thus, using \ref{gd:vertexdegree} for each vertex in $X\setminus X_{j-1}$,
\[
|X_{j-1}\setminus X_j|\cdot \frac{\gamma pd}{4}\leq |X\setminus X_{j-1}| \cdot 2pd+ E(G[X]\setminus G')\overset{\eqref{eqn:EGupperbound}}{\leq} (16/\gamma)^{j-1}\cdot \frac{10Kn}{\sqrt{d}}\cdot 2pd+4nk\leq 3\cdot (16/\gamma)^{j-1}\cdot \frac{10Kn}{\sqrt{d}}\cdot pd,
\]
so that $|X_{j-1}\setminus X_j|\leq (12/\gamma)\cdot (16/\gamma)^{j-1}\cdot 10Kn/\sqrt{d}$, and hence
\[
|X\setminus X_j|\leq (12/\gamma)\cdot (16/\gamma)^{j-1}\cdot \frac{10Kn}{\sqrt{d}}+(16/\gamma)^{j-1}\cdot \frac{10Kn}{\sqrt{d}}\leq (16/\gamma)^{j}\cdot \frac{10Kn}{\sqrt{d}},
\]
as required. This completes the proof of the induction hypothesis, and therefore, \eqref{eqn:induct} holds for every $0\leq j\leq k$. In particular, then, $|X\setminus X_k|\leq (16/\gamma)^k\cdot 10Kn/\sqrt{d}\leq n/d^{1/3}$, as $1/d\ll \gamma,1/k,1/K$.

Now, recalling that $\mathcal{P}''_i$ contains at least $16n/Kd$ paths, we will show that at least $8n/Kd$ of these paths has an endpoint which is not in $V(\mathcal{F})$ which has at least $pd/4$ neighbours in $G'$ in $X_k$. Suppose to the contrary that there is a set $A$ of at least $8n/Kd$ of the endpoints of $\mathcal{P}''_i$ which are each not in $V(\mathcal{F})$ and have at most $pd/4$ neighbours in $G'$ in $X_k$. For each $v\in A$, as $v\notin V(\mathcal{F})$, and as  $\mathcal{P}_1, \ldots, \mathcal{P}_{d/2}$ is $(2n/d,d^{1/4},d^{1/4})$-bounded, $v$ is the endvertex of at most $d^{1/4}$ paths in $\mathcal{P}_1, \ldots, \mathcal{P}_{d/2}$,
and, hence, has degree at most $d^{1/4}$ in $G-G'$. Thus, by the choice of $A$ and by \ref{gd:vertexdegree}, $v$ has at least $pd/2$ neighbours in $G$ in $X\setminus X_k$. On the other hand, again as  $\mathcal{P}_1, \ldots, \mathcal{P}_{d/2}$ is $(2n/d,d^{1/4},d^{1/4})$-bounded, each vertex in $X\setminus X_k$ is the neighbour of at most $d^{1/4}$ of the vertices in $A$ (a subset of the endpoints of $\mathcal{P}_i$). Therefore,
\[
\frac{4pn}{K}=\frac{8n}{Kd}\cdot \frac{pd}{2}\leq |A|\cdot \frac{pd}{2}\leq |X\setminus X_k|\cdot d^{1/4}\leq \frac{n}{d^{1/12}},
\]
a contradiction as $1/d\ll 1/K,p$.

Therefore, $\mathcal{P}''_i$ contains at least $8n/Kd$ paths with an endpoint which is not in $V(\mathcal{F})$ and has at least $pd/4$ neighbours in $G'$ in $X_k$. Thus, taking $r=8n/Kd$, we can let $v_1,\ldots,v_{r}\notin V(\mathcal{F})$, be endvertices of different paths in $\mathcal{P}_i''$ which each have at least $pd/4$ neighbours in $X_k$ in $G'$.

Now, by the maximality of the path forests  $\mathcal{Q}_1, \ldots, \mathcal{Q}_{d/2}$ and \ref{Zreason1}, there is no path of length at most $2k$ in $G'-Z_i$ with internal vertices in $X$ between any two of the vertices $v_j$, $j\in [r]$. Therefore, we can pick some $j\in [r]$ for which there are at most $|X|/r\leq 2pd/r\leq Kpd/2$ vertices in $X\setminus Z_i$ which can be reached by a path of length at most $k$ from $v_j$ in $G'$ with vertices otherwise in $X\setminus Z_i$.
Let $w_1=v_j$ and $A_1=N_{G'}(v_j,X_k)$, so that $|A_1|\geq pd/4$.
Then, for each $2\leq i'\leq k$, if possible, pick $w_{i'}\in A_{i'}$ with $|N_{G'}(w_{i'-1},X_{k-i'+1}\setminus A_{i'-1})|\geq \eta pd/12$ and let $A_{i'}=A_{i'-1}\cup N_{G'}(w_{i'},X_{k-i'+1})$, and, otherwise, stop.
As each $|A_{i'}|\leq Kpd/2$ and $|A_{i'}\setminus A_{i'-1}|\geq \eta pd/12$, this must stop for some $i'<k$ as $1/k\ll \eta,1/K$.
 Let $H=G'[A_{i'}]$. We will show that $H$ is $(\eta/6,pd,K/2)$-dense and $(k,\eta/6,pd)$-approximable in $G$ by neighbourhoods in $X$.

As $i'<k$, for each $v\in A_{i'}\subset X_{k-i'+1}$, $v$ has at least $(1-\gamma/2)pd$ neighbours in $G'$ in $X_{k-i'}$ and at most $\eta pd/12$ neighbours in $X_{k'-i'+1}\setminus A_{i'}$, so at least  $(1-\eta/6)pd$ neighbours in $A_{i'}$.
 Thus, as $|V(H)|= |A_{i'}|\leq Kpd/2$, $H$ is $(\eta/6,pd,K/2)$-dense. Furthermore, each vertex $w_1,\ldots,w_{i'}\in X_1$ has at most $2\gamma pd$ neighbours in $G[X]$ outside of $H$ by \ref{gd:vertexdegree} and the definition of $X_1$, and therefore, as $\gamma\ll 1/k, \eta,p$, the vertices  altogether have at most $k\cdot (2\gamma pd)\leq \eta pd/6$ vertices outside of $V(H)$ in $X$. Therefore, as $V(H)\subset \bigcup_{i''\in [i']}N_G(w_{i''},X)$, $H\subset G[X]$ is  $(k,\eta/6,pd)$-approximable in $G$ by neighbourhoods in $X$.

Therefore, by \ref{gd:denseinheritanceupdated}, there exists some $m\in [t]$ with $V(G_m)\cap V(H)\neq \emptyset$. Then, we can find a path from $v_j$ to $V(G_m)\cap V(H)$ with length at most $k$, which contradicts the maximality of the path forests $\mathcal{R}_{i''}$, $i''\in [d/2]$, where we recall that $V(H)$ contains no vertex in $Z_i$ and use \ref{Zreason1}, \ref{Zreason2} and \ref{Zreason3}.
\claimproofend

As all but at most $(1-\eps)d$ edges of each path in  $\mathcal{P}''_i$, $i\in [d/2]$, can be decomposed into paths of length $(1-\eps)d$, and $1/K\ll \eps$, by Claim~\ref{clm:actuallyallpathsgood} all but at most $(d/2)\cdot (1-\eps)d\cdot (16n/Kd)\leq \eps nd/4$ edges of  $E(\mathcal{P}''_1\cup\ldots\cup \mathcal{P}''_{d/2})$ can be decomposed into copies of $P_{(1-\eps)d}$, so that \ref{propX:three} holds, as required.
\end{proof}

\section{Decomposing inside the dense bits}\label{sec:insidedensebits}

In this section, we will (mostly) decompose the dense spots with attached paths (see Section~\ref{sec:proofsketch}).
Firstly we note that so far we have no conditions on the dense spots that imply they are connected, which we will need to connect various paths together. Therefore, in Section~\ref{sec:connectivity} we give the notation of connectivity we will use (Definition~\ref{def:connected}) and show that we can partition the vertices of our dense spots into (essentially) connected dense spots (Lemma~\ref{Lemma_partition_to_connected_sets}) and in each such connected dense spot find a small vertex subset which can connect each pair of vertices in the connected dense spot using many possible short edge-disjoint paths using internal vertices only in that vertex subset. In Section~\ref{sec:forestsindensebits}, we show we can decompose a dense spot into many path forests of some specified numbers of vertices and few overall paths (which we later connect up, including to an attached path where appropriate, using our connectivity property). Finally, in Section~\ref{sec:finalproofdecomposedensebits}, we put this all together and prove Lemma~\ref{Lemma_decompose_inside_dense_bits}, thus completing the proof of all our key lemmas, and hence Theorem~\ref{Theorem_main}.
\subsection{Connectivity}\label{sec:connectivity}
First, we define our connectivity.
\begin{definition}\label{def:connected}
A graph $G$ is $(\zeta, \lambda, d)$-\textbf{connected} if, for all $A,B\subseteq V(G)$ with $|A|,|B|\geq \zeta d$, and for all $F\subseteq E(G)$ such that $|F|\leq \lambda d^2$, there exists a path between $A$ and $B$ in $G-F$.
\end{definition}

Observe that any graph with fewer than $2\zeta d$ vertices is trivially $(\zeta, \lambda, d)$-connected. We also have the following observation.

\begin{observation}\label{obs_remains_connected}
   If $G$ is $(\zeta, \lambda, d)$-connected and $F\subseteq E(G)$ satisfies $|F|\leq \tau d^2$, then $G-F$ is $(\zeta, \lambda-\tau, d)$-connected. \qed
\end{observation}

We now show how to partition a dense graph into highly connected subgraphs.
\begin{lemma}\label{Lemma_partition_to_connected_sets} Let $1/d\ll \lambda \polysmall 1/K, \beta$ with $\beta\leq 1/4$. Let $G$ be a $(\beta,d, K)$-dense graph. Then, for some $t\leq 2K$, there exists a set $J\subseteq V(G)$ and vertex-disjoint subgraphs $G_1,\ldots,G_t\subset G$ with $V(G)=V(G_1)\cup \ldots\cup V(G_t)$ such that
\stepcounter{propcounter}
\begin{enumerate}[label = {{\emph{\textbf{\Alph{propcounter}\arabic{enumi}}}}}]
    \item $|J|\leq \sqrt{\lambda} d$,\labelinthm{connectprop:1}
    \item $G_i- J$ is $(2\beta, d,K)$-dense for each $i\in [t]$, and\labelinthm{connectprop:2}
    \item $G_i$ is $(\lambda^{1/4},\lambda,d)$-connected for each $i\in [t]$.\labelinthm{connectprop:3}
\end{enumerate}
\end{lemma}
    \begin{proof} 
    Set $\zeta:=\lambda^{1/4}$.
        Initialise $\mathcal{S}=\{G\}$, and, while there is some $G'\in \mathcal{S}$ which is not $(\zeta, \lambda,d)$-connected, take such a $G'$ and use this to find a set $F\subseteq E(G')$ and partition $V(G')=V_1\cup V_2$ such that $|F|\leq \lambda d^2$, $|V_1|,|V_2|\geq \zeta n$ and $e_{G-F}(V_1,V_2)=0$, before deleting the edges of $F$ and replacing $G'$ in $\mathcal{S}$ by $G'[V_1]$ and $G'[V_2]$. 
        \par Once this process terminates, for $t:=|\mathcal{S}|$, enumerate $\mathcal{S}$ as $\{G_1,\ldots, G_t\}$. Note that, from the process, if $i\neq j$, then we have deleted all of the edges in $G$ between $V(G_i)$ and $V(G_j)$, while each $G_i$ is $(\zeta,\lambda,d)$-connected (and so, in particular, \ref{connectprop:3} holds) with size at least $\zeta d$, and we deleted at most $t\cdot \lambda d^2$ edges from $G$ to get $\bigcup_{i\in[t]}G_i$. Noting that $V(G)=V(G_1)\cup\ldots\cup V(G_t)$ is a partition, we will define a $J$ satisfying \ref{connectprop:1} and \ref{connectprop:2}, and show that $t\leq 2K$, completing the proof of the lemma.

        Note that, $t\leq |V(G)|/(\zeta d)\leq K/\zeta=K\lambda^{-1/4}$, and the number of deleted edges is at most $t\cdot \lambda d^2$. Denote by $J$ the vertices of $G$ which are incident to more than $\lambda^{1/5} d$ deleted edges. Note that
        \begin{equation}\label{eqn:JSize}
        |J|\leq \frac{2t\cdot \lambda d^2}{\lambda^{1/5}d}\leq \frac{2K\lambda^{-1/4}\cdot \lambda d^2}{\lambda^{1/5}d}\leq  \sqrt{\lambda}d, 
        \end{equation}
        where we have used that $\lambda\polysmall 1/K$. Thus, \ref{connectprop:1} holds.

        Now, for each $i\in [t]$ and $v\in V(G_i)\setminus J$, $v$ is adjacent to at most $\lambda^{1/5}d$ edges in $G$ that were deleted (as $v\notin J$) and at most $|J|\leq \sqrt{\lambda} d$ edges to $J$. Thus, as $G$ is $(\beta,d, K)$-dense and $\lambda\polysmall \beta$, $v$ has at least $(1-\beta)d -  \sqrt{\lambda} d - \lambda^{1/5}d\geq (1-2\beta)d$ neighbours in $G_i-J$. As $|V(G_i)|\leq |V(G)|\leq Kd$, $G_i-J$ is therefore $(2\beta,d,K)$-dense, so \ref{connectprop:2} holds.
        As, for each $i\in [t]$, we now have $|V(G_i)|\geq (1-2\beta)d\geq d/2$, we get the improved bound $t\leq 2|V(G)|/d\leq 2K$, completing the proof of the required properties. 
    \end{proof}

    Next, we show that our connectivity implies we can find a short connecting path between any two vertex sets which are not too small.

    \begin{lemma}\label{lem_short_paths_connectivity} Let $1/d, \zeta,\lambda\leq 1$ and $K\geq 1$. Let $G$ be $(\zeta, \lambda, d)$-connected and have at most $Kd$ vertices, and suppose $U,V\subset V(G)$ are sets with size at least $\zeta n$.

    Then, $G$ contains a path from $U$ to $V$ with length at most $8K/\lambda$.
    \end{lemma}
    \begin{proof}
    Let $\ell=4/(K\lambda)$.
    Let $U_0=U$. For each $1\leq i\leq 2\ell$ in turn, let $U_i=U_{i-1}\cup N_G(U_{i-1})$. Then, $U_0\subset U_1\subset U_2\subset \ldots \subset U_{2\ell}$ is a nested sequence of sets in $V(G)$, and $|V(G)|\leq Kd$, so there must be $>\ell$ values of $i\in [2\ell]$ for which $|U_i\setminus U_{i-1}|\leq 2Kd/\ell$. Therefore, we can find some $j\in [2\ell-1]$ such that $|U_j\setminus U_{j-1}|\leq 2Kd/\ell$ and $|U_{j+1}\setminus U_{j}|\leq 2Kd/\ell$. Then,
    \[
    e_G(U_j,V(G)\setminus U_j)=e_G(U_j\setminus U_{j-1},U_{j+1}\setminus U_{j})\leq \Big(\frac{2Kd}{\ell}\Big)^2\leq \lambda d^2.
    \]
    As $G$ is $(\zeta, \lambda, d)$-connected, we therefore must have that either $|U_j|< \zeta d$ or $|V(G)\setminus U_j|<\zeta d$. As $|U_j|\geq |U_0|\geq \zeta d$, we thus must have $|V(G)\setminus U_j|<\zeta d$, and hence $V\cap U_j\neq \emptyset$. That is, there is a path from $U$ to $V$ in $G$ with length at most $j\leq 2\ell-1\leq 8K/\lambda$, as required.
    \end{proof}
Combined with Observation~\ref{obs_remains_connected}, this tells us that many such (almost as short) paths will exist, as follows.

    \begin{corollary}\label{cor_short_paths_connectivity} Let $1/d, \zeta,\lambda\leq 1$ and $K\geq 1$. Let $G$ be $(\zeta, \lambda, d)$-connected and have at most $Kd$ vertices, and suppose $U,V\subset V(G)$ are sets with size at least $\zeta n$.

    Then, $G$ contains at least $\lambda^2d^2/32K$ edge-disjoint paths from $U$ to $V$ with length at most $16K/\lambda$.
    \end{corollary}
    \begin{proof}
    Take a maximal collection $\mathcal{R}$ of edge-disjoint paths from $U$ to $V$ in $G$ which each have length at most $16K/\lambda$ and let the set of all of their edges be $E$. Then, by Lemma~\ref{lem_short_paths_connectivity}, the maximality of $\mathcal{R}$ implies that $G-E$ is not $(\zeta, \lambda/2, d)$-connected. Therefore, by Observation~\ref{obs_remains_connected}, we must have $|E|\geq \lambda d^2/2$. Thus, $|\mathcal{R}|\geq |E|/(16K/\lambda)\geq \lambda^2d^2/32K$, as required.
    \end{proof}

From this corollary, we can now show that we can even find many of these paths which only use internal vertices in a pre-selected, small vertex set (chosen randomly within the proof).

\begin{lemma}\label{connectingwithindense} $1/d\ll q\ll \lambda, \zeta, 1/K, \eps\leq 1$. Let $G$ be a $(\zeta,\lambda,d)$-connected graph such that $\Delta(G)\leq d$, $|V(G)|\leq Kd$, and $\delta(G)\geq \zeta d$.

Then, there is some $W\subseteq V(G)$ with $|W|\leq 2\eps |V(G)|$ and the following properties.
\stepcounter{propcounter}
\begin{enumerate}[label = {{\emph{\textbf{\Alph{propcounter}\arabic{enumi}}}}}]
\item  For each distinct $v,w\in V(G)$,  $G[W]$ contains at least $qd^2$ edge-disjoint paths of length at most $1/q$ between $N_G(v)$ and $N_G(w)$.\labelinthm{recall:1}
\item For all $v\in V(G)$, $d_G(v,W)\leq 2\eps d$.\labelinthm{recall:2}
\end{enumerate}
\end{lemma}
\begin{proof}
\par Let $W$ be a random subset of $V(G)$, sampled by including each vertex independently at random with probability $\eps$. For each $u,v\in V(G)$, we will show that \ref{recall:1} does not hold for that pair of vertices with probability at most $e^{-\sqrt{d}}$. Taking a union bound over all pairs of vertices will then show that \ref{recall:1} holds with probability at least $3/4$.

Fix then $u,v\in V(G)$. Using Corollary~\ref{cor_short_paths_connectivity}, take a collection $\mathcal{R}_{u,v}$ of at least $\lambda^2 d^2/32K$ edge-disjoint paths from $N_{G}(u)$ to $N_{G}(v)$ in $G$ which each have length at most $16K/\lambda$. Let $X_{u,v}$ be the number of paths in $\mathcal{R}_{u,v}$ whose vertices all lie in $W$. Note that, as the paths are edge-disjoint and $\Delta(G)\leq d$, changing whether a vertex is in $W$ or not can change $X_{u,v}$ by at most $d$, so therefore $X_{u,v}$ is $d$-Lipschitz. Furthermore, each path in $\mathcal{R}_{u,v}$ has all of its vertices in $W$ with probability at least $\eps^{1+16K/\lambda}$, so $\mathbb{E} X_{u,v}\geq (\lambda^2 d^2/32K)\cdot \eps^{1+16K/\lambda}\geq 2qd^2$ as $q\ll \lambda, 1/K,\eps$. Therefore, by Lemma~\ref{lem:Azuma}, we have
\[
\mathbb{P}(X_{u,v}<qd^2)\leq 2\exp\Big(-\frac{1}{2|V(G)|}\cdot \Big(\frac{qd^2}{d}\Big)^2\Big)\leq 2\exp\Big(-\frac{(qd)^2}{2Kd}\Big)\leq \exp(-\sqrt{d}),
\]
as claimed. Thus, as described above, by a suitable union bound, we have that \ref{recall:1} holds with probability at least $3/4$.

    Furthermore, by a simple application of Chernoff's bound and a union bound, as $\Delta(G)\leq d$, with probability at least $3/4$ we have that \ref{recall:2} holds. Finally, again by a simple application of Chernoff's bound, with probability at least $3/4$ we have that $|W|\leq 2\eps |V(G)|$. Together, then, $W$ satisfies all of the properties in the lemma with positive probability, and hence some such $W$ as desired must exist.
\end{proof}


\subsection{Decomposing dense spots into path forests with specified sizes}\label{sec:forestsindensebits}
Using the work in Section~\ref{sec:connectivity} we can decompose each dense spot in Lemma~\ref{Lemma_decompose_inside_dense_bits} into a few connected dense spots and join up any paths in the same path forest coming into the same connected dense spot. This will leave at most $d/2$ paths coming into each dense spot, as there are at most $d/2$ distinct path forests to begin with. For example, suppose we have paths $P_1,\ldots,P_{d/2}$ which are attached to a well connected dense spot $G$, where every path $P_i$ has length $d/2$. Having set aside a vertex subset of $G$ to make connections (as in Section~\ref{sec:connectivity}), we will then wish to find in $G$ plenty (almost $d/2$) edge-disjoint paths of length $(1-\eps)d-d/2$ to join to the paths $P_i$ to get (roughly) a path of length $(1-\eps)d$, while decomposing the rest of the edges mostly into paths of length $(1-\eps)d$. To make this easier, we find instead an edge-disjoint path \emph{forest} of few paths which have the right number of vertices, before using the connectivity property to join them up (including possibly to some path $P_i$).

To find edge-disjointly path forests with specified numbers of vertices in our dense spots we will use the following lemma. Its proof is similar to that of Lemma~\ref{Lemma_decomposition_path_forests_regular}: we partition the vertices into sets $A_i$, $i\in [s]$, with $s=2d^{0.15}$ and find many large matchings between each of them, putting them together to find large path forests without too many paths. Instead of decomposing the auxiliary graph $K_s$ into paths of length $(s-1)$, however, we will use Theorem~\ref{packingpaths} to mostly decompose $K_s$ into paths of different lengths, chosen so that the path forests we produce will have the desired size. For this, we start the proof by batching together the desired path forest sizes into groups with similar sizes, as each path we find in $K_s$ will produce many path forests with a similar size.

\begin{lemma}\label{Lemma_dense_specfic_lengths}
Let $1/d\ll \beta,1/K\polysmall\eps\leq 1$. Let $G$ be a $(\beta,d,K)$-dense graph. Let $r\in \N$, and suppose $n_1,\ldots,n_r\in [\eps d,(1-\eps)d]$ are such that $\sum_{i\in [r]}n_i\leq (1-10^3\eps)|V(G)|d/2$.

Then, $G$ contains edge-disjoint path forests $F_1,\ldots,F_r$ such that, for each $i\in [r]$, $|V(F_i)|=n_i$ and $F_i$ contains at most $d^{9/10}$ paths.
\end{lemma}
\begin{proof}
Note that we can assume that $\eps\leq 1/10^3$. Let $n=|V(G)|\leq Kd$ and note that $r\leq |E(G)|/(\eps d)\leq Kd/2\eps$. Let $s=2d^{0.15}$ and $\bar{n}=\lfloor n/s\rfloor\geq (1-\eps)n/s$.
Take $\mu$ such that $1/d\ll \mu \ll \beta,1/K$ and $\mu n\in \mathbb{N}$ with $\mu n =0\;\textrm{mod}\;\bar{n}$. We will batch the lengths $n_i$ together so that in each batch they have length varying by only up to $\mu n$. For this, for each $i\in [r]$, let $n_i'=\lceil n_i/\mu n\rceil\cdot \mu {n}\leq n_i+\mu n$. Note that there are at most $d/(\mu d)+1\leq 2/\mu$ different values taken by $n_i'$, $i\in [r]$, and 
\begin{equation}\label{eqn:nprimesum}
\sum_{i\in [r]}n_i'\leq \sum_{i\in [r]}(n_i+\mu n)\leq (1-10^3\eps)n\cdot \frac{d}{2}+\frac{Kd}{2\eps}\cdot \mu n\leq (1-5\cdot 10^2\eps)n\cdot \frac{d}{2},
\end{equation}
as $1/d\ll \mu \ll \beta,1/K,\eps$.

Let  $i_{\mathrm{min}}=\min\{n_i'/\bar{n}:i\in [r]\}$ and $i_{\mathrm{max}}=\max\{n_i'/\bar{n}:i\in [r]\}$,
so that $i_{\mathrm{min}}\geq \eps n/\bar{n}\geq \eps s$ and $i_{\mathrm{max}}\leq (1-\eps+\mu)n/\bar{n}\leq (1-\eps/2)s$. Let $I=\{i:\exists j\in [r]\text{ s.t.\ }n_j'=i\cdot \bar{n}\}$, so that $i_{\mathrm{min}}\leq i\leq i_{\mathrm{max}}$ for each $i\in I$, and $|I|\leq 2/\mu$. 
For each $i\in I$, let $m_i$ be the number of $n_j'$, $j\in [r]$, with $n_j'=i\cdot \bar{n}$.
Let $\eta'=40d^{-0.2}$ and $d'=(1-\eta')(1-100\beta)d/s$. For each $i\in I$, let $s_i=\lceil m_i/d'\rceil$.
Then,
\begin{align*}
\sum_{i\in I}i\cdot s_i&\leq \sum_{i\in I}i+\sum_{i\in I}i\cdot\frac{m_i}{d'}\leq \frac{2s}{\beta}+\sum_{i\in I}i\cdot\frac{|\{j\in[r]:n_j'=i\cdot \bar{n}\}|}{(1-\eta')(1-100\beta)d/s}\\
&\leq \frac{\eps s^2}{10}+\sum_{j\in [r]}\frac{n_j'}{\bar{n}}\cdot\frac{s}{(1-\eta')(1-100\beta)d}\\
&\leq \frac{\eps s^2}{10}+\sum_{j\in [r]}\frac{n_j'}{(1-\eps)n/s}\cdot\frac{s}{(1-\eta')(1-100\beta)d}\\
&\overset{\eqref{eqn:nprimesum}}{\leq} \frac{\eps s^2}{10}+(1-5\cdot 10^2\eps)n\cdot\frac{d}{2}\cdot \frac{1}{(1-\eps)n/s}\cdot\frac{s}{(1-\eta')(1-100\beta)d}\\
&\leq (1-\eps/4)\binom{s}{2}\leq \binom{(1-\sqrt{\eps}/3)s}{2}.
\end{align*}
Thus, by Theorem~\ref{packingpaths}, we can take a collection $\mathcal{Q}$ of edge-disjoint paths in $K_s$ such that, for each $i\in I$,
 there are $s_i$ paths with $i$ vertices in $\mathcal{Q}$. For each $i\in I$, let $\mathcal{Q}_i$ be the set of paths in $\mathcal{Q}$ with $i$ vertices, so that $|\mathcal{Q}_i|=s_i$.

Let $\gamma=2d^{-0.4}$. Apply Lemma~\ref{lem:allverticesregularspanningnearregular} to $G$ with $(1-\eps)d$ in place of $d$ and $2\eps$ in place of $\gamma$ to find a subgraph $G'\subset G$ with $|V(G')|\geq (1-100\eps)|V(G)|$ and some $d_0\geq (1-100\eps)d$ such that, for each $v\in V(G')$, $d_{G'}(v)=(1\pm \gamma)d_0$.

\par  Let $\eta=4d^{-0.4}$. As in the proof of Lemma~\ref{Lemma_decomposition_path_forests_regular}, using the local lemma, take a partition of $V(G')$ as $A_1,\ldots, A_s$ for which the following properties hold.
\stepcounter{propcounter}
\begin{enumerate}[label = {{{\textbf{\Alph{propcounter}\arabic{enumi}}}}}]
\item For each $v\in V(G)$ and $i\in [s]$, $d_{G}(v,A_i)=(1\pm \eta)d/s$.\label{prop:degreeintoAi2}
\item For each $i\in [s]$, $|A_i|=(1\pm \eta)n/s$.\label{prop:Aisize2}
\end{enumerate}
Note that $\eta'=40d^{-0.2}=20\sqrt{\eta}$ and $d'=(1-\eta')d_0/s$. For each edge $e=jk$ in the complete $s$-vertex graph $K_s$, using Lemma~\ref{lem:vizing}, \ref{prop:degreeintoAi2}, and \ref{prop:Aisize2},  find $d'$ edge-disjoint matchings in ${G}[A_j,A_k]$ which each have at least $(1-\eta')n/s$ edges.
Call these matchings $M_{e,i}$, $i\in [d']$.

\par For each $P\in \mathcal{Q}$ and $i\in [d']$, let $F_{P,i}$ be the path forest with vertex set $\cup_{j\in V(P)}A_j$ and edge set $\cup_{e\in E(P)}M_{e,i}$. Observe that, as the paths in $\mathcal{Q}$ are edge-disjoint, the subgraphs $F_{P,i}$, $P\in \mathcal{Q}$ and $i\in [d']$ are edge disjoint. Moreover, observe that, for each $P\in \mathcal{Q}$ and $i\in [d']$, if $j$ is an interior vertex of $P$, then there are at most $2(|A_j|-(1-\eta')n/s)$ vertices in $A_j$ which are an endvertex of a path in $F_{P,i}$. Thus,
the number of paths in $F_{P,i}$ is at most, by \ref{prop:Aisize2},
\[
2(1+\eta)\frac{n}{s}+(s-2)\cdot 2((1+\eta)-(1-\eta'))\frac{n}{s}\leq \frac{3n}{s}+{4\eta'n}\leq \frac{4n}{s}\leq d^{9/10},
\]
as $n\leq Kd$, $s=2d^{0.15}$ and $1/d\ll 1/K$.

Thus, the path forests $F_{P,i}$, $P\in \mathcal{Q}$ and $i\in [d']$, are edge-disjoint, each have at most $d^{9/10}$ paths, and, for each $i\in I$, there are at least $s_i\cdot d'\geq m_i$ path forests with $\bar{n}\cdot i$ vertices. Thus, by the choice of the $m_i$, $i\in I$, we can find $r$ edge-disjoint path forests $F_i$, $i\in [r]$, in $G$ such that, for each $i\in [r]$, $F_i$ has at most $d^{9/10}$ paths and at least $n'_i\geq n_i$ vertices. Iteratively removing $n_i'-n_i$ leaves from $F_i$, for each $i\in [r]$, then gives the desired paths.
\end{proof}

For our application, we now prove a simple variant of Lemma~\ref{Lemma_dense_specfic_lengths} in which the endvertices of the paths additionally satisfy some weak spreadness condition.

\begin{corollary}\label{Corollary_dense_specfic_lengths}
Let $1/d\ll \beta,1/K\polysmall\eps\leq 1$. Let $G$ be a $(\beta,d,K)$-dense graph. Let $r\in \N$, and suppose $n_1,\ldots,n_r\in [\eps d,(1-\eps)d]$ are such that $\sum_{i\in [r]}n_i\leq (1-2\cdot 10^3\eps)|V(G)|d/2$.

Then, $G$ contains edge-disjoint path forests $F_1,\ldots,F_r$ such that, for each $i\in [r]$, $n_i\leq |V(F_i)|\leq (1+\eps)n_i$ and $F_i$ contains at most $d^{9/10}$ paths, and each vertex in $G$ appears as the endpoint of at most $d^{19/20}$ paths in $F_1,\ldots,F_r$.
\end{corollary}
\begin{proof}
For each $i\in [r]$, let $n_i'=(1+\eps)n_i$, noting that
\[
\sum_{i\in [r]}n_i'=(1+\eps)\sum_{i\in [r]}n_i\leq (1-\eps)(1-2\cdot 10^3\eps)\cdot \frac{|V(G)|}{2}\leq (1-10^3\eps)\cdot \frac{|V(G)|}{2}.
\]
Therefore, by Lemma~\ref{Lemma_dense_specfic_lengths}, we can find edge-disjoint path forests $F'_1,\ldots,F_r'$  in $G$ such that, for each $i\in [r]$, $|F'_i|=n_r'$ and $F_i'$ contains at most $d^{9/10}$ paths.

For each $1\leq i\leq r$ in turn do the following to find $F_i\subset F_i'$ with at most $d^{9/10}$ paths. Let $Z_i$ be the vertices of $G$ which appear as the endpoints of at least $d^{19/20}$ different paths in $F_j$, $j<i$, so that $|Z_i|\leq r\cdot 2d^{9/10}/d^{19/20}\leq (Kd^2/\eps d)\cdot 2d^{-1/20}\leq \eps^2d\leq\eps n_i$. Thus, we can iteratively delete any endpoints in $Z_i$ from the paths in $F_i'$ until they are all outside of $Z_i$, to get $F_i$, while deleting at most $|Z_i|\leq \eps n_i$ vertices, so that, hence, $|F_i|\geq n_i$. By the choice of $Z_i$, $i\in [r]$, the path forests $F_i$, $i\in [r]$, have our desired properties.
\end{proof}

We now use Corollary~\ref{Corollary_dense_specfic_lengths} to deduce an almost decomposition of a connected dense spot with a few attached paths, as follows.

\begin{lemma}\label{lem:densespotplusattachedpathdecomposes} Let $1/d\ll \lambda \ll p,1/K\polysmall\eps\leq 1$. Let $G$ be a $(\lambda^{1/4},\lambda/2,d)$-connected graph which contains a set $J\subset V(G)$ with $|J|\leq \eps d/4$ such that $G-J$ is $(p,d,K)$-dense, and suppose that $\delta(G)\geq d/8K$ and $\Delta(G)\leq d$. Let $P_1,\ldots,P_{d/2}$ be edge-disjoint paths which each have exactly one vertex in $G$, which is, moreover, an endvertex and in $J$. Suppose that no vertex of $G$ is the endvertex of more than $\sqrt{d}$ paths $P_i$, $i\in [d/2]$.

Then, all but at most $\eps|V(G)|d$ edges of $G\cup P_1\cup \ldots \cup P_{d/2}$ can be decomposed into copies of $P_{(1-\eps)d}$.
\end{lemma}
\begin{proof} Let $q$ be such that $1/d\ll q\ll \lambda$.
By Lemma~\ref{connectingwithindense}, there is a set $W\subset V(G)$ with $|W|\leq \lambda |V(G)|$ and the following property.

\stepcounter{propcounter}
\begin{enumerate}[label = {{{\textbf{\Alph{propcounter}}}}}]
\item  \label{Wiconnects} For each distinct $v,w\in V(G)$, there are at least $qn^2$ edge-disjoint paths with length at most $1/q$ between $N_G(v)$ and $N_G(w)$ in $G[W]$ with all their vertices in $W$. \label{Wconnects} 
\end{enumerate}

Now, as $G-J$ is $(p, d, K)$-dense, and $|W|\leq \lambda d$,
$G':=G-J-W$ is $(2p,d,K)$-dense, and
\begin{equation}\label{eqn:Gprimesize}
|V(G')|\geq |V(G)\setminus J|-|W|\geq \delta(G-J)-\lambda d\geq (1-p)d-pd=(1-2p)d.
\end{equation}
 For each $i\in [d/2]$, let $\ell_i$ be the length of $P_i$, and let $n_i\in [(1-\eps)d]$ be such that $\ell_i+n_i=0\mod(1-\eps)d$.
 Note that 
 \[
 \sum_{i\in [d/2]}n_i\leq (1-\eps)d\cdot d/2\overset{\eqref{eqn:Gprimesize}}{\leq} (1-\eps/4)|V(G')|d/2.
 \]
Using this, let $s$ be the largest integer such that
\begin{equation}\label{eqn:sdefn}
(s-d/2)\cdot (1-\eps)d+\sum_{i\in [d/2]}n_i\leq (1-\eps/4)|V(G')|d/2,
\end{equation}
and, for each $d/2<i\leq s$, let $\ell_i=0$ and $n_i=(1-\eps)d$. Note that, by the choice of $s$,
\begin{align}
(1-\eps/4)|V(G')|d/2&\overset{\eqref{eqn:sdefn}}{\geq}\sum_{i\in [s]}n_i=(s-d/2)\cdot (1-\eps)d+\sum_{i\in [d/2]}n_i\geq (1-\eps/4)|V(G')|d/2-(1-\eps)d\nonumber \\
&\geq (1-\eps)|V(G)|d/2\geq |E(G)|-\eps |V(G)|d/2,\label{eqn:last}
\end{align}
where we have used that $\Delta(G)\leq d$.

By Corollary~\ref{Corollary_dense_specfic_lengths}, there are edge-disjoint path forests $F_1,\ldots,F_{s}\subset G'$, each of at most $d^{9/10}$ paths, such that $n_i\leq |V(F_i)|\leq (1+\eps)n_i$ for each $i\in [s]$ and every vertex in $V(G_j)$ appears as the endvertex of at most $d^{19/20}$ of the paths in $F_1,\ldots,F_s$.
Now, for each $i\in [s]$,  let $E_i$ be a set of pairs of vertices from the endvertices of $F_i$ and $P_i$ so that $P_i\cup F_i+{E}_i$ is a path, and note that $|E_i|\leq d^{9/10}$ as $F_i$ contains at most $d^{9/10}$ paths.
Let $F\subset \{(i,xy):i\in [d/2],xy\in E_i\}$ be a maximal set for which there are edge-disjoint paths $R_{i,xy}$, $(i,xy)\in F$, for which the following hold.\stepcounter{propcounter}
\begin{enumerate}[label = {{{\textbf{\Alph{propcounter}\arabic{enumi}}}}}]
\item For each $(i,xy)\in F$, $R_{i,xy}$ is an $x,y$-path in $G$ with length at most $2/q$ and internal vertices in $W$.\label{nearlynearly1}
\item For each $i\in [d/2]$, the internal vertices of $R_{i,e}$, $(i,e)\in F$, are all distinct.\label{nearlynearly2}
\item Every vertex in $V(G)$ appears in at most $2d^{199/200}$ edges in $\bigcup_{(i,xy)\in F:v\notin \{x,y\}}R_{i,xy}$.\label{nearlynearly3}
\end{enumerate}

Suppose for contradiction that there is some $(i,xy)$ with $i\in [d/2]$ and $xy\in E_i$ such that $(i,xy)\notin F$.  By \ref{nearlynearly3} and as each vertex in $G$ appears as the endvertex of at most $\sqrt{d}$ paths $P_i$, $i\in [d/2]$, and at most $d^{19/20}$ of the paths in $F_1,\ldots,F_s$, each vertex is in at most $2d^{199/200}+2\sqrt{d}+2d^{19/20}\leq qd/4$ edges in paths in  $\bigcup_{(i,xy)\in F}R_{i,xy}$.
Using \ref{Wiconnects}, find a set $\mathcal{R}$ of $qd^2$ edge-disjoint paths of length at most $1/q$ between $N_G(x)$ and $N_G(y)$ and internal vertices in $W$. Note that, due to the maximality of $F$, every path in $\mathcal{R}$ must either contain an edge of $R_{i',e}$, $(i',e)\in F$, or contain an internal vertex in $R_{i,e}$, $(i,e)\in F$, or contain a vertex appearing in at most $2d^{199/200}$ edges in $\bigcup_{(i',e)\in F}R_{i',e}$, or there is an edge from its endpoints to $\{x,y\}$ which is in $\bigcup_{(i,xy)\in F}R_{i,xy}$.
However, in order, this implies that  
 \begin{equation}\label{eqn:lastmaybe}
|\mathcal{R}|\leq s\cdot d^{9/10}\cdot \frac{2}{q}
+d^{9/10}\cdot \frac{2}{q}\cdot d
+\frac{s\cdot 2\cdot d^{9/10}\cdot (2/q)}{2d^{199/200}}\cdot d
+2\cdot d^{9/10}\cdot d 
\leq \frac{qd^2}{2},
 \end{equation}
 where we have used that $s\leq |V(G)|\cdot (d/2)/((1-\eps)d)\leq Kd$.
As \eqref{eqn:lastmaybe} contradicts that $|\mathcal{R}|\geq qd^2$, we must then have that $F=\{(i,xy):xy\in E_i\}$.

Now, note that, for each $i\in [s]$, $P_i\cup F_i\cup(\bigcup_{e\in E_i}R_{i,e})$ is a path with length at least $\ell_i+n_i$. As, for each $i\in [s]$, $\ell_i+n_i=0\mod (1-\eps)d$, $P_i\cup F_i\cup(\bigcup_{e\in E_i}R_{i,e})$ contains edge-disjoint copies of $P_{(1-\eps)d}$ which cover at least $\ell_i+n_i$ edges. Therefore, in total, we can find edge-disjoint copies of $P_{(1-\eps)d}$ which cover at least $\sum_{i\in [s]}(\ell_i+n_i)$ edges of $G\cup (P_1\cup\ldots\cup P_{d/2})$. As the number of edges in $P_1\cup\ldots \cup P_{d/2}$ is $\sum_{i\in [s]}\ell_i$, and, from \eqref{eqn:last}, 
these paths cover all but at most $\eps |V(G)|d$ edges of $G\cup P_1\cup\ldots\cup P_{d/2}$, as required.
\end{proof}


\subsection{Proof of Lemma~\ref{Lemma_decompose_inside_dense_bits}}\label{sec:finalproofdecomposedensebits}
Finally, then, we can prove Lemma~\ref{Lemma_decompose_inside_dense_bits}.
\begin{proof}[Proof of Lemma~\ref{Lemma_decompose_inside_dense_bits}] Take $q$ and $\lambda$ with $1/d\ll q\ll \lambda \polysmall \eta$.
We will prove the lemma by induction on $t=|\mathcal{F}|$. If $|\mathcal{F}|=0$, then we must have that all of the path forests are empty, as there is nowhere for their endvertices to go, and thus the result is trivial in this case. Suppose then that $t\geq 1$ and that the lemma is true if $|\mathcal{F}|=t-1$, and let $|\mathcal{F}|=t$.

Let $H=G_t$, and, using \ref{prop:densedecomp1}, let $X_H\subseteq V(H)$ be such that $X_H$ contains all the vertices of $\cup_{i\in [d/2]}\mathcal{P}_i$ in $V(G_i)$ (which are necessarily endpoints), and, for each $v\in V(H)$, $d_{H}(v,X_H)=(1\pm \eta)pd$. From this, and as $H$ is $(\eta,d,K)$-dense, we have that $H-X_H$ is $(2p,d,K)$-dense. Therefore, by Lemma~\ref{Lemma_partition_to_connected_sets}, we can find some $s\in \N$ and a set $J\subset V(H)\setminus X_H$ and vertex-disjoint subgraphs $H_{1},\ldots,H_{s}$ whose vertex sets partition $V(H)\setminus X_H$ such that each $H_{i}$, $i\in [s]$, is $(\lambda^{1/4},\lambda,d)$-connected, each $H_{i}-J$ is $(4p, d, K)$-dense, $s\leq 2K$, and $|J|\leq \sqrt{\lambda}d$.
Furthermore, using that $H$ is $(\eta,d,K)$-dense, and so each vertex in $X_H$ has at least $(1-\eta)d-(1+\eta)pd\geq d/2$ neighbours in $V(H)\setminus X_H$, partition $X_H$ as $\bigcup_{i\in [s]}X_{H,i}$ so that, for each $i\in [s]$ and $v\in X_{H,i}$, $v$ has at least $d/4K$ neighbours in $G$ in $V(H_{i})$.

Let $\mathcal{Q}^0=\emptyset$ and, for each $i\in [d/2]$, let $\mathcal{P}_i^0=\mathcal{P}_i$.
We will show by induction the following claim, where, essentially, for each $j$ in turn, we look to connect paths in $\mathcal{P}_i^{j-1}$ (for each $i\in [d/2]$) with endvertices in $H_j$ together while decomposing at most $1$ path from each $\mathcal{P}_i^{j-1}$, $i\in [d/2]$, along with the unused edges in $H_j$ into copies of $P_{(1-\eps)d}$ using Lemma~\ref{lem:densespotplusattachedpathdecomposes} (below, these copies of $P_{(1-\eps)d}$ appear as those in the set $\mathcal{Q}^j$, while $\mathcal{P}_i^{j}$ will be some of the paths in $\mathcal{P}_i^{j-1}$ which have been joined up where necessary so that they have no endpoints in $H_j$).

\begin{claim}\label{lastclaim}
For each $0\leq j\leq s$, there are edge-disjoint path forests $\mathcal{P}_i^j$, $i\in [d/2]$, and $\mathcal{Q}^j$ such that the following hold (where \emph{\ref{prop:densedecomp1new}}--\emph{\ref{prop:densedecomp3new}} allow us to maintain the properties for the paths $\mathcal{P}_i^j$ that will later allow us to apply the induction hypothesis to these paths and $G_1,\ldots,G_{t-1}$).
\stepcounter{propcounter}
\begin{enumerate}[label = {{\emph{\textbf{\Alph{propcounter}\arabic{enumi}}}}}]
        \item\labelinthm{prop:densedecomp1new} For each $i\in [t-1]$, there is some $X_i\subseteq V(G_i)$ so that $X_i$ contains all of the vertices of $V(\mathcal{P}^j_1)\cup\ldots\cup V(\mathcal{P}^j_{d/2})$ in $V(G_i)$, and, for each $v\in V(G_i)$, $d_{G_i}(v,X_i)=(1\pm \eta)pd$.
        \item\labelinthm{prop:densedecomp2new}  Each vertex $v\in V(G_1)\cup \ldots \cup V(G_{t-1})$ is an endpoint of in total at most $\sqrt{d}$ paths from $\mathcal{P}^j_1,\ldots, \mathcal{P}_{d/2}^j$.
        \item\labelinthm{prop:densedecomp3new} For each $j'\in [t-1]$ and $i\in [d/2]$ at most $\sqrt{d}$ of the paths in $\mathcal{P}_i^j$ have at least one endpoint in $G_{j'}$.
\item For each $i\in [d/2]$, the path forest $\mathcal{P}_i^j$ is contained in $\mathcal{P}_i\cup (\bigcup_{j'\leq j}H_{j'})$ and its paths have both endpoints in $\cup_{j'>j}X_{j'}$. \labelinthm{prop1}
\item $\mathcal{Q}^j$ is a collection of edge-disjoint copies of $P_{(1-\eps)d}$ with edges in $E(\bigcup_{j'\leq j}H_{j'})\cup (\bigcup_{i\in [d/2]}E(\mathcal{P}_i))$ which is edge-disjoint from each $\mathcal{P}_i^j$, $i\in [d/2]$.\labelinthm{prop2}
\item The number of edges in $\bigcup_{j'\leq j}H_{j'}$ and $\bigcup_{i\in [d/2]}\mathcal{P}_i$ which are not in $\mathcal{P}_i^j$, $i\in [d/2]$, or $\mathcal{Q}^j$ is at most $\sum_{j'\leq j}\eps |V(H_{j'})|d/2$.\labelinthm{prop3}
\end{enumerate}
\end{claim}
\begin{proof}
Note that this is easily true for $j=0$ using that we set $\mathcal{P}_i^0=\mathcal{P}_i$ for each $i\in [d/2]$ and $\mathcal{Q}^0=\emptyset$, where \ref{prop:densedecomp1new}--\ref{prop:densedecomp3new} follow from \ref{prop:densedecomp1}--\ref{prop:densedecomp3}. Assume then that $j\in [s]$ and that we have edge-disjoint path forests $\mathcal{P}_i^{j-1}$, $i\in [d/2]$, and $\mathcal{Q}^{j-1}$ satisfying \ref{prop:densedecomp1new}--\ref{prop3} with $j-1$ in place of $j$.
For each $i\in [d/2]$, let $E_i\subset X_{H,j}^{(2)}$ be a maximal set of pairs of vertices from the endpoints of $\mathcal{P}_i^{j-1}$ in $X_{H,j}$ such that $\mathcal{P}_i^{j-1}+E_i$ is a path forest (where it may be that none of the pairs in $E_i$ are edges in the graph). Note that in $\mathcal{P}_i^{j-1}+E_i$ at most 1 path will have an endpoint in $V(H_j)$, by the maximality of $E_i$.

We will now greedily find edge-disjoint paths $R_{i,e}$, $i\in [d/2]$ and $e\in E_i$, such that, the following hold.
\stepcounter{propcounter}
\begin{enumerate}[label = {{{\textbf{\Alph{propcounter}\arabic{enumi}}}}}]
\item For each $i\in [d/2]$ and $xy\in E_i$, $R_{i,xy}$ is an $x,y$-path in $H_j$ with length at most $2/q$.\label{nearly1}
\item For each $i\in [d/2]$, the internal vertices of $R_{i,e}$, $e\in E_i$, are all pairwise disjoint.\label{nearly2}
\item Every vertex in $V(H_j)$ appears in at most $2d^{3/4}$ edges in $\bigcup_{i\in [d/2],e\in E_i}R_{i,e}$.\label{nearly3}
\end{enumerate}
To see this is possible, suppose we have found paths $R_{i,e}$, $(i,e)\in F$, for some $F\subset \{(i,e):i\in [d/2],e\in E_j^i\}$ and are looking to find the path $R_{i',xy}$, for some $i'\in [d/2]$ and $xy\in E_i$ with $(i',xy)\notin F$.
 Let $Z_{i',xy}$ be the set of vertices $v$ in $H_j$ which either 
\begin{itemize}
    \item are in at least $d^{3/4}$ edges in $\bigcup_{(i,e)\in F}R_{i,e}$, or
    \item are an internal vertex of one of the paths $R_{i',e}$, $(i',e)\in F$, or
    \item are such that $xv$ or $yv$ are in some $R_{i,e}$, $(i,e)\in F$.
\end{itemize}
By \ref{prop:densedecomp3new} we have $|F|\leq d^{3/2}$, so that, using \ref{prop:densedecomp3new} and \ref{prop:densedecomp2new}, we have 
\begin{equation}\label{eq:newZsize}
|Z_{i',xy}|\leq \frac{d^{3/2}\cdot (2/q)}{d^{3/4}}+(2/q)\cdot \sqrt{d}+2\sqrt{d} \leq \frac{qd}{4}.
\end{equation}

Let $E_F$ be the set of all the edges of all the paths $R_{i,e}$, $i\in [d/2]$ and $e\in E_i$, so that $|E_F|\leq (2/q)\cdot d^{3/2}\leq q d^2/2$ by \ref{prop:densedecomp3new}. 
By the choice of $X_{H,j}$, and as $x,y\in X_{H,j}$ implies $d_G(x,V(H_j)),d_G(y,V(H_j))$, and as $\lambda \ll 1/K$ and $H_j$ is $(\lambda^{1/4},\lambda,d)$-connected, by Corollary~\ref{cor_short_paths_connectivity} 
there is a set $\mathcal{R}$ of at least $qd^2$ edge-disjoint paths of length at most $1/q$ between $N_G(x,V(H_j)\setminus Z_{i',xy})$ and $N_G(y,V(H_j)\setminus Z_{i',xy})$ in $G$.
As $\Delta(G)\leq d$, for each $v\in V(G)$ at most $d$ of the paths in $\mathcal{R}$ go through $v$. Therefore, as $|E_F|\leq qd^2/2$ and as $|Z_{i',xy}|\leq qd/4$ by \eqref{eq:newZsize}, we can find a path $R'_{i',xy}\in \mathcal{R}$ which has no vertex in $Z_{i',xy}$ or edge in $E_F$. 
Say $R'_{i',xy}$ has endpoints $x'\in N_G(x,V(H_j)\setminus Z_{i',xy})$ and $y'\in N_G(y,V(H_j)\setminus Z_{i',xy})$. 
Add edges $xx'$ and $yy'$ to $R'_{i',xy}$ and call this path $R_{i',xy}$. Then, by the choice of $Z_{i',xy}$ (so that $xx',yy'\notin E_F$) and $E_F$, this path avoids all of the edges in $E_F$, and its internal vertices are disjoint from those of  $R_{i',e}$, $(i',e)\in F$, and no vertex in $V(H_j)$ appears in more than $2d^{3/4}$ edges in $\bigcup_{(i,e)\in F\cup \{(i',xy)\}}R_{i,e}$. Therefore, we can pick $R_{i',xy}$ as required, and thus can pick paths $R_{i,e}$, $i\in [d/2]$ and $e\in E_i$ satisfying \ref{nearly1}--\ref{nearly3}.

Now, for each $i\in [d/2]$, by the choice of $E_i$ and by \ref{nearly1} and \ref{nearly2}, we have that $\mathcal{P}^{j-1}_i\cup(\bigcup_{e\in E_i}R_{i,e})$ is a path forest. Let $\mathcal{P}^j_i$ be the path forest of paths in $\mathcal{P}^{j-1}_i\cup(\bigcup_{e\in E_i}R_{i,e})$ with no endvertices in $X_j$, so that, therefore, \ref{prop1} holds.
For each $i\in [d/2]$, by the maximality of $E_i$ we have that there is at most one path in $\mathcal{P}^{j-1}_i\setminus \mathcal{P}^j_i$ -- call this path $P_i$ where it exists, and otherwise set $P_i=\emptyset$. Let $H'_j$ be $H_j$ with the edges of the paths in $\mathcal{P}^j_i$ removed. 
Then, as $H_j-J$ is $(4p,d,K)$-dense, by \ref{nearly3}, we have that $H_j'-J$ is $(5p,d,K)$-dense. 
Furthermore, as we removed at most $(2/q)\cdot d^{3/2}\leq \lambda d^2/2$ edges from $H_j$ to get $H_j'$, we have that $H_j'$ is $(\lambda^{1/4},\lambda/2,d)$-connected. 
Finally, each path $P_i$ which is non-empty has an endvertex with at least $d/4K-\sqrt{d}\geq d/8K$ neighbours in $V(H_j)$ in $G'$.
Thus, by Lemma~\ref{lem:densespotplusattachedpathdecomposes},  there is a set, $\mathcal{Q}$ say, of copies of $P_{(1-\eps)d}$ which decomposes all but at most $\eps |V(H_j)|d/2$ of the edges of $G'_j\cup (\cup_{i\in [d/2]}P_i)$. Let $\mathcal{Q}^j=\mathcal{Q}^{j-1}\cup \mathcal{Q}$, and note that both \ref{prop2} and \ref{prop3} are satisfied. 
Finally, then, note that \ref{prop:densedecomp1new}--\ref{prop:densedecomp3new} hold as, for each $i\in [t]$, all the endvertices of $\mathcal{P}^j_1,\ldots,\mathcal{P}^j_{d/2}$ in $V(G_i)$ are endvertices of $\mathcal{P}^{j-1}_1,\ldots,\mathcal{P}^{j-1}_{d/2}$. Thus, \ref{prop:densedecomp1new}--\ref{prop3} hold, as required.
\claimproofend
Setting $j=s$ in Claim~\ref{lastclaim}, take edge-disjoint path forests $\mathcal{P}_i^s$, $i\in [d/2]$, and $\mathcal{Q}^s$, such that \ref{prop:densedecomp1new}--\ref{prop3} hold. Now, \ref{prop1} implies that the collections of paths all have no endpoints in $H$, and therefore $\mathcal{Q}^s$ is a collection of edge disjoint copies of $P_{(1-\eps)d}$ in $G$ by \ref{prop2} which, by \ref{prop3}, decomposes all but at most $\eps \sum_{j\in [s]}|V(H_j)|d/2$ of the edges of $\bigcup_{j\in [s]}G_j$ and $\mathcal{P}_i$, $i\in [d/2]$, which are not in $\mathcal{P}^s_i$, $i\in [d/2]$. As there are at most $|J|d+|V(G_t)|\cdot pd\leq \eps |V(G_t)|d/2$ edges of $G_j$ which are not in $\bigcup_{j\in [s]}H_j$, all but at most $\eps |V(G_t)|d$ edges of $G_t$ and $\mathcal{P}_i$, $i\in [d/2]$, are contained in $\mathcal{Q}^s$ or $\mathcal{P}^s_i$, $i\in [2]$. Then, using \ref{prop:densedecomp1new}--\ref{prop:densedecomp3new}, by the induction hypothesis on $\mathcal{F}'=\{G_1,\ldots,G_{t-1}\}$ and the paths $\mathcal{P}^s_i$, $i\in [d/2]$, there is a set $\mathcal{Q}$ of copies of $\mathcal{P}_{(1-\eps)d}$ which decomposes all but at most $\eps \sum_{i=1}^{t-1}|V(G_i)|d$ edges of $\mathcal{F}'=\{G_1,\ldots,G_{t-1}\}$ and the paths $\mathcal{P}^s_i$, $i\in [d/2]$. Combined with $\mathcal{Q}^s$, this gives a set of copies of $P_{(1-\eps)d}$ which decomposes all but at most $\eps \sum_{i=1}^{t}|V(G_i)|d=\eps |V(G)|d$ of the edges of $\mathcal{F}$ and $\mathcal{P}_i$, $i\in [d/2]$. This completes the proof of the induction step, and hence the lemma.
\end{proof}


\section{Concluding remarks}
\label{sec:concluding}
In this paper, we showed that it is possible to approximately decompose $d$-regular graphs into paths with length approximately $d$. Improving this to give a full answer to Kotzig's orginal problem (Problem~\ref{prob:kotzig}), even for large $d$, appears very hard and certainly requires further new ideas and methods. This is true even for strengthening Theorem~\ref{Theorem_main} to find paths of length $d$ instead of $\lceil(1-\eps)d\rceil$, where it should be noted that our `dense spots' may have (slightly) fewer than $d$ vertices and thus may not approximately decompose into paths of length $d$. Though this paper is motivated by the paucity of results on the decomposition of sparse graphs, we note that Kotzig's problem is unsolved even in the dense regime, and showing that, when $d=\Omega(n)$, any $d$-regular $n$-vertex graph can be decomposed into copies of $P_d$ is an interesting open problem.

\par Decompositions of regular sparse graphs into other subgraphs have also been studied, where the existence of sparse regular graphs with high girth mean that we can only study comparable questions to Problem~\ref{prob:kotzig} when the graph will be decomposed into trees. Here, Graham and H\"aggkvist \cite{haggkvist1989decompositions} conjectured in 1989 that any $2d$-regular graph decomposes into any $d$-edge tree, giving a far reaching generalisation of Ringel's conjecture. This problem is wide open. Theorem~\ref{Theorem_main} implies that every $2d$-regular graph decomposes approximately into copies of $P_d$, and it would be very interesting to generalise this to obtain an approximate decomposition of any $2d$-regular graph into any $d$-edge tree, which appears to be beyond the capabilities of the methods used here.

\textbf{Acknowledgements.} Parts of this work were carried out when the first and third authors visited the Institute for Mathematical Research (FIM) of ETH Z\"urich and when the fourth author visited the London School of Economics (LSE). We would like to thank FIM and LSE for their hospitality and for creating a
stimulating research environment.

\bibliographystyle{abbrv}
\bibliography{bib}
\appendix
\section{Near regularisation: Proof of Lemmas~\ref{thm:nearreg} and~\ref{lem:allverticesregularspanningnearregular}}

In this appendix, we prove the two lemmas we need to efficiently find near-regular subgraphs, which we restate here for convenience.

\regularisingtwo*
\regularisingone*

As noted in~\cite[Section 8]{edgedisjointcycles}, these lemmas can be proved using a recent technique of Chakraborti, Janzer, Methuku and Montgomery~\cite{regularising,edgedisjointcycles}. In this, we take any graph $G$ which is approximately regular and carefully take a random subgraph $G'$ which we can show is slightly closer to being regular with positive probability (see~\cite[Section~2.4]{edgedisjointcycles} for a more detailed sketch). This will give us the following lemma, which we can then apply iteratively to ultimately find a very nearly regular subgraph without losing very much in the average degree.

\begin{lemma}\label{lemma:foriteration} Let $1/d\ll 1$, $\eps\leq 1/100$ and $\gamma\geq 10\eps$ such that $\eps d\geq 10^3\log d$.
Let $G$ be a graph in which $d\leq d(v)\leq (1+\gamma)d$ for each $v\in V(G)$. Then, for some $d'\geq (1-2\eps)d$, $G$ contains a subgraph $G'$ with $d'\leq d_{G'}(v)\leq (1+\gamma)(1-\eps/2)d'$ and $|V(G')|\geq (1-2\eps)|V(G)|$.
\end{lemma}
\begin{proof} Let $n=|V(G)|$ and assume that $G$ has no edges between any vertex with degree at least $d+1$ (for otherwise we could delete such an edge and maintain that $\delta(G)\geq d$). Let $U_{{L}}=\{v\in V(G):d(v)\leq (1+\gamma/2)d\}$ and $U_{{H}}=\{v\in V(G):d(v)> (1+\gamma/2)d\}$ be the set of low and high degree vertices in $G$ respectively, and note that there are no edges in $G[U_{{H}}]$.
Let $G'$ be a random subgraph of $G$ given by
\begin{itemize}[noitemsep,nolistsep]
    \item deleting each edge from $U_{L}$ to $U_{H}$ independently at random with probability $\eps$, and
    \item deleting each vertex in $U_{L}$ independently at random with probability $\eps$.
\end{itemize}

For each $v\in V(G)$, let $B_v$ be the event that $v\in V(G')$ but $d_{G'}(v)\notin [(1-5\eps/4)d,(1-7\eps/4)(1+\gamma)d]$. Let $t=\lceil |V(G)|/d\rceil$ and let $V(G)=A_1\cup\ldots\cup A_t$ be an arbitrary partition of $V(G)$ into sets with size between $d$ and $d/2$. For each $i\in [t]$, let $B_i$ be the event that $|V(G')\cap A_i|<(1-2\eps)|A_i|$. Let $d'=(1-5\eps/4)d$, so that, as $(1-7\eps/4)\leq (1-5\eps/4)(1-\eps/2)$, if no event $B_v$, $v\in V(G)$, occurs, then the degrees of $G'$ are in $[d',(1-\eps/2)(1+\gamma)d']$. Furthermore, if no event $B_i$, $i\in [t]$, occurs, then $|V(G')|\geq (1-2\eps)|V(G)|$. Therefore, it is sufficient to show that, with positive probability, no event $B_v$, $v\in V(G)$, or $B_i$, $i\in [t]$, holds.

For each $i\in [t]$, as $|A_i|\geq d/2$, we have, by Chernoff's bound, that
\begin{equation}\label{eqn:localbound1}
\mathbb{P}(B_i)=\mathbb{P}(|A_i\setminus V(G')|> 2\eps|A_i|)\leq 2\exp(-\eps d/24)\leq d^{-3}.
\end{equation}

If $v\in U_{H}$, then, for each $u\in V(G)$ with $uv\in E(G)$, we have $u\in U_L$, and so $uv\in E(G')$ exactly when $u$ and $uv$ are not deleted, so the probability that $uv\in E(G')$ is $(1-\eps)^2$. Thus, if $v\in U_H$, then
$\E(d_{G-G'}(v))= (2\eps-\eps^2)\cdot d_G(v)$. 
As $\gamma\geq 10\eps$ and $\eps\leq 1/100$, we have here that $(1-9\eps/4)\cdot d_G(v)\geq (1-9\eps/4)(1+\gamma/2)d\geq (1-5\eps/4)d$. Thus, 
\[
\mathbb{P}(B_v)\leq \mathbb{P}(d_{G'}(v)\notin [(1-9\eps/4)\cdot d_G(v),(1-7\eps/4)\cdot d_G(v)])=\mathbb{P}(d_{G-G'}(v)\notin [(7\eps/4)\cdot d_G(v),(9\eps/4)\cdot d_G(v)]).
\]
Then, by Chernoff's bound (in particular, Lemma~\ref{chernoff} applied with $\gamma=1/10$)
\begin{equation}\label{eqn:localbound2}
\mathbb{P}(B_v)\leq 2\exp(-(2\eps-\eps^2)d/300)\leq d^{-3}.
\end{equation}

Now, suppose $v\in U_L$. If $v$ survives into $V(G')$, then, note that, for each $uv\in E(G)$, if $v\in U_L$, then the probability that $uv\notin E(G')$ is $\eps$ (the probability that $v$ is deleted), while if $v\in U_H$ then the probabilty that $uv\in E(G')$ is also $\eps$ (the probability that $uv$ is deleted). Thus, if $v\in U_L$, then 
$\E(d_{G-G'}(v)|v\in V(G'))=\eps\cdot d_G(v)$.
As $\gamma\geq 10\eps$ and $\eps\leq 1/100$, we have that $d_G(v)\leq (1+\gamma/2)d\leq (1-7\eps/4)(1+\gamma)d$, so that, by Chernoff's bound
\begin{equation}\label{eqn:localbound3}
\mathbb{P}(B_v)\leq \mathbb{P}(d_{G-G'}(v)\geq 5\eps d_G(v)/4|v\in V(G'))\leq \exp(-\eps d_G(v)/48)\leq \exp(-\eps d/48)\leq d^{-3}.
\end{equation}

By \eqref{eqn:localbound1}, \eqref{eqn:localbound2}, and \eqref{eqn:localbound3}, all the `bad events' we have defined occur with probability at most $d^{-3}$. Each `bad event' is affected by the possible deletion of at most $(1+\gamma)d$ vertices and $(1+\gamma)d$ edges, and the possible deletion of each vertex/edge affects at most $(1+\gamma)d+1$ `bad events'. As $1/d\ll 1$, we have $e\cdot d^{-3}\cdot (2(1+\gamma)d((1+\gamma)d+1)+1)\leq 1$, so by the local lemma there is some choice of $G'$ for which none of the `bad events' hold, as required.
\end{proof}

Note that if $\gamma\leq 1/10$, then we can iteratively apply Lemma~\ref{lemma:foriteration} with $\eps=\gamma/10$ 50 times to immediately get the following corollary.

\begin{corollary}\label{cor:morereg}
Let $1/d\ll 1$ and $\gamma\leq 1/10$ such that $\gamma d\geq 10^4\log d$.
Let $G$ be a graph in which $d\leq d(v)\leq (1+\gamma)d$ for each $v\in V(G)$. Then, for some $d'\geq (1-10\gamma)d$, $G$ contains a subgraph $G'$ with $d'\leq d_{G'}(v)\leq (1+\gamma/2)d'$ and $|V(G')|\geq (1-10\gamma)|V(G')|$.
\end{corollary}

We can now apply Corollary~\ref{cor:morereg} iteratively to prove Lemma~\ref{lem:allverticesregularspanningnearregular}.
\begin{proof}[Proof of Lemma~\ref{lem:allverticesregularspanningnearregular}] Let $G_0=G$, and let $r$ be the smallest integer with $\gamma d/2^r\leq 10^4\log d$.
For each $0\leq i<r$ in turn, apply Corollary~\ref{cor:morereg} with $\gamma'=\gamma/2^i$ to $G_i$ to get $G_{i+1}$ with $|V(G_{i+1})|\geq (1-10\gamma/2^i)|V(G_{i+1})|$ and vertex degrees in $[d_{i+1},(1+\gamma/2^{i+1})d_{i+1}]$ for some $d_{i+1}\geq (1-10\gamma/2^i)d_i$. Then, we have 
\[
|V(G_r)|\geq \prod_{i=0}^{r-1}(1-10\gamma/2^i)|V(G)|\geq \Big(1-10\gamma \cdot \sum_{i=0}^{\infty}2^{-i}\Big)|V(G)|=(1-20\gamma)|V(G)|
\]
and the degrees in $G_r$ are between $[d_{r},(1+\gamma/2^{r+1})d_{r}]\subset [d_r,d_r+10^4\log d]$, where 
\[
d_r\geq \prod_{i=0}^{r-1}(1-10\gamma/2^i)d\geq (1-20\gamma)d,
\]
so that $G'=G_r$ satisfies the required conditions with $C=2\cdot 10^4$ as $\log d_r\geq (\log d)/2$.
\end{proof}

Finally, we can prove Lemma~\ref{thm:nearreg}.

\begin{proof}[Proof of Lemma~\ref{thm:nearreg}] Note that we can assume that $C\geq 10$ and $1/d\ll 1$.
Let $\eps=1/1000$ and $k=10^4\log C$. Let $G$ be a graph with vertex degrees between $d$ and $Cd$. Apply Lemma~\ref{lemma:foriteration} iteratively $k$ times to $G$ to get a subgraph $G''$ in which the vertex degrees differ by a factor of at most $\max\{1+10\eps,(1-\eps/4)^{k}\cdot C\}\leq 1+\frac{1}{100}$ with $d(G')\geq (1-2\eps)^{k}d(G)\geq 2d/C'$. Then, applying Lemma~\ref{lem:allverticesregularspanningnearregular} gives the required subgraph.
\end{proof}

\end{document}